\documentclass[11pt,english,leqno]{amsart}

\usepackage{layout}
\usepackage{amssymb,lscape}
\usepackage{amsfonts}
\usepackage{amsmath,amscd,amsthm}
\usepackage{newlfont}
\usepackage[dvips]{graphicx}
\usepackage{epsf}

\usepackage[dvipsnames,svgnames,x11names]{xcolor}
\usepackage[latin1]{inputenc}
\usepackage{lscape}
\usepackage{verbatim}

\usepackage[bookmarks=false,colorlinks=true,linkcolor=magenta,urlcolor=magenta,citecolor=magenta,filecolor=maginta,linktocpage=true,pdfstartview={XYZ null null 0.702}]{hyperref}

\RequirePackage[hyperpageref]{backref}
\renewcommand*{\backref}[1]{}
\renewcommand*{\backrefalt}[4]{
     \ifcase #1 (no cit.)
      \or (cit. on p. #2)
      \else (cit. on pp. #2)
      \fi}

\usepackage{enumitem}
\usepackage{nccmath}
\usepackage[normalem]{ulem}

\usepackage{subfigure}
\usepackage{caption}
\usepackage{array,makecell,adjustbox}
\usepackage{diagbox}
\usepackage{colortbl}
\usepackage{hhline}
\usepackage{rotating,multirow}

\usepackage{titletoc}
\DeclareRobustCommand{\SkipTocEntry}[5]{}

\newcommand{\E}{\mathbb E}

\newcommand{\M}{\mathbb M}
\renewcommand{\P}{\mathbb P}\newcommand{\R}{\mathbb R}\newcommand{\T}{\mathbb T}

\newcommand{\Z}{\mathbb Z}

\newcommand{\cA}{\cal{A}}\newcommand{\cB}{\cal{B}}\newcommand{\cC}{\cal{C}}
\newcommand{\cF}{\cal{F}}\newcommand{\cG}{\cal{G}}\newcommand{\cH}{\cal{H}}
\newcommand{\cK}{\cal{K}}\newcommand{\cL}{\cal{L}}\newcommand{\cN}{\cal{N}}
\newcommand{\cP}{\cal{P}}\newcommand{\cR}{\cal{R}}\newcommand{\cT}{\cal{T}}
\newcommand{\cU}{\cal{U}}\newcommand{\cV}{\cal{V}}

\newcommand{\indicat}{\text{1}\negthickspace\text{I}}
\newcommand{\la}{\lambda}

\newcommand{\eps}{\varepsilon}

\newcommand{\dd}{\text{d}}

\newcommand{\pass}{\vspace{0.3cm}\noindent}
\newcommand{\ppass}{\vspace{0.1cm}\noindent}
\newcommand{\noi}{\noindent}

\newcommand{\tmax}{t_{\mbox{\tiny{max}}}}

\DeclareMathOperator{\dimbox}{\dim_B}\DeclareMathOperator{\dimh}{\dim_H} 
 
\DeclareMathOperator{\card}{card} \DeclareMathOperator{\diam}{diam}
\DeclareMathOperator{\cl}{cl}
\DeclareMathOperator{\intt}{int}

\theoremstyle{plain}
\newtheorem{theo}{Theorem}[section]
\newtheorem{prop}[theo]{Proposition}
\newtheorem{lem}[theo]{Lemma}

\theoremstyle{definition}

\numberwithin{equation}{section}

\definecolor{note}{rgb}{0.20,0.10,1}

\definecolor{new}{rgb}{0.75,0.40,0.30}

\author[P. Calka and Y. Demichel]{Pierre Calka\textsuperscript{1} and Yann Demichel\textsuperscript{2}}

\address{\textsuperscript{1} Univ Rouen Normandie, CNRS, Normandie Univ, LMRS UMR 6085, F-76000 Rouen, France.}
\email{pierre.calka@univ-rouen.fr}%
\address{\textsuperscript{2}Laboratoire MODAL'X, UMR CNRS 9023, UPL, Universit\'e Paris Nanterre, 200 avenue de la R\'e\-pu\-bli\-que, 92001 Nanterre, France.}
\email{ydemichel@parisnanterre.fr \\ ORCID Id: 0009-0008-4306-9469}%

\address{This work was partially supported by the French \textit{r\'eseaux th\'ematiques} MAIAGES (RT CNRS 2179) and ANAIS (RT CNRS 2169) and the \textit{Institut Universitaire de France}. We also thank our attentive and interested colleagues who attended our talks on the subject and contributed through their useful remarks in improving the presentation of the paper.}

\title{Fractal random sets associated with multitype Galton-Watson trees}

\begin{document}

\begin{abstract}
In this paper, we consider a regular tessellation of the Euclidean plane and the sequence of its geometric scalings by negative powers of a fixed integer. We generate iteratively random sets as the union of adjacent tiles from these rescaled tessellations. We encode this geometric construction into a combinatorial object, namely a multitype Galton-Watson tree. Our main result concerns the geometric properties of the limiting planar set. In particular, we show that both box and Hausdorff dimensions coincide and we calculate them in function of the spectral radius of the reproduction matrix associated with this branching process. We then make that spectral radius explicit in several concrete examples when the regular tessellation is either hexagonal, square or triangular. 
\end{abstract}


\vspace*{-1.5cm}
\maketitle

\tableofcontents
\thispagestyle{empty}

\addtocontents{toc}{\vspace{0.2cm}}%
\section{A growth model of planar random sets: An introduction}\label{sec:intro}

Let us start with a historical wandering through three different sets which are designed with an iterative construction and whose boundary is expected to be very irregular. 

The most famous and popular example of fractal set is certainly the {\it von Koch curve} named after the Swedish mathematician Helge von Koch. In his paper \cite{koch04} published in 1904, he constructs with a simple geometric iterative process a closed curve of infinite length, continuous but nowhere differentiable. 
The same procedure may be used to construct the so-called {\it von Koch snowflake} which is a basic but representative example of the limit of an increasing sequence of planar compact sets, which has finite area but infinite perimeter and, actually, a fractal boundary, see e.g. \cite{Dem24} for new historical considerations. Basically, the snowflake is obtained as follows: starting from an equilateral triangle, the von Koch iterative procedure consists in dividing each side of the set in construction into three equal parts, and adding at the middle third of each side a smaller equilateral triangle. Notice that the von Koch procedure has been modified in a natural way to generate random fractal curves, see e.g. \cite[pp. 244--245]{falco03}, and remains one of the prominent examples in the fractal world, see e.g. \cite{BBE15,DSL18,KW24}.

The {\it Eden model} was introduced in 1956 by the American physical chemist Murray Eden in the biological context of the growth of colonies of bacteria, see \cite{eden56,eden61}. It consists in considering the cubical tessellation of $\R^m$, $m\ge 2$, induced by the mesh $\Z^m$, starting with one cube and at each time step adding a neighboring cube uniformly at random. The limiting random interface is conjectured to belong, after rescaling, to the KPZ universality class although it is notoriously very difficult to tackle, see e.g. \cite{MRS23}.

In another direction, Konstantin Tchoumatchenko and Sergei Zuyev have introduced in 2001 a little-known model of random set which is based on a sequence of Poisson-Voronoi tessellations with increasing intensity, see \cite{tchou01}. We recall that the Poisson-Voronoi tessellation is the Voronoi partition generated by the set of nuclei provided by a Poisson point process. Their construction starts with an initial Poisson-Voronoi tessellation, then at each step, the current tessellation is overlayed with a new Poisson-Voronoi tessellation and each current cell is replaced by the union of new Voronoi cells whose nuclei belong to the cell in question. The considered random set is the union of the boundaries of all cells. Tchoumatchenko and Zuyev derive several notable properties for this set when the number of iterations goes to infinity, including an upper bound for its Hausdorff dimension and an estimate for its associated spherical contact distribution but their investigation falls short of showing its fractal nature. 

Inspired by all three examples above, we introduce a growth process based on one of three regular two-dimensional tessellations, i.e. triangular, square and hexagonal. The construction consists in starting with one tile from the tessellation, then adding randomly along each of its edges rescaled versions of the same tile and iterating the procedure, see Figure \ref{fig:NosSimuIntro} for simulations in the three cases. We take our inspiration from the three previous historical models. Indeed, we keep from the von Koch approach the ideas of an increasing sequence of compact sets, of the underlying self-similarity through the addition of rescaled versions of the same shape and of the example of the triangular tessellation as the starting object. We keep from the Eden model the randomness of the addition and the example of the square tessellation. Finally, we keep from the Tchoumatchenko-Zuyev model the idea of generating consecutively overlaid tessellations with an exponentially increasing underlying intensity and the example of the hexagonal tessellation. Indeed, the hexagonal tessellation may be seen as the most natural deterministic idealization of the Poisson-Voronoi tessellation because the expected number of edges of a typical tile of a homogeneous Poisson-Voronoi tessellation is equal to $6$, see \cite{mol94}. Actually, all three regular tessellations are Voronoi tessellations induced by a set of nuclei along a regular dual grid.

\renewcommand{\arraystretch}{0.8}
\begin{figure}[h!]
\begin{center}
\begin{tabular}{ccc}
\includegraphics[scale=1]{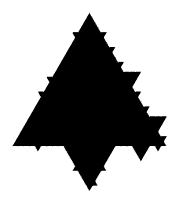} &
\includegraphics[scale=0.35]{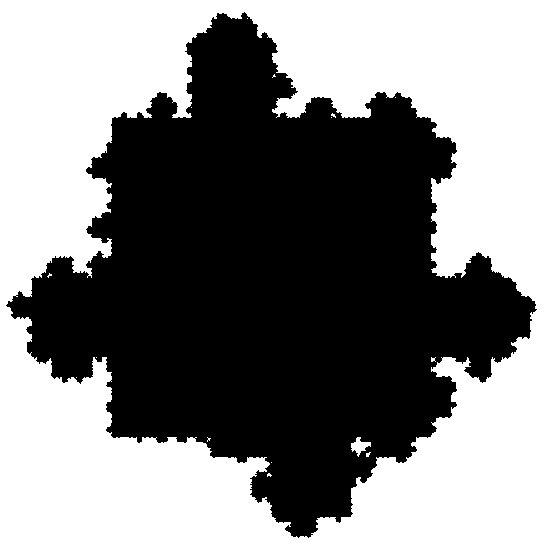} &
\includegraphics[scale=0.62]{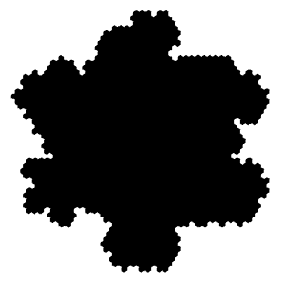} \\
\footnotesize{(a)} & \footnotesize{(b)} & \footnotesize{(c)}
\end{tabular}
\end{center}
\vspace{-.35cm}
\caption{Simulation of the random growth model for the triangular (resp. square, hexagonal) tessellation with the particular choice $p_*=0$, $p=0.5$ and $\la=3$ (resp. $\la=4$, $\la=3$).}
\label{fig:NosSimuIntro} 
\end{figure}
\renewcommand{\arraystretch}{1.5}

We make the construction more formal in the following lines. We assume that each tile of the regular tessellation $\cT$ has diameter one and that the origin is either at the center of one tile or one of the vertices of a tile. Let $\la\geq3$ be an integer and let $(\la^{-n}\cT)_{n\ge 0}$ be the associated sequence of tessellations generated by consecutive rescalings of $\cT$. We construct an increasing and bounded sequence $(\cK_n)_{n\ge 0}$ of compacts sets with an iterative geometric procedure which guarantees that at each step, $\cK_n$ is a union of tiles of $\la^{-n}\cT$, see Figure \ref{fig:setKn}. Namely, we start by fixing the tile of the grid $\cT$ which contains $0$ and call it $\cK_0$. Once $\cK_n$ is constructed for some $n\ge 0$, we define
\begin{equation}\label{eq:iterKn}
\cK_{n+1} = \cK_n^\bullet\cup \cK_n^\circ\supset \cK_n
\end{equation}
where:
\begin{itemize}[parsep=-0.15cm,itemsep=0.2cm,topsep=0.2cm,wide=0.15cm,leftmargin=0.65cm]
\item[$-$] $\cK_n^\bullet$ is the deterministic set constituted with all of the tiles $\cT_{n+1,\ell}\in \la^{-(n+1)}\cT$, $\ell\ge 1$, which intersect the interior of $\cK_n$,
\item[$-$] $\cK_n^\circ$ is a random union of tiles of $\la^{-(n+1)}\cT$. 
\end{itemize}

\begin{figure}[h!]
\centering
\includegraphics[scale=0.2]{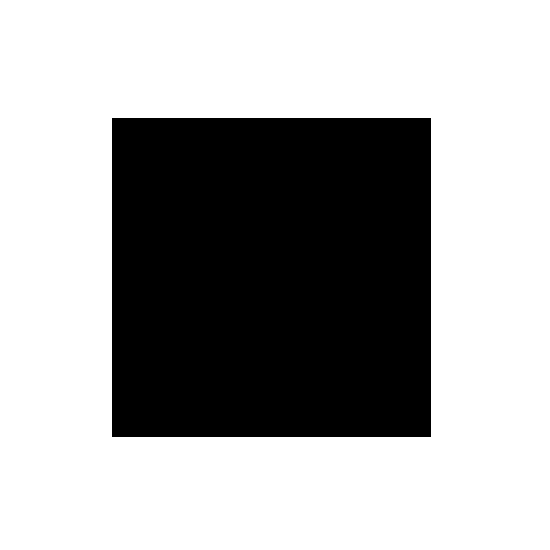}
\includegraphics[scale=0.2]{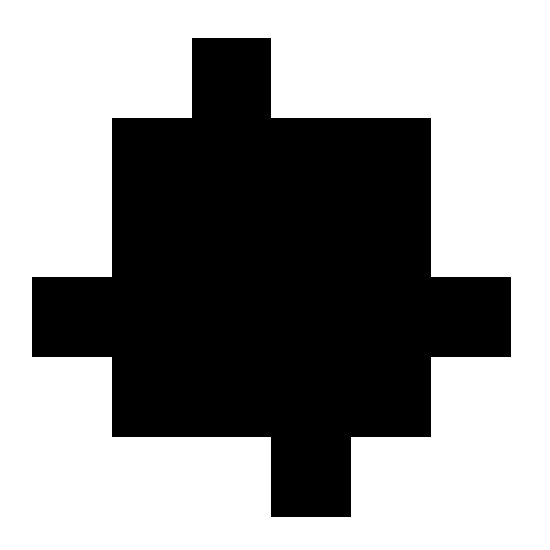}
\includegraphics[scale=0.2]{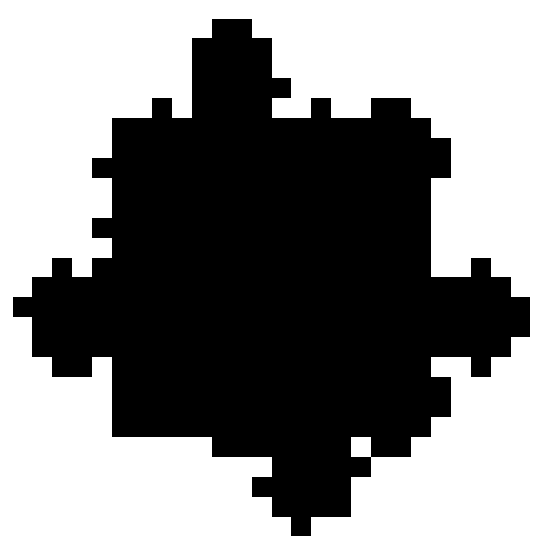}
\includegraphics[scale=0.2]{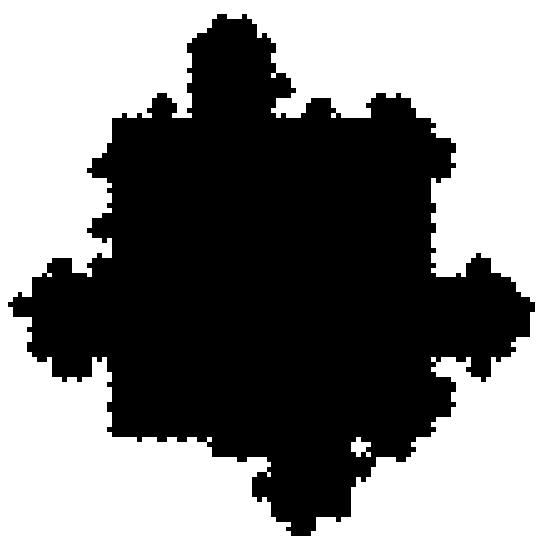}
\includegraphics[scale=0.2]{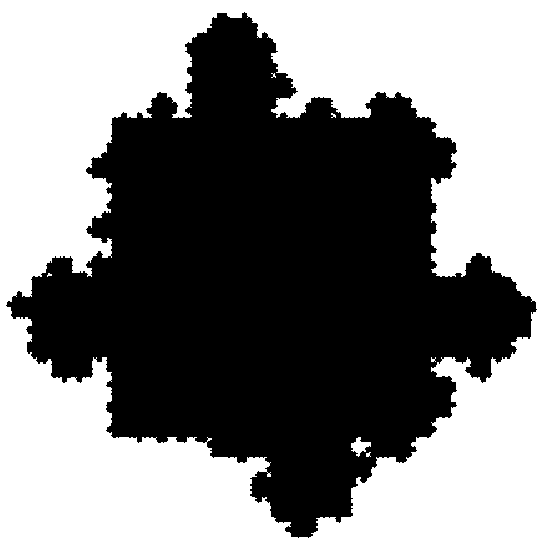}
\caption{Iterative construction of $\cK_{n+1}$ as the union of the deterministic set $\cK_n^\bullet$ and the random set $\cK_n^\circ$ for the square tessellation with parameters $\la=4$, $p_*=0$ and $p=0.5$.}\label{fig:setKn}
\end{figure}

\noi More precisely, each additional tile $\cT_{n+1,\ell}$ from $\la^{-(n+1)}\cT$ is chosen independently from the others and according to a Bernoulli variable whose parameter $p_\ell$ does not depend on $n$ but only on the geometry of $\partial \cT_{n+1,\ell}\cap\partial \cK_n$ according to the two following groups, see Figure \ref{fig:SquareSnake}:
\begin{itemize}[parsep=-0.15cm,itemsep=0.25cm,topsep=0.2cm,wide=0.15cm,leftmargin=0.65cm]
\item[$-$] Group $1$: the tiles $\cT_{n+1,\ell}$ which either share exactly one edge with $\partial \cK_n$ but no end of an edge of $\partial \cK_n$ or share exactly two edges with $\partial \cK_n$;
\item[$-$] Group $2$: the tiles $\cT_{n+1,\ell}$ which share exactly one edge with $\partial \cK_n$ and contain the end of an edge of $\partial \cK_n$.
\end{itemize}

\noi We fix $p\in [0,1]$ and $p_*\in\{0,1\}$. When $\cT_{n+1,\ell}$ is in Group $1$ (resp. Group $2$), we take $p_\ell=p$ (resp. $p_\ell=p_*$), see Figure \ref{fig:SquareSnake}. Otherwise, we take $p_\ell=0$. The particular choice $p_*=0$ or $1$ for the tiles at the end of an edge of $\cK_n$ guarantees the spatial independence of the evolution of the growth inside two neighboring tiles. This will allow a solvable probabilistic coding of the model. 

\renewcommand{\arraystretch}{0.8}
\begin{figure}[h!]
\begin{center}
\begin{tabular}{ccc}
\includegraphics[scale=0.75]{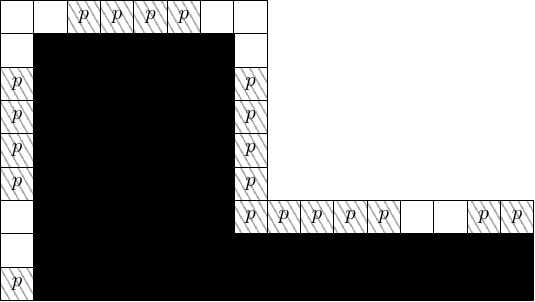} & &
\includegraphics[scale=0.75]{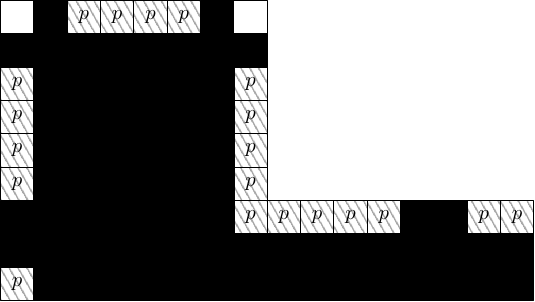} \\
\footnotesize{(a)} & & \footnotesize{(b)}
\end{tabular}
\end{center}
\vspace{-.35cm}
\caption{Description of the random rule in the case of the square tessellation (a) $p_*=0$ and (b) $p_*=1$. Each square $\cT_{n+1,\ell}$ from $\la^{-(n+1)}\cT$ along $\partial\cK_n$ is added to $\cK_n$ to obtain $\cK_{n+1}$ according to its Bernoulli variable $\cB(p_\ell)$.}
\label{fig:SquareSnake}
\end{figure}
\renewcommand{\arraystretch}{1.5}

By \eqref{eq:iterKn} the sequence $(\cK_n)_{n\ge 0}$ is increasing. Moreover, the distance from $\{0\}$ to $\partial \cK_n$ is upper bounded by $\sum_{k\geq0}\la^{-k}$. Consequently, it converges almost surely for the Hausdorff distance to the closure of the union of the sets $\cK_n$, namely to the compact set
\begin{equation}\label{eq:limcup}
\cK_{\infty}=\cl \bigg(\bigcup_{n\ge 0} \cK_n \bigg).
\end{equation}

We aim at investigating the geometric and fractal properties of the boundary $\partial \cK_\infty$ of the limit set $\cK_\infty$. In particular, we are interested in the box and Hausdorff dimensions of $\partial\cK_\infty$. Our main result below is an explicit calculation of both the Hausdorff dimension $\dimh(\partial \cK_\infty)$ and box-dimension $\dimbox(\partial \cK_\infty)$ of the limiting set $\partial \cK_\infty$ (see \eqref{eq:defhsdim} and \eqref{eq:defdimbox} for the definition of these two dimensions). The calculation depends on an explicit deterministic matrix $\M$ which is defined at \eqref{eq:defMatrixM}, is non-negative and will be proved to be primitive. As we will see in Section \ref{sec:galton}, this matrix is the reproduction matrix associated with a multitype Galton-Watson process encoding our model. The spectral radius $\rho_\M$ of $\M$ is the key parameter of the whole machinery.

\begin{theo}\label{theo:main}
The set $\partial \cK_\infty$ has Hausdorff dimension and box-dimension which coincide almost surely and are equal to
\begin{equation}\label{eq:dimfrac}
d=\frac{\log\rho_\M}{\log \la}\in [1,2).
\end{equation}
Moreover, the $d$-dimensional Hausdorff measure $\cH^d(\partial \cK_\infty)$ of $\partial \cK_\infty$ is positive in mean and finite almost surely.
\end{theo}

In addition to Theorem \ref{theo:main}, our second result provides the asymptotics for the perimeter $\cL(\partial \cK_n)$ of $\partial\cK_n $ and the area $\cA(\cK_{n+1}\setminus \cK_n)$ of the difference set $\cK_{n+1}\setminus \cK_n$. 

\begin{theo}\label{theo:perimaire}
There exist two positive random variables $L_\infty$ and $A_\infty$ such that $\P(L_{\infty}>0,A_{\infty}>0)=1$ and such that with probability $1$,
\begin{equation}\label{eq:limperimaire}
\lim_{n\to \infty}\bigg(\frac{\la}{\rho_\M}\bigg)^n\cL(\partial \cK_n)=L_\infty
\text{ and }
\lim_{n\to\infty}\bigg(\frac{\la^2}{\rho_\M}\bigg)^n\cA(\cK_{n+1}\setminus \cK_n)=A_\infty.
\end{equation}
As a consequence,
\begin{equation}\label{eq:limairedefaut}
\lim_{n\to\infty}\bigg(\frac{\la^2}{\rho_\M}\bigg)^n\cA(\cK_\infty\setminus \cK_n)=\frac{\la^2}{\la^2-\rho_\M}A_\infty.
\end{equation}
\end{theo}

The key idea for proving Theorems \ref{theo:main} and \ref{theo:perimaire} consists in coding the construction of the set $\cK_\infty$ with a canonical multitype Galton-Watson process. This consists in particular in enumerating the tiles which surround $\cK_n$ and assigning to each of them a fixed type according to its intersection with $\cK_n$. We show that the Hausdorff dimension of the geometric set coincides with the dimension of the boundary of the tree as introduced by Furstenberg \cite{furst71} when the tree is endowed with the classical $\la$-adic ultrametric distance. Actually, we show independently that the dimension of the tree coincides with \eqref{eq:dimfrac}, see Theorem \ref{theo:dimtree}. The calculation of the dimension of a simple branching process dates back to \cite{holmes73} and \cite{hawk81}. Since then, it has been extended in \cite{wat07} to the determination of the exact gauge for the Hausdorff measure and more recently proved in an elementary way in \cite{Duq09}. The multitype case was investigated in \cite{Lal00}. More precisely, they obtain the dimension of a subset of the boundary of the tree when the limiting frequency is fixed. The actual dimension of the whole boundary is then expressed implicitly as the maximum of a certain function over a set of invariant ergodic measure, see Remark (C) therein. To the best of our knowledge, although we think that the classical methods used in the monotype case naturally extend, there is no known explicit value for the dimension of the boundary of the tree and our Theorem \ref{theo:dimtree} fills the gap. 

In the study of deterministic fractal models, in particular those enjoying a self-similar structure as the attractor of an Iterated Functions System (IFS), it is not unusual to construct an associated offspring matrix and express the fractal dimension of the set in terms of the spectral radius of such a matrix, see e.g. \cite{Duv99} for the dimension of the L\'evy dragon, then \cite{Duv00} and \cite{lau03} for several extensions. In the case of random fractal sets, several works rely on a representation by a tree and most notably a Galton-Watson tree, in particular when considering sets defined as the intersection of unions of rescaled tiles, like for instance the random Cantor set or the Mandelbrot percolation, see the appropriate survey \cite{mor09}. Nevertheless, the model considered in this paper does not exhibit the same self-similarity feature as a random IFS. As a matter of fact, in comparison with the literature, we consider that our approach presents several specificities. Indeed, the coding is done on the covering of the random set rather than on the set itself. As a consequence, this induces the appearance of several types and the construction of a multitype Galton-Watson process, which, to the best of our knowledge, has never been used before in the context of random fractals. Considering \eqref{eq:limcap}, one could assume that we could adopt a dual point of view by forgetting about the construction of $\cK_\infty$ and concentrating on the intersection of the covering sets $\cR_n$, $n\ge 0$, as defined at \eqref{def:cover}. Only, this approach would not fit into the general theory of fractals governed by Galton-Watson trees as presented in \cite{mor09} since the construction of $\cR_n$ is inextricably tied up with the one of $\cK_n$ and could not be expressed as a classical percolation. The technique that we develop here is based on the explicit calculation of the reduced reproduction matrix in Sections \ref{sec:exaintro} and \ref{sec:Model} and makes the Hausdorff dimension of the set fully computable in practice, which illustrates the efficiency of the method. 

\newpage
The paper is structured as follows. We describe the coding in Section \ref{sec:main} and investigate the Hausdorff dimension of the boundary of the associated multitype Galton-Watson tree through martingale techniques and geometric measure theory. Section \ref{sec:proof} is then devoted to the proofs of Theorems \ref{theo:main} and \ref{theo:perimaire}. In Section \ref{sec:exaintro}, we put in motion the machinery to solve two toy examples, namely a random version of the von Koch curve and a generalized von Koch model based on a deterministic pattern. Incidentally, we also explain how to reduce the size of the reproduction matrix $\M$, which facilitates an exact computation of its spectral radius $\rho_\M$. In Section \ref{sec:Model}, we then identify explicitly that reproduction matrix and whenever possible, its spectral radius associated with the model described at \eqref{eq:limcup} for the three tessellations, hexagonal, square and triangular, respectively. Our choice of starting with the hexagonal tessellation is due to two reasons: first, this allows us to deal with matrices of size $2\times2$, as opposed to the square case, and second, this induces an interesting artefact, i.e. the calculation depends on the remainder of the Euclidean division of $\la$ by $3$. For all three tessellations, we determine concretely each entry of the reproduction matrix by partitioning the set of tiles of $\cT_n$ which are at the border of $\cK_n$. This induces a heavy ensemble of tables which deal with all the different cases and which is postponed to the Appendix in Section \ref{sec:annexe}. Finally, we discuss some related open problems in Section \ref{sec:outro}.

\addtocontents{toc}{\vspace{0.2cm}}%
\section{The multitype Galton-Watson tree}\label{sec:main}
In this section, we focus on the geometric and probabilistic tools needed for proving Theorems \ref{theo:main} and \ref{theo:perimaire}. We first introduce a good economic covering of the boundary $\partial \cK_\infty$ of the limit set, which allows to express its box-dimension in function of the cardinality of a random set of tiles. We then present the core idea of the paper, namely the construction of a coding of the tiles into a multitype Galton-Watson tree. This branching process has several well-known properties, notably in terms of martingales, that we describe for our purpose. We conclude the section with our main result on multitype branching trees, i.e. the explicit calculation of the Hausdorff dimension of the boundary of the tree as a function of the spectral radius of its reproduction matrix.

\subsection{A natural covering of $\partial \cK_\infty$}\label{sec:cover}

\pass 

We recall that the box-dimension $\dimbox({\partial\cK_\infty})$ of $\partial \cK_\infty$ is obtained by counting the number of squares of the regular $\eps$-mesh needed to cover the whole set. Precisely, 
\begin{equation}\label{eq:defdimbox} 
\dimbox({\partial\cK_\infty})=\limsup_{\eps\to 0}\frac{\log N_\eps}{-\log \eps}
\end{equation}
where $N_\eps$ is the number of tiles from $\eps \cT$, where $\cT$ is the square tessellation, which intersect $\partial \cK_\infty$
(see e.g. \cite[Chap. 3, p. 43]{falco03}).

We aim at constructing an explicit covering of the limit set $\partial \cK_\infty$ which is tailored to improve the calculation of the box-dimension and later on, of the Hausdorff dimension of $\partial \cK_\infty$. The trick consists in using a union of tiles of $\la^{-n}\cT$ to build an economic covering. We consider the following set $\cR_n$ pictured in Figure \ref{fig:setRn}:
\begin{equation}\label{def:cover}
\cR_n=\left\{\begin{array}{ll}\big\{\cT_{n,\ell}\in \la^{-n}\cT\setminus \cK_n: \cT_{n,\ell}\cap \partial \cK_{n}\mbox{ contains at least one edge}\big\} & \mbox{ if $p_*=0$} \\
\big\{\cT_{n,\ell}\in \la^{-n}\cT\setminus \cK_n: \cT_{n,\ell}\cap \partial\cK_n\ne \emptyset\big\} & \mbox{ if $p_*=1$,}\end{array}\right.
\end{equation} 
where we recall that $p_*$ is the probability to add a tile of Group $2$, i.e. at the end of an edge.

\renewcommand{\arraystretch}{0.8}
\begin{figure}[h!]
\begin{center}
\begin{tabular}{cc}
\includegraphics[scale=0.2]{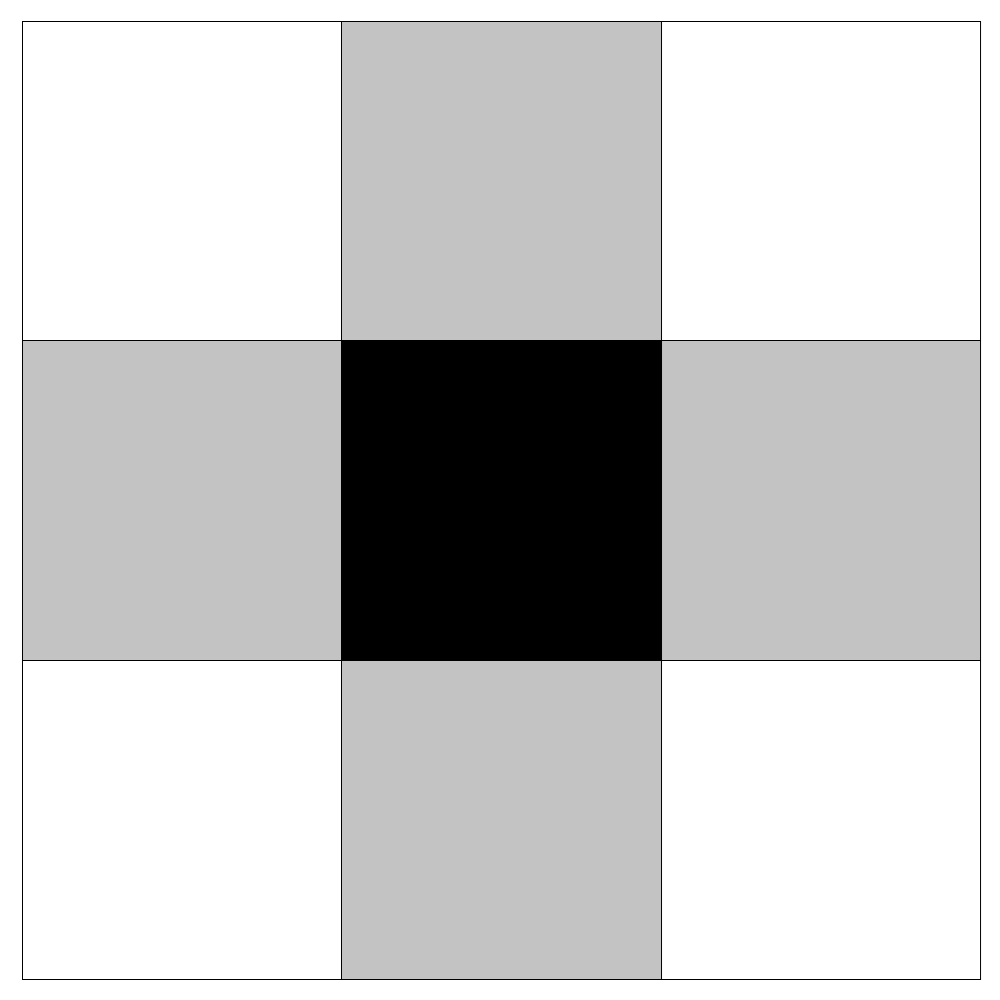} &
\includegraphics[scale=0.2]{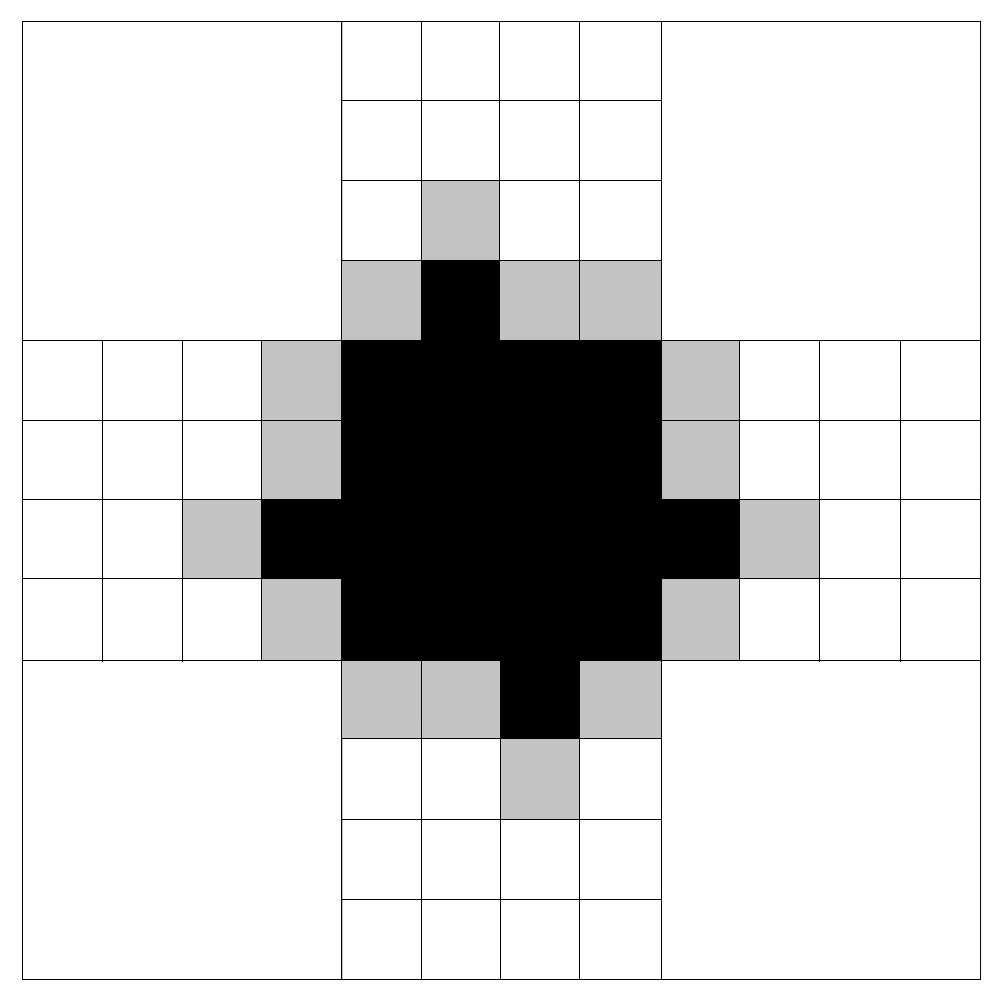} \\
\footnotesize{(a)} & \footnotesize{(b)}
\end{tabular}
\end{center}
\vspace{-.35cm}
\caption{Example for the square tessellation of the construction (in grey) of $\cR_0$ (a) and $\cR_1$ (b) for the sets $\cK_0$ and $\cK_1$ (in black) associated with the parameters $\la=4$, $p_*=0$ and $p=0.5$.}\label{fig:setRn}
\end{figure}
\renewcommand{\arraystretch}{1.5}

\newpage 
We classically endow the space of non-empty compact sets of $\R^2$ with the Hausdorff metric denoted by $\mathrm{d_H}$.

\begin{lem}\label{lem:cover}
The sequence $(\cR_n)_{n\ge 0}$ is decreasing and converges to $\partial \cK_\infty$, i.e.
\begin{equation}\label{eq:limcap}
\partial\cK_\infty = \bigcap_{n\ge 0}\cR_n.  
\end{equation}
Moreover, the box-dimension of $\partial \cK_\infty$ is almost surely given by
\begin{equation}\label{eq:dimbox}
\dimbox(\partial \cK_\infty) = \limsup_{n\to \infty} \frac{\log R_n }{\log \la^n}
\end{equation}
where $R_n=\card(\cR_n)$ is the (random) number of tiles belonging to $\cR_n$. 
\end{lem}

\begin{proof} Since $\la\ge 3$, the set $\partial \cK_{n+1}$ is at most at distance $\frac12\la^{-n}$ from $\partial \cK_n$. This means that the tiles of $\cR_{n+1}$ are located in $\cR_n$. Moreover, the sequence $(\cR_n)_{n\ge 0}$ has been designed so that it satisfies the two following properties:
\begin{itemize}[parsep=-0.15cm,itemsep=0.25cm,topsep=0.2cm,wide=0.175cm,leftmargin=0.5cm]
\item[(P1)] Each tile of $\cR_{n+1}$ is included in exactly one tile of $\cR_n$.
\item[(P2)] Each tile of $\cR_n$ contains at least one tile of $\cR_{n+1}$.
\end{itemize}

Property (P1) is clear in the case of a nested tessellation, i.e. square or triangular. In the case of the hexagonal tessellation, we fix a tile $\cT_{n,\ell}$ of $\cR_n$ and observe that any tile of $\la^{-(n+1)}\cT$ whose interior intersects $\cT_{n,\ell}$ but is not included in $\cT_{n,\ell}$ is either in $\cK_{n+1}$ or does not meet $\partial \cK_{n+1}$. This implies (P1). The definition of $\cR_n$ in the two cases $p_*=0$ and $p_*=1$ implies Property (P2).

The two properties (P1) and (P2) imply that the sequence $(\cR_n)_{n\ge 0}$ is decreasing and that each tile of $\cR_n$ contains at least one edge or vertex of $\partial \cK_m$ for any $m>n$. Consequently, there exists a positive constant $c>0$ such that $\mathrm{d_H}(\partial \cK_\infty,\partial \cR_n)\le c \la^{-n}$, which implies in turn \eqref{eq:limcap}
and that each of the tiles contained in $\cR_n$ intersects $\partial \cK_{\infty}$. 

\enlargethispage*{0.2cm}
We now prove \eqref{eq:dimbox}. The definition \eqref{eq:defdimbox} can be modified for practical purpose in two different directions: first, we can replace the square shape by another fixed shape and secondly, we can discretize the size $\eps$ of the mesh as a sequence which decays geometrically, see e.g. \cite[Chap. 3, pp. 44-45, Chap. 4, Prop. 4.1]{falco03}. The trick here is to use the subset $\cR_n$ of tiles of $\la^{-n}\cT$ since the diameter of each of these tiles decays as $\la^{-n}$. This completes the proof of Lemma \ref{lem:cover}. 
\end{proof}

The equality of sets given at \eqref{eq:limcap} suggests that we could study the fractal dimension of the intersection of all sets $\cR_n$, $n\ge 0$, instead of that of $\partial \cK_\infty$. Indeed, the sequence $(\cR_n)_{n\ge 0}$ is non-increasing and the recursive construction is nested, i.e. the set $\cR_{n+1}$ consists of a random choice of subtiles of each tile of $\cR_n$. For this kind of model, the classical method consists in expressing the dimension as the solution of a self-similarity equation, as in the deterministic setting, see \cite{mor09}. Nevertheless, this machinery does not apply here. Indeed, the choice of subtiles is not a traditional percolation process, which means that the fate of each subtile is not independent of the others. In fact, it is impossible to flout the underlying construction of $\partial \cK_\infty$. That is why we henceforth adopt the coding by a multitype Galton-Watson tree. 

\subsection{Coding with a multitype branching process}\label{sec:galton}

\noi

Thanks to \eqref{eq:dimbox}, the calculation of $\dimbox(\partial \cK_\infty)$ only requires to estimate the asymptotics of $R_n$ which is a combinatorial problem. Because of the symmetry of one tile and the definition of the model, we observe that when $\cT$ is the square (resp. hexagonal, triangular) tessellation, the set $\cR_n$ can be naturally partitioned into $c_\cT=4$ (resp. $6$, $3$) i.i.d. pieces corresponding to the construction emanating from each edge of the initial tile of the considered tessellation\footnote{Actually, the situation is slighty more intricate when $p_*=1$ for the square tessellation, see the discussion preceding Lemma \ref{lem:galton}.}. This implies that it is enough to study the model starting above one of the edges of the initial tile. Our strategy described in detail below then consists in constructing a multitype Galton-Watson process such that $R_n$ is exactly, up to the multiplicative constant $c_\cT$, the number of children of the $n$th generation. 

For each $n$, a tile $\cT_{n,\ell}$ included in $\cR_n$ is seen as a {\it parent} which gives birth to a collection of {\it children} which are tiles included in $\cT_{n,\ell}$ and belonging to $\cR_{n+1}$. The subtlety here is that the number of such children depends on the local geometry of $\cT_{n,\ell}\cap \cK_n$. This is why we attach to each tile a precise {\it type}, which corresponds to the geometric nature of $\cT_{n,\ell}\cap \cK_n$, namely the edges and vertices of $\cT_{n,\ell}$ which belong to $\cK_n$. In particular, the set of types is always finite. In the rest of the paper, we label all the types with integers.

\begin{figure}[h!]
\centering
\includegraphics[scale=0.66]{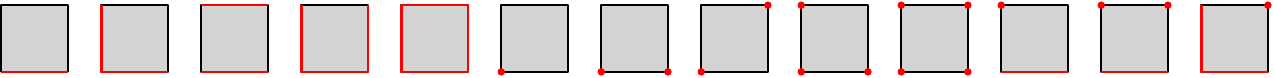}
\caption{Example for the square tessellation of the 13 different possible types labelled according to the intersection (in red) of the grey tile with $\cK_n$ during the iterative construction of $\cK_\infty$.}\label{fig:HexaAllTheTypes}
\end{figure}

Let $\tmax$ be the number of types and $Z_n^t$, $1\le t\le \tmax$, be the cardinality of $c_\cT^{-1}\cR_n^{t}$ where $\cR_n^{t}$ is the set of tiles of $\cR_n$ with type $t$.
In particular,
\begin{equation}\label{eq:RnavecZn}
R_n=c_\cT\sum_{t=1}^{\tmax} Z_n^t= c_\cT\langle Z_n,{\mathbf 1}\rangle
\end{equation}
where ${\mathbf 1}$ is the vector with $\tmax$ entries all equal to $1$ and $\langle\cdot,\cdot\rangle$ stands for the usual Euclidean scalar product with $\|\cdot\|$ its associated norm. Let us notice that when $n=0$ and $p_*=0$, the set $\cR_0$ consists in $c_\cT$ tiles of same type, denoted by T1 in Section \ref{sec:Model} and pictured as the first type in Figure \ref{fig:HexaAllTheTypes}, which corresponds to sharing exactly one edge with $\cK_0$, see Figure \ref{fig:setRn} (a). This implies that $Z_0$ is the vector with first entry equal to $1$ and all the other entries equal to $0$. When $p_*=1$, the initialization actually depends on the nature of the tessellation, see Section \ref{sec:Model}. For instance, when the underlying tessellation is square, the set $\cR_n$ consists in $4$ squares of type T1 and $4$ squares of type T0, which means that they share with $\cK_0$ exactly one vertex, see the sixth type in Figure \ref{fig:HexaAllTheTypes}. Consequently, $Z_0$ is the vector with first and final entries equal to $1$ and all the other entries equal to $0$. As we are exclusively interested in the asymptotics of $(Z_n)_{n\ge 0}$, we observe that the initialization of the sequence does not influence our study.

The structure of the sequence $(Z_n=(Z_n^1,\ldots,Z_n^{\tmax}))_{n\geq0}$ is described in Lemma \ref{lem:galton} below.

\begin{lem}\label{lem:galton}
The sequence $(Z_n)_{n\geq0}$ is a $t_{\mathrm{max}}$-dimensional multitype Galton-Watson process and belongs to the super-critical case.
\end{lem}

\begin{proof}
Looking back at the construction of the sequence $(\cK_n)_{n\ge 0}$ described in Section \ref{sec:intro}, we notice that the type of a tile $\cT_{n,\ell}$ fixes the random rules for the construction of $\cT_{n,\ell}\cap\cK_{n+1}$
and subsequently, the set of tiles from $\cR_{n+1}$ included in $\cT_{n,\ell}$ and their types. In particular, the reproduction law associated with a particular tile is time-homogeneous, i.e. it only depends on its type and does not depend on the number of its generation. Let $(C_k^{t,u})_{t,u,k}$ be a collection of real random variables such that
\begin{itemize}[parsep=-0.15cm,itemsep=0.2cm,topsep=0.2cm,wide=0.15cm,leftmargin=0.65cm]
\item[$-$] they are independent when $t$ and $k$ are fixed,
\item[$-$] they are identically distributed for fixed types $t$ and $u$, distributed as the number of children of type $u$ of a parent of type $t$. 
\end{itemize}

\noi The previous remark then implies that $Z_{n+1}^u$, $1\le u\le \tmax$, can be rewritten in distribution as
\begin{equation}\label{eq:Zn}
Z_{n+1}^u\overset{(d)}{=}\sum_{t=1}^{\tmax}\sum_{k=1}^{Z_n^t}C_k^{t,u}. \end{equation}
This shows that $(Z_n)_{n\ge 0}$ is a Galton-Watson process.

Now, let $\cG_n$ be the $\sigma$-algebra generated by the Bernoulli variables necessary to construct the $n$th generation. In particular, the $\sigma$-algebras $(\cG_n)_{n\ge 1}$ are mutually independent. Denoting by $\cF_n$ the $\sigma$-algebra generated by $\cup_{m\le n} \cG_m$, we notice that each $Z_n$ is $\cF_n$-measurable. Finally, since the set $\cR_n$ is not empty at each step, it follows that the Galton-Watson process $(Z_n)_{n\ge 0}$ survives almost surely, i.e. it belongs to the super-critical case. 
\end{proof}

We denote by $\T$ the random tree associated with $(Z_n)_{n\ge 0}$ and by $\partial \T$ the boundary set of $\T$, i.e. the set of infinite sequences of integers. We then endow $\partial \T$ with the following distance: for any $(i_n)_{n\ge 0}$ and $(j_n)_{n\ge 0}$, let 
\begin{equation}\label{eq:treedist}
\dd_{\partial \T}\big((i_n)_{n\ge 0}, (j_n)_{n\ge 0}\big)=\la^{-m}
\end{equation}
where $m$ is the largest integer such that $i_n=j_n$ for every $0\le n\le m$. In particular, $\partial \T$ is a compact set. To each sequence $(i_n)_{n\ge 0}$ of $\partial \T$, we associate a sequence $(t_n)_{n\ge 0}$ with values in $\{1,\ldots,\tmax\}$.

For almost every $x\in\partial \cK_{\infty}$, there exists a unique sequence of tiles $(\cT_{n}^{(x)})_{n\ge 0}$ such that $\cT_n^{(x)}\in \la^{-n}\cT$, $x\in \cT_n^{(x)}$ and $\cT_{n+1}^{(x)}\subset \cT_n^{(x)}$ for every $n\ge 0$. By the encoding described above, a unique element $\varphi(x)$ of $\partial \T$ corresponds to that sequence of tiles. The application
\begin{equation}\label{eq:defvarphi}
\varphi:\partial \cK_\infty\longrightarrow\partial \T
\end{equation}
is defined almost everywhere and its inverse is Lipschitz. Indeed, let $x,y\in\partial \cK_\infty$ and let $n\ge 0$ be the maximal integer such that there exists a common tile belonging to $\la^{-n}\cT$ and containing both $x$ and $y$. Then $\|x-y\|\le \la^{-n}=\dd_{\partial\T}(\varphi(x),\varphi(y))$. 

\newpage
This allows us to use more practical notation for the balls of $\partial \T$. Indeed, the inverse image of the ball centered at some $(i_n)_{n\ge 0}\in \partial \T$ and of radius $\la^{-m}$ is the intersection of $\partial \cK_{\infty}$ with a fixed tile $\cT_{m,\ell}$ from $\la^{-m}\cT$. Consequently, for sake of simplicity, we use henceforth $\varphi(\cT_{m,\ell}\cap \partial \cK_\infty)$ for a generic ball of $\partial \T$ of radius $\la^{-m}$.

\subsection{Properties of the process $(Z_n)_{n\ge 0}$}\label{sec:propZn}

\pass

Let $\M$ be the matrix constituted with the entries $\M^{t,u}$, $1\le t,u\le \tmax$ such that $\M^{t,u}$ is the expectation of the number $C^{t,u}$ of children of type $u$ given by a parent of type $t$, i.e.
\begin{equation}\label{eq:defMatrixM}
\M=\left[\begin{matrix} 
\E[C^{1,1}] & \ldots & \E[C^{1,\tmax}] \\
\vdots & \ddots & \vdots \\
\E[C^{\tmax,1}] & \ldots & \E[C^{\tmax,\tmax}]
\end{matrix}\right].  
\end{equation}

\noi We assume that $\M$ is primitive, i.e. there exists a power of $\M$ with positive entries only. In particular, by Perron-Frobenius theorem, its spectral radius denoted by $\rho_\M$ is a simple eigenvalue associated with a unit eigenvector $v=(v^1,\ldots,v^{\tmax})$ with positive entries $v^i>0$.

\begin{lem}\label{lem:rhod}
The spectral radius of $\M$ satisfies $\rho_\M\in[\la,\la^2)$.
\end{lem}

\begin{proof} 
By Perron-Frobenius Theorem \cite[Th. 2.35]{BP94}, we get that
\begin{equation}\label{eq:maxminPF}
\min_{1\le t\le \tmax}\sum_{u=1}^{\tmax} C^{t,u}\le \rho_\M\le \max_{1\le t\le \tmax}\sum_{u=1}^{\tmax} C^{t,u}.
\end{equation}
For every $1\le t\le \tmax$, the sum of all coefficients $C^{t,u}$ of $\M$ over $u$ is bounded almost surely by $\la^2-1$. Indeed, it represents the mean number of tiles in $\cR_{n+1}$ included in some tile $\cT_{n,\ell}$ of $\la^{-n}\cT$. Since $\la\ge 3$, it is strictly less than the total number of tiles of generation $n+1$ included in $\cT_{n,\ell}$, i.e. less than $\la^2-1$. Using the upper bound of $\rho_\M$ given at \eqref{eq:maxminPF}, we obtain that $\rho_\M<\la^2$. We now concentrate on the lower bound for $\rho_\M$. We denote by $T$ the set of types such that the corresponding tile $\cT_{n,\ell}$ in $\cR_n$ shares at least one edge with $\cK_n$. For such tile $\cT_{n,\ell}$, the number of tiles of $\cR_{n+1}$ in $\cT_{n,\ell}$ is at least $\la$. Indeed, in the case of the square tessellation, there are exactly $\la$ squares of $\la^{-(n+1)}\cT$ which share an edge with $\cK_n$ and for each such square $\cT_{n+1,\ell}$, either $\cT_{n+1,\ell}$ or the square above belongs to $\cR_{n+1}$, see Table \ref{tab:SquareChild}. A similar argument works for the other two tessellations, see Figures \ref{fig:Model0HexaR0AllTheChild} and \ref{fig:Model0TriangleAllTheChild}. We assert that $\rho_\M$ is lower bounded by the spectral radius of the submatrix $\M_{T\times T}=[\M^{t,u}]_{(t,u)\in T\times T}$, see \cite[Cor. 1.6]{BP94}, which is larger than $\la$ by the lower bound in \eqref{eq:maxminPF} applied to $\M_{T\times T}$. 
\end{proof}

We recall that $(\cF_n)_n\ge 0$ is the $\sigma$-algebra generated by the first $n$ generations, i.e. the type of each tile of the first $n$ generations. Thanks to the almost sure identity $\E[Z_{n+1}^*\,|\,\cF_n]=\M^*Z_n^*$, where $*$ stands for the transpose of a matrix, the sequence $(M_n)_{n\ge0}$ defined by 
\begin{equation}\label{eq:martingale}
M_n=\rho_\M^{-n}\langle Z_n,v\rangle  
\end{equation}
is a positive $(\cF_n)_{n\ge 0}$-martingale, see \cite[Th. 4 p. 193]{AN72}, which converges almost surely. 
In particular, this implies that the random vector $Z_n$ satisfies an almost sure law of large numbers.

\newpage

\begin{prop}\label{prop:convbyAN}
There exists a positive random variable $W$ such that 
\begin{equation}\label{eq:TheoremAthreyaNey}
\P\big(\lim_{n\to\infty}\rho_\M^{-n}Z_n=W v\big)=1 \text{ and } \P(W>0)=1.
\end{equation}
\end{prop}

\begin{proof}
This is a direct consequence of \cite[Th. 2 (ii) p. 192]{AN72} whose assumptions are satisfied in our case. Indeed, observe that the random variables $C^{t,u}$ are bounded almost surely by $\la^2-1$. Consequently, we get
\begin{equation}\label{eq:condGW} 
\E\bigg[\sum_{1\le t, u\le \tmax} C^{t,u}\log C^{t,u}\bigg]<\infty   
\end{equation}    

Moreover, since at any generation, the set $\cR_n$ with cardinality $R_n$ given by \eqref{eq:RnavecZn} is not empty, the branching process $(Z_n)_{n\ge 0}$ does not vanish almost surely.  
\end{proof}

We conclude this section by introducing a new set of martingales in the same spirit as $(M_n)_{n\ge 0}$. When $\cT_{m,\ell}$ is a fixed tile of generation $m$, i.e. belonging to $\la^{-m}\cT$, we can define $Z_n(\cT_{m,\ell})$ as the vector constituted with the cardinality of the tiles of each type at generation $n\ge m$ which belong to $\cR_n$ and are included in $\cT_{m,\ell}$. In particular, when $\cT_{m,\ell}\not\in \cR_m$, $Z_n(\cT_{m,\ell})=0$ and when $\cT_{m,\ell}\in \cR_m$, only one entry of $Z_m(\cT_{m,\ell})$ is different from $0$ and equal to $1$ and this entry corresponds to the type of the tile $\cT_{m,\ell}\in \cR_m$. 
Moreover, by subtiles additivity, we obtain for any $n\ge m$
\begin{equation}\label{eq:subtilesadd}
Z_n=\sum_{\cT_{m,\ell}\in \cR_m}Z_n(\cT_{m,\ell}).    
\end{equation}
We then consider the sequence $(M_n(\cT_{m,\ell}))_{n\ge m}$ defined by 
\begin{equation}\label{eq:martingalebis}
M_n(\cT_{m,\ell}) = 
\rho_\M^{-n}\langle Z_n(\cT_{m,\ell}),v\rangle
\end{equation}
which is again a positive and convergent martingale with respect to the $\sigma$-algebra $(\cF_n)_{n\ge m}$. Moreover, \eqref{eq:subtilesadd} implies the decomposition
\begin{equation}\label{eq:decompMnarfixe}
M_n=\sum_{\ell\ge 1}M_n(\cT_{m,\ell}).    
\end{equation}

\subsection{Hausdorff dimension of the boundary set of a multitype Galton-Watson tree}\label{sec:GWTree}

\noi

We can introduce now a random measure on $\partial \T$ or equivalently on the random set $\partial \cK_\infty$ through the pushforward by the application $\varphi$ introduced at \eqref{eq:defvarphi}. Thanks to the convergence of the martingale defined at \eqref{eq:martingalebis}, we can associate to each tile $\cT_{m,\ell}$, $m\ge 0$, $\ell\ge1$, the random variable 
\begin{equation}\label{eq:muk}
\mu(\cT_{m,\ell})=\lim_{n\to\infty} M_n(\cT_{m,\ell})
\end{equation}
which satisfies 
\begin{equation}\label{eq:espermuk}
\E[\mu(\cT_{m,\ell})\,|\,\cF_m]=M_m(\cT_{m,\ell})=\rho_\M^{-m}v^{t_{m,\ell}}
\end{equation}
where $t_{m,\ell}$ is the type of the tile $\cT_{m,\ell}$ when $\cT_{m,\ell}$ belongs to $\cR_m$. Moreover, $\mu$ is additive with respect to the tiles of generation $m+1$. Indeed, denoting by $\cT_{m+1,\ell,\ell'}$, $\ell'\ge1$, the subtiles of generation $m+1$ (in finite number) included in $\cT_{m,\ell}$, we see immediately that for each $n\ge m+1$, $M_n(\cT_{m,\ell})=\sum_{\ell'\ge1}M_n(\cT_{m+1,\ell,\ell'})$ which implies that
\begin{equation*}
\mu(\cT_{m,\ell})=\sum_{\ell'\ge1}\mu(\cT_{m+1,\ell,\ell'}).  
\end{equation*}
This classically extends to an outer measure 
\begin{equation}\label{eq:defmesext}
\mu(B)=\inf \bigg\{\sum_{i}\mu(\cT_{m_i,\ell_i}):B\subset\bigcup_i \cT_{m_i,\ell_i}\bigg\}, \text{ $B$ Borel set of $\R^2$}.
\end{equation}
This, in turn, induces the definition of a Borel measure $\mu$ which satisfies \eqref{eq:espermuk} and whose support is included in $\cap_{m\ge 0}\cR_m=\partial \cK_\infty$. Moreover, thanks to \eqref{eq:decompMnarfixe} and Proposition \ref{prop:convbyAN}, the total mass $\mu(\partial \cK_\infty)$ of $\mu$ is equal to $W \|v\|^2=W$ which is positive almost surely. 

\pass We intend now to deduce the calculation of the Hausdorff measure of $\partial \T$ from the use of the pushforward $\mu_\varphi$ by $\varphi$ of the measure $\mu$ defined at \eqref{eq:defmesext}. Indeed, $\mu_\varphi$ is expected to play the role of a Frostman measure for $\partial \T$, see \cite[Section 3.1]{BP16}. In particular, this requires to estimate the measure of a generic ball of $\partial \T$ that we rewrite as $\varphi(\cT_{m,\ell}\cap \partial \cK_\infty)$ according to the discussion at the end of Section \ref{sec:galton}. This paves the way for Section \ref{sec:proof} but we note here that a wording purely in terms of tree would have been possible as well.

\pass We recall the definition of the Hausdorff dimension $\dimh(E)$ of a metric space $E$, i.e.
\begin{equation}\label{eq:defhsdim}
\dimh(E)=\inf \{s\geq0 : \cH^s(E)=0\} = \sup \{s\geq0 : \cH^s(E) =\infty\}
\end{equation}
where $\cH^s(E)$, $s>0$, denotes the $s$-dimensional Hausdorff measure of $E$, i.e.
\begin{equation}\label{eq:defhs}
\cH^s(E) = \lim_{\eps\to 0} \cH^s_\eps(E) = \inf_{\eps>0} \cH^s_\eps(E)
\end{equation}
where
\begin{equation}\label{eq:defhsbis}
\cH^s_\eps(E)= \inf\bigg\{ \sum_{i\ge0} (\diam U_i)^s \,:\, E\subset \bigcup_{i\ge 0} U_i \text{ and } \diam U_i\le \eps\bigg\}.
\end{equation}

In Theorem \ref{theo:dimtree}, we obtain the calculation of the Hausdorff dimension of $\partial \T$ through the asymptotics of $\mu(\cT_{m}^{(x)})$ for almost all $x\in \R^2$, where we recall that $\cT_m^{(x)}$ is the tile of $\la^{-m}\cT$ which contains $x$ (and which is unique for almost every $x$). Although our statement and proof are specific to the tree as constructed in Section \ref{sec:galton}, we claim that the method naturally extends to any multitype Galton-Watson which satisfies the assumption given at \eqref{eq:condGW}. 

\begin{theo}\label{theo:dimtree}
The boundary set $\partial \T$ has Hausdorff dimension almost surely equal to
\begin{equation}\label{eq:dimtree}
d=\frac{\log\rho_\M}{\log \la}\in[1,2).
\end{equation}
\end{theo}

\begin{proof}
The reasoning is done in two steps. First, we provide a logarithmic equivalent for the measure of a ball, i.e. we show that for $\mu$-almost every $x\in\partial \cK_\infty$ and $\P$-almost surely,
\begin{equation}\label{eq:mesureduneboule}
\lim_{m\to\infty}\frac{\log\mu(\cT_m^{(x)})}{m}=-\log\rho_\M.
\end{equation}
Once \eqref{eq:mesureduneboule} is derived, we deduce \eqref{eq:dimtree} by combining it with the definition \eqref{eq:defhsdim}--\eqref{eq:defhsbis} applied to a particular covering of $\partial\T$ with balls.

\pass{\bf\it Step 1: Proof of \eqref{eq:mesureduneboule}.} We follow closely the method developed in \cite[Theorem 1]{hawk81} and adapt it to the multitype setting. We denote by $t_m^{(x)}$ the type of the tile $\cT_m^{(x)}$ and calculate the following energy-type integral
\begin{equation}\label{eq:relationHawkes}
\E\bigg[\int_{\partial \cK_\infty} \frac{v^{t_m^{(x)}}}{\rho_\M^m\mu(\cT_{m}^{(x)})}\dd\mu(x)\bigg]=\rho_\M^{-m}\E\bigg[\sum_{\cT_{m,\ell}\in \cR_m}v^{t_{m,\ell}}\bigg]=\rho_\M^{-m}\sum_{t=1}^{\tmax} v^t\E[Z_m^t]=1
\end{equation}
where we have used in the last equality the martingale property of $(M_n)_{n\ge 0}$ defined at \eqref{eq:martingale}.
\pass 
Using Fubini's theorem, we deduce from \eqref{eq:relationHawkes} that
\begin{equation}\label{eq:relationHawkesapresFubini}
\E\bigg[\int_0^{\infty}\mu\big(\big\{x\in \partial \cK_\infty:\rho_\M^m\mu(\cT_m^{(x)})\le \frac{v^{t_m^{(x)}}}{s}\big\}\big)\dd s\bigg]=1
\end{equation}
which, in turn, implies, thanks to Markov's inequality, that for any $n\ge 1$,
\begin{equation}
\P\bigg(\mu\big(\big\{x\in\partial\cK_\infty:\rho_\M^m\mu(\cT_m^{(x)})\le \frac{v^{t_m^{(x)}}}{n^4}\big\}\big)\ge \frac1{n^2}\bigg)\le \frac1{n^2}.    
\end{equation}
In the same way as in \cite[p. 375]{hawk81}, Borel-Cantelli's lemma used twice then leads us to 
\begin{equation}\label{eq:liminfmuT}
\liminf_{m\to\infty}\frac{\log\mu(\cT_m^{(x)})}{m}\ge -\log\rho_\M
\end{equation}
almost surely and for $\mu$-almost every $x\in \partial \cK_\infty$.

The reverse inequality relies on a self-similarity argument combined with the independence property of the Galton-Watson construction. Indeed, let $m\ge 0$ be fixed again. We recall that for any $\cT_{m,\ell}\in \cR_m$ and $n\ge m$, $Z_n(\cT_{m,\ell})$ is the vector constituted with the number of descendants of generation $n$ and of each type $1\le t\le \tmax$. The sequences $\{Z_n(\cT_{m,\ell}):\cT_{m,\ell}\in \cR_m^t, n\ge m\}$ for $1\le t\le \tmax$ are mutually independent and each of them is constituted with i.i.d. variables which satisfy the equality in law 
\begin{equation*}
Z_n(\cT_{m,\ell})\overset{(d)}{=}Z_{n-m}(\cT_{0}^{t})    
\end{equation*}
where $t=t_{m,\ell}$ is the type of the tile $\cT_{m,\ell}$, and $\cT_0^t$ is a generic tile of $\cT$ with type $t$. Consequently, recalling the definition of $M_n(\cT_{m,\ell})$ and $\mu(\cT_{m,\ell})$ at \eqref{eq:martingalebis} and \eqref{eq:muk} respectively, we get that the sequences $\{\frac{\rho_\M^m\mu(\cT_{m,\ell})}{v^t}:\cT_{m,\ell}\in \cR_m^t\}$, $1\le t\le \tmax$, are mutually independent, independent of $Z_m$, and constituted with i.i.d. random variables that we denote by $W_{m,\ell}^t$. This implies that, for any $w>0$,
\begin{align}\label{eq:waytolimsup}
\E\big[\mu(\{x\in\partial \cK_\infty:\rho_\M^m   \mu(\cT_{m}^{(x)})\ge v^{t_m(x)}w\})\big]
 & =\rho_\M^{-n}\sum_{t=1}^{\tmax}v^{t}\E\bigg[\sum_{\cT_{m,\ell}\in\cR_m^t}W_{m,\ell}^t{\mathbf 1}_{\{W_{m,\ell}^t\ge w\}}\bigg]\notag\\
 &=\rho_\M^{-m}\sum_{t=1}^{\tmax}v^t\E[Z_m^t]\E\big[W_{m,\ell}^t{\mathbf 1}_{\{W_{m,\ell}^t\ge w\}}\big]\notag\\
 &\le \rho_\M^{-m}\sum_{t=1}^{\tmax}v^t\E[Z_m^t]\sum_{t=1}^{\tmax} \E\big[W_{m,\ell}^t{\mathbf 1}_{\{W_{m,\ell}^t\ge w\}}\big]\notag\\
 &=\sum_{t=1}^{\tmax} \E\big[W_{m,\ell}^t{\mathbf 1}_{\{W_{m,\ell}^t\ge w\}}\big]
 \end{align}
where the last equality comes again from the martingale property of $(M_n)_{n\ge 0}$. Since the entries of $\M$ are bounded, it turns out that $\E[W_{m,\ell}^t\log W_{m,\ell}^t]$ is finite for any $1\le t\le \tmax$. Thus, \eqref{eq:waytolimsup} implies that the series
$\sum\E\big[\mu(\{x\in\partial \cK_\infty:\rho_\M^m \mu(\cT_{m}^{(x)})\ge v^{t_m^{(x)}}(1+\varepsilon)^k\})\big]$ converges for any $\varepsilon>0$, and by Borel-Cantelli's lemma, we get that
\begin{equation}\label{eq:limsupmuT}
\limsup_{m\to\infty}\frac{\log\mu(\cT_{m}^{(x)})}{m}\le -\log\rho_\M
\end{equation}
almost surely and for $\mu$-almost every $x\in \partial \cK_\infty$. We then deduce \eqref{eq:mesureduneboule} from both \eqref{eq:liminfmuT} and \eqref{eq:limsupmuT}.

\pass{\bf\it Step 2: Proof of \eqref{eq:dimtree} through the use of \eqref{eq:mesureduneboule}.} The calculation of the Hausdorff dimension of $\partial \T$ is derived from \eqref{eq:mesureduneboule} by classical techniques from geometric measure theory, see \cite{falco03,BP16}. To the best of our knowledge, the literature does not include a general result designed for our purpose. That is why we opted to include below an exhaustive proof which goes along lines very similar to the references cited above.

Let $\eps>0$ and $n_\eps\geq 0$ such that $\la^{-n_\eps}<\eps$. Let $\eta\in(0,d)$. It follows from \eqref{eq:mesureduneboule} that, for $\mu$-almost every $x\in\partial \cK_\infty$ and $\P$-almost surely, we can find a (minimal) integer $m_x\geq n_\eps$ such that, for all $m\geq m_x$,
\begin{equation}\label{eq:encadrement muTmx}
\la^{-(d+\eta)m} \leq  \mu(\cT_{m}^{(x)}) \leq \la^{-(d-\eta)m}. 
\end{equation}
The family $\{\varphi(\cT_{m_x}^{(x)}\cap\partial\cK_\infty) : x\in\partial\cK_\infty\}$ is then a covering of $\partial\T$ by balls with radius $\la^{-m_x}<\eps$. Since $\partial\T$ is an ultrametric and compact space, we can assume that these balls are disjoint and in finite number. We denote the associated tiles by $\cT_{m_{x_1}}^{(x_1)},\ldots,\cT_{m_{x_q}}^{(x_q)}$ where $q\ge 1$. Therefore, 
\begin{equation*}
 2^{-(d+\eta)}\sum_{i=1}^q \big(\diam
\varphi(\cT_{m_{x_i}}^{(x_i)}\cap\partial\cK_\infty)
\big)^{d+\eta} \leq  \sum_{i=1}^q (\la^{-m_{x_i}})^{d+\eta}
\leq \sum_{i=1}^q  \mu(\cT_{m_{x_i}}^{(x_i)}) \leq \mu (\partial\cK_\infty) < \infty.
\end{equation*}
Recalling \eqref{eq:defhsbis}, we obtain that $\cH^{d+\eta}_\eps(\partial\T)<\infty$. Letting $\eps\to0$, we deduce that $\cH^{d+\eta}(\partial\T)<\infty$, which implies in turn that $\dimh(\partial\T)\leq d+\eta$. Since this holds for every $\eta>0$, we get $\dimh(\partial\T)\leq d$.

Let us now prove the reverse inequality. For every $n\geq0$, we consider the set
\begin{equation*}
I_n=\left\{x\in\partial\cK_\infty : m_x\le n\right\}.
\end{equation*}
Observe that $I_n\subset I_{n+1} \nearrow \partial\cK_\infty$ and $\mu(I_n)\nearrow \mu(\partial\cK_\infty)<\infty$. For a fixed family $(n_x)_{x\in \partial \cK_\infty}$ of integers such that $n_x\ge n$ for any $x\in \partial \cK_\infty$, we consider a covering of $\partial \T$ by balls $\varphi(\cT_{n_x}^{(x)}\cap \partial \cK_\infty)$ with radii $\la^{-n_x}$. Again, we may assume that the covering is finite and constituted with disjoint balls $\varphi(\cT_{n_{x_1}}^{(x_1)}\cap \partial \cK_\infty),\ldots,\varphi(\cT_{n_{x_q}}^{(x_q)}\cap \partial\cK_\infty)$. We fix $n\ge n_\varepsilon$ so that $\la^{-n_x}< \eps$. Thanks to the inequality $n_x\ge n$, we also get that $m=n_x$ satisfies \eqref{eq:encadrement muTmx} as soon as $x\in I_n$. Moreover, as soon as $\cT_{n_x}^{(x)}$ intersects $I_n$ at a point $y$, the ball $\varphi(\cT_{n_x}^{(x)}\cap \partial \cK_\infty)$ coincides with $\varphi(\cT_{n_x}^{(y)}\cap \partial \cK_\infty)$ by property of the ultrametric topology, which implies that $\cT_{n_x}^{(x)}=\cT_{n_x}^{(y)}$ satisfies again \eqref{eq:encadrement muTmx} for $m=n_x$. Consequently, recalling that the diameter of one tile of $\cT$ is $1$, we obtain that
\begin{equation*}
\sum_{i=1}^q \big(\diam\cT_{n_{x_i}}^{(x_i)}\big)^{d-\eta} 
\geq  \sum_{i=1}^q \indicat_{\{\cT_{n_{x_i}}^{(x_i)}\cap I_n\ne \emptyset\}}\la^{-(d-\eta)n_{x_i}}
\geq  \sum_{i=1}^q\indicat_{\{\cT_{n_{x_i}}^{(x_i)}\cap I_n\ne \emptyset\}}\mu(\cT_{n_{x_i}}^{(x_i)}) \geq \mu(I_n).
\end{equation*}
Actually, in the same way as it is classically done for fractal sets, see e.g. the discussion (2) in \cite[Section 1.4]{BP16}, the series of inequalities above extends to all $\varepsilon$-coverings of $\partial \cK_\infty$ since the family of balls of type $\varphi(\cT_{n_x}^{(x)}\cap \partial \cK_\infty)$, $x\in \partial\cK_\infty$ and $n_x\ge n$, is a Vitali covering with the bounded subcover property. Hence, taking the infimum and  using \eqref{eq:defhsbis}, we deduce from the previous inequality that $\cH^{d-\eta}_\eps(\partial\T)\geq \mu(I_n)$. Letting $\eps\to 0$ we obtain $\cH^{d-\eta}(\partial\T)\geq \mu(I_n)$. Letting $n\to\infty$ we get $\cH^{d-\eta}(\partial\T)\geq \mu(\partial\cK_\infty)>0$. Hence $\dimh(\partial\T)\geq d-\eta$. Since this holds for all $\eta>0$, we get $\dimh(\partial\T)\geq d$ and this completes the proof of Theorem \ref{theo:dimtree} .
\end{proof}

\addtocontents{toc}{\vspace{0.2cm}}%
\section{Proof of Theorems \ref{theo:main} and \ref{theo:perimaire}}\label{sec:proof}

The results from Theorems \ref{theo:main} and \ref{theo:perimaire} are proved in the following order: we start with the calculation of the box-dimension of $\partial \cK_\infty$ as a consequence of the martingale convergence in Proposition \ref{prop:convbyAN}, and we show that the Hausdorff dimension has the same value almost surely thanks to the identification with $\partial \T$ and the application of Theorem \ref{theo:dimtree}. Next, we use the mass distribution principle to show that the $d$-dimensional Hausdorff measure of $\partial\cK_\infty$ is positive in mean and we also prove that it is finite almost surely. Finally, we turn to the almost sure limits for both the perimeter of $\partial \cK_n$ and the area of $\cK_{n+1}\setminus \cK_n$ which are again byproducts of the coding and the martingale argument.

\subsection{Proof of Theorem \ref{theo:main}: exact calculation of $\dimbox(\partial \cK_{\infty})$ and $\dimh(\partial \cK_{\infty})$}\label{sec:dimfrac}

\noi

First, we calculate $\dimbox(\partial\cK_\infty)$ by combining Lemma \ref{lem:cover} and \eqref{eq:RnavecZn} with Proposition \ref{prop:convbyAN}. Indeed, the result \eqref{eq:TheoremAthreyaNey} implies that $R_n=\langle Z_n,{\mathbf 1}\rangle$ is asymptotically proportional to $\rho_\M^n$. Then, recalling \eqref{eq:dimbox}, we deduce that almost surely,
\begin{equation}\label{eq:dimbox2}
\dimbox(\partial \cK_\infty) = 
\limsup_{n\to \infty} \frac{\log R_n}{\log \la^n} 
= \limsup_{n\to \infty} \frac{\log \rho_\M^n }{\log \la^n} = \frac{\log \rho_\M}{\log \la}.
\end{equation}
In particular, this proves that the limsup in \eqref{eq:dimbox} is a real limit.

We now turn to the determination of $\dimh(\partial \cK_\infty)$. Contrary to the box-dimension, the calculation of the Hausdorff dimension requires to consider all the coverings of $\partial \cK_\infty$ and not only the regular $\varepsilon$-mesh. Nevertheless, it is classically bounded from above by the box-dimension. Thanks to \eqref{eq:dimbox2}, this implies that almost surely
\begin{equation}
\dimh(\partial \cK_\infty)\le \dimbox(\partial \cK_\infty)= d= \frac{\log \rho_\M}{\log \la}.   
\end{equation}
In order to bound from below $\dimh(\partial \cK_\infty)$, we use both the coding of $\partial \cK_\infty$ by the boundary $\partial \T$ of the Galton-Watson tree, such coding being provided by the function $\varphi$ defined at \eqref{eq:defvarphi}, and the expression of $\dimh(\partial \T)$ obtained in Theorem \ref{theo:dimtree}. The only obstacle here comes from the fact that $\varphi$ is not bi-Lipschitz and actually not even Lipschitz. But we make the important observation that the image by $\varphi$ of a Borel set $U$ with $\diam U \leq1$ is included in a finite number $k_0$ of balls of $\partial \T$ with radius $\la^{-n}$ where $n=\lfloor \frac{-\log\diam U|}{\log\la}\rfloor$ and where $k_0$ only depends on the tessellation $\cT$ and $\la$. Indeed, if $x,y\in \partial \cK_\infty$ such that $\|x-y\|\le \la^{-n}$, then  $y$ belongs to a tile of $\la^{-n}\cT$ which is at graph distance from the tile containing $x$ bounded by some constant $c$ depending on $\la$ and on the diameter of one tile of $\cT$, assumed here to be equal to $1$. The required integer $k_0$ is then the number of neighboring tiles at graph distance at most $c$ from a fixed tile of $\cT$.

We now implicitly fix the sample $\omega$ in the event of probability $1$ on which $\dimh(\partial\T)=d$. Let us consider $s\in (0,d)$, $n\ge 1$ and a covering of $\partial \cK_\infty$ by Borel sets $U_i$, $i\ge 0$, with diameter bounded by $\la^{-n}$. Then, for each $i\ge 0$, $\varphi(U_i)$ is included in the union of $k_0$ balls $B_{i,1},\ldots,B_{i,k_0}$ of radius $\la^{-n_i}$ of $\partial \T$ where $n_i=\lfloor \frac{-\log\diam U_i}{\log\la}\rfloor$. This implies that $\partial \T$ is covered by the union of all of these balls $B_{i,j}$ for $i\ge 0$ and $1\le j\le k_0$. Using now \eqref{eq:defhs}, \eqref{eq:defhsbis} and Theorem \ref{theo:dimtree}, we deduce that
\begin{align*}
\sum_{i\ge 0}(\diam U_i)^s\ge \la^{-s}\sum_{i\ge 0}\la^{-n_is}=\frac{\la^{-s}}{2^sk_0}\sum_{i\ge 0}\sum_{j=1}^{k_0}(\diam B_{i,j})^s\ge \frac{\la^{-s}}{2^sk_0}\cH^s(\partial \T)>0.
\end{align*}
Taking now the minimum in the left-hand side over all such coverings $(U_i)_{i\ge 0}$ and letting then $n\to\infty$, we deduce that $\cH^s(\partial \cK_\infty)>0$ almost surely. This being true for any $s<d$, we obtain thanks to \eqref{eq:defhsdim} that $\dimh(\partial \cK_\infty)\ge d$ and subsequently the required equality.

\subsection{Proof of Theorem \ref{theo:main}: the $d$-dimensional Hausdorff measure of $\partial\cK_{\infty}$}\label{sec:proofhausmeasure}

\noi

First, we state a useful property of $\mu$, namely that it is $d$-H\"olderian in mean. In other words, there exists a constant $c_\mu>0$ such that, for every bounded Borel set $U\subset\R^2$ with $\diam U <1$, 
\begin{equation}\label{eq:boundmu}
\E\big[\mu(U)\big] \le c_\mu (\diam U)^d.
\end{equation}
Indeed, let $n\ge 1$ such that $\la^{-n-1}\le \diam U < \la^{-n}$. Denoting by $\cU_n(U)$ the set of tiles of generation $n$ needed to cover $U$, it holds that the cardinality of $\cU_n(U)$ is at most $(\diam U)/\la^{-n-1} < \lambda$ up to a multiplicative constant $c>0$. Therefore, we obtain with the help of \eqref{eq:espermuk} and the identity $\rho_\M=\la^d$,
\begin{equation}\label{eq:boundmu2}
\E\big[\mu(U)\,|\,\cF_n\big]
 \le \sum_{\cT_{n,\ell}\in \cU_n(U)}\E\big[\mu(\cT_{n,\ell})\,|\,\cF_n\big]\le c\la\big(\max_{1\le i\le m}v^i\big)\la^{-nd} \le c_\mu (\diam U)^d
\end{equation} 
where $c_\mu=c\la^{d+1}\max_{1\le i\le m}v^i$. Taking the expectation in \eqref{eq:boundmu2} gives \eqref{eq:boundmu}.

We then use the classical mass distribution principle, see e.g. \cite[Chap. 4, p. 60]{falco03}, to deduce from \eqref{eq:boundmu} that 
\begin{equation*}
\E\big[\cH^d(\partial \cK_\infty)\big]\ge \E[\mu(\partial \cK_\infty)]/c_\mu.    
\end{equation*}
Indeed, we can follow line by line the technique developed in \cite{falco03}, only we slightly adapt it by taking the expectation of the infimum required in \eqref{eq:defhsbis}. 

We then notice that
\begin{equation}\label{eq:espermuinfini}
\E[\mu(\partial \cK_\infty)]=\E\big[\E[\mu(\partial \cK_\infty) \,|\, \cF_n]\big]=\E[M_n]=\E[M_0]=\langle Z_0,v\rangle>0, \end{equation}
to deduce that $\E\big[\cH^d(\partial \cK_\infty)\big]>0$. 

The almost sure upper bound of $\cH^d(\partial \cK_\infty)$ is obtained by considering the particular covering of $\partial \cK_\infty$ with the tiles from $\cR_n$. 
Indeed,  for every $n\ge 1$, we get almost surely, thanks to \eqref{eq:defhsbis}, \eqref{eq:TheoremAthreyaNey} and again the identity $\rho_\M=\la^d$,
\begin{equation*}
\cH^d(\partial \cK_\infty) \le c_\mu\la^{-nd}R_n = c_\mu\rho_\M^{-n}\langle Z_n,{\mathbf 1}\rangle\longrightarrow c_\mu W \langle v,{\mathbf 1}\rangle.
\end{equation*}
This proves that $\cH^d(\partial \cK_\infty)<\infty$ almost surely.

\subsection{Proof of Theorem \ref{theo:perimaire}: estimates of the perimeter and defect area}\label{sec:proofperimaire}

\noi

We use \eqref{eq:TheoremAthreyaNey} in order to get an asymptotic estimate for the perimeter of $\partial \cK_n$. Indeed, the length $\cL(\partial \cK_n)$ of $\partial \cK_n$ is $\la^{-n}$ times the number of edges which are commmon to $\cK_n$ and to any tile of $\cR_n$. In other words, $\cL(\partial \cK_n)$ is equal to $\la^{-n}\langle Z_n, v_0\rangle$ where $v_0$ is a deterministic vector such that for $1\le t\le \tmax$, the $t$th entry of $v_0$ is equal to the number of edges of a tile of type $t$ at generation $n$ which belong to the boundary of $\cK_n$. We recall that the definition of the type of a tile guarantees that the number of such edges only depends on the type. In particular, it implies that we have the almost sure convergence
\begin{equation*}
\frac{\cL(\partial \cK_n)}{(\rho_\M/\la)^n}\longrightarrow W \langle v,v_0\rangle.   
\end{equation*}
where we recall that $v$ is the unique unit eigenvector associated with $\rho_\M$ with positive entries. This proves the first part of \eqref{eq:limperimaire} with the choice $L_\infty=W \langle v,v_0\rangle.$

\pass In the same way, we can obtain an asymptotic estimate for the area of $\cR_n$. Indeed, recalling that the sequence of sets $\cR_n$, $n\ge 1$, is decreasing, we obtain
\begin{align*}
\cA(\cK_{n+1}\setminus \cK_n)
= \la^{-2(n+1)}\sum_{t=1}^{\tmax}\sum_{\ell=1}^{Z_n^t}N_\ell^t
\end{align*}
where the sequence $(N_\ell^t)_{\ell\ge 1}$ is constituted with i.i.d. random variables distributed as the number of tiles from $\cK_{n+1}\setminus \cK_n$ which belong to a fixed tile of $\cR_n^t$. Moreover, for each $1\le t\le \tmax$, $\rho_\M^{-n}Z_n^t$ converges almost surely to $W v^t$ by \eqref{eq:TheoremAthreyaNey}. Thanks to the law of large numbers, this implies that almost surely when $n\to \infty$,
\begin{equation}\label{eq:LGN1type}
\frac1{\rho_\M^n}\sum_{\ell=1}^{Z_n^t}N_\ell^t=(\rho_\M^{-n}Z_n^t)\frac1{Z_n^t}\sum_{\ell=1}^{Z_n^t}N_\ell^t
\longrightarrow (W v^t)\E[N_\ell^t].
\end{equation}
Summing \eqref{eq:LGN1type} over $1\le t \le \tmax$, we obtain that
\begin{equation*}
\lim_{n\to\infty} \frac{\cA(\cK_{n+1}\setminus \cK_n)}{(\rho_\M/\la^2)^n}=\la^{-2}\sum_{t=1}^{\tmax} (W v^t)\E[N_\ell^t].
\end{equation*}
This proves the second part of \eqref{eq:limperimaire} with the choice $A_\infty=\la^{-2}\sum_{t=1}^{\tmax}(W v^t)\E[N_\ell^t]$.

Finally, we deduce \eqref{eq:limairedefaut} by a summation over $n\ge m$:
\begin{equation*}
\cA(\cK_{\infty}\setminus \cK_m)=\sum_{n\ge m} \cA(\cK_{n+1}\setminus\cK_n).
\end{equation*}
It then remains to fix $\omega\in \Omega$ such that the convergence \eqref{eq:limperimaire} occurs, fix $\eps>0$ and take $m\ge0$ large enough such that for every $n\ge m$,
\begin{equation*}
(1-\eps)(\rho_\M/\la^2)^nA_\infty(\omega)\le \cA(\cK_{n+1}(\omega)\setminus \cK_n(\omega))\le (1+\eps)(\rho_\M/\la^2)^nA_\infty(\omega).
\end{equation*}
We sum both lower and upper bounds and obtain the required convergence \eqref{eq:limairedefaut}.

\addtocontents{toc}{\vspace{0.2cm}}%
\section{Von Koch-type examples and reduction of $\M$}\label{sec:exaintro}

In this section, we expose two short examples derivated from the classical von Koch construction. This requires to make explicit the reproduction matrix $\M$. It provides a new proof of well-known box dimension calculations and it also paves the way for a general technique of reduction of the reproduction matrix.

\subsection{A classical random version of the original von Koch model}\label{sec:vonkochalea}

\noi

The classical von Koch curve has given birth to several randomized generalizations. The most well-known is the so-called flip-flop, see \cite[Ex. 15.3, page 256]{falco03}, which consists in tossing a coin at each step and deciding accordingly if the newly added triangle should point up or down. In this section, we propose a different extension which is a small variant of our growth model associated with the triangular tessellation $\cT$ and the parameter $\la\ge 3$. To do so, we slightly modify the random rule described in Section \ref{sec:intro} and pictured in Figure \ref{fig:SquareSnake}. More precisely, we start with $\cK_0$ as a unique triangle from $\cT$. Then, along each edge of that triangle, we add with probability $p$ one triangle from $\la^{-1}\cT$ which is the $\lceil \la/2 \rceil$-th along the edge and we leave the edge unchanged with probability $(1-p)$, see Figure \ref{fig:KochAlea}. This gives birth to $\cK_1$. Given $\cK_n$ for $n\ge 1$, we iterate the procedure along each edge of $\partial \cK_n$ to construct $\cK_{n+1}$. As in Section \ref{sec:intro}, the sequence $(\cK_n)_{n\ge 0}$ is increasing and converges to a random compact set $\cK_\infty$. We assert that Theorem \ref{theo:main} still holds for this variant and Proposition \ref{prop:KochAlea} below provides the explicit calculation of the box and Hausdorff dimension of $\partial\cK_\infty$.

\begin{figure}[h!]
\centering
\includegraphics[scale=0.5]{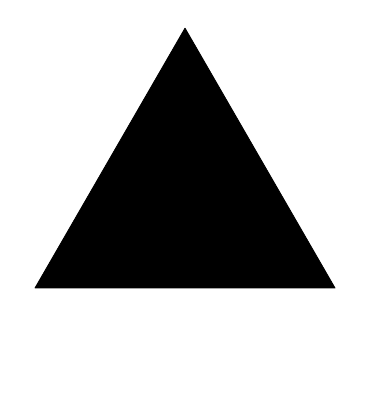} \hspace{-0.25cm}
\includegraphics[scale=0.5]{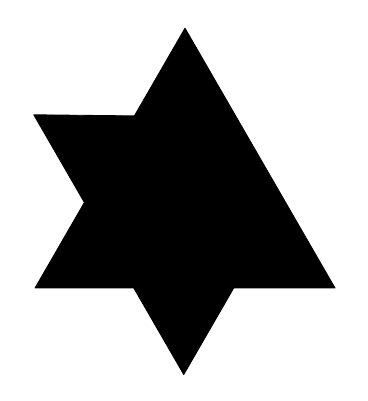} \hspace{-0.25cm}
\includegraphics[scale=0.5]{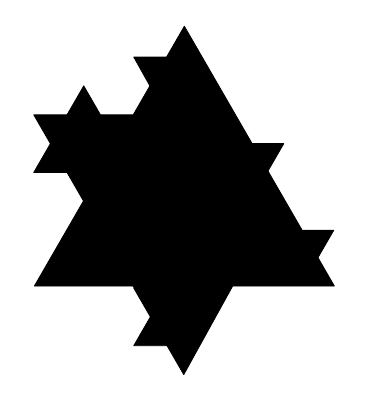} \hspace{-0.25cm}
\includegraphics[scale=0.5]{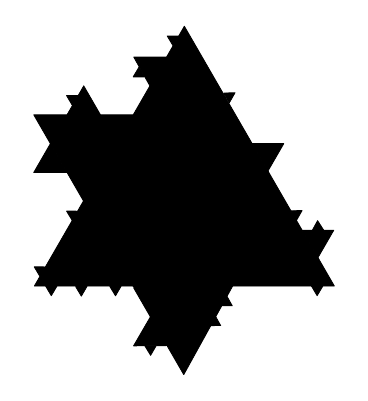}
\vspace{-0.2cm}\caption{A variant of the flip-flop von Koch random curve with $\lambda=3$ and $p=0.5$.}
\label{fig:KochAlea}
\end{figure}

\newpage
\begin{prop}\label{prop:KochAlea}
For any $\la\ge 3$ and $p\in [0,1]$, the limiting set $\cK_\infty$ satisfies
\begin{equation}\label{eq:dimrandomvonkoch}
\dimh(\partial\cK_\infty)=\dimbox(\partial\cK_\infty) = \frac{\log(\la+p)}{\log \la}.
\end{equation}
\end{prop}

\begin{proof}
We start by shortly explaining why the results from Sections \ref{sec:intro} and \ref{sec:main} extend to the model described above. Indeed, the coding of the tiles of the set $\cR_n$ into a multitype Galton-Watson tree is still relevant. Since the construction procedure does not interrupt almost surely, that branching process does not die with probability one. Moreover, from one generation to the next, this new model clearly produces less children than the initial one described in Section \ref{sec:intro}. Actually, the set $\cK_n$ generated by the new model is included in the former set $\cK_n$ for each $n\ge 0$. Consequently, the integrability conditions at \eqref{eq:condGW} are again clearly satisfied. Subsequently, the convergence of the underlying martingale occurs and the rest of the proofs goes along the exact same lines.

Our task is then to identify the reproduction matrix $\M$. In Figure \ref{fig:ModelKochALea}, we have represented the situation along the edge of one triangle which belongs to $\partial\cK_n$. We observe that there are $\la+2$ potential triangles along that edge which could belong to the set $\cR_n$, namely the $\la$ upward triangles from the next generation and the $2$ downward triangles which surround the central upward triangle. Moreover, each of these potential children would be also triangles with only one edge belonging to $\cK_{n+1}$, i.e. they are of type T1. Consequently, this means that there is only one type T1 and that $\M$ is a matrix of size $1\times 1$. We then only need to calculate the expected value of the number of triangles in $\cR_n$ along the edge, which is obviously the spectral radius $\rho_\M$ of $\M$. 

We study the case of each potential child with respect to its color. All upward triangles but the central one (in yellow in Figure \ref{fig:ModelKochALea}) belong to $\cR_n$ with probability $1$ while the central upward triangle (in green in Figure \ref{fig:ModelKochALea}) is in $\cR_n$ if and only if it has not been added in $\cK_{n+1}$, which happens with probability $1-p$. To the contrary, the two downward triangles (in blue in Figure \ref{fig:ModelKochALea}) are in $\cR_n$ as soon as the central upward triangle is not, which happens with probability $p$. 

\begin{figure}[h!]
\centering
\includegraphics[scale=1.25]{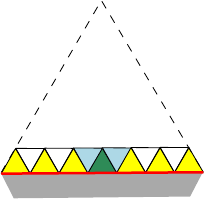}
\caption{The potential children of a T1-parent in the random version of the von Koch model for $\la=7$.}
\label{fig:ModelKochALea}
\end{figure}

\noi The expectation of the number of triangles in $\cR_n$ along the edge is then
\begin{equation*}
(\la-1)\times 1+ 1\times (1-p)+ 2\times p=\la+p.
\end{equation*}
Applying then Theorem \ref{theo:main}, we deduce that the box and Hausdorff dimension of $\partial \cK_\infty$ is $\frac{\log(\la+p)}{\log \la}$.
\end{proof}

\newpage
\subsection{A deterministic von Koch square model}\label{sec:squaredeter}

\noi

This specific model is based on the square tessellation with $\la\ge 3$. We slightly modify the random rule as described in Section \ref{sec:intro}. The Group 1 is now composed by the tiles $\cT_{n+1,\ell}$ which share exactly one edge with $\partial \cK_n$ but no end of an edge of $\partial \cK_n$. In other words, compared with the Group 1 from Section \ref{sec:intro}, we exclude the tiles which share two edges with $\partial \cK_n$. We keep the same Group 2 as in Section \ref{sec:intro} and choose $p_*=0$. The novelty in the current model is that the Bernoulli distributions associated with the tiles of the Group $1$ are deterministic but allowed to be different, i.e. they are characterized by a probability vector $\texttt{p}=(p_2,\ldots,p_{\la-1})\in\{0,1\}^{\la-2}$ where $p_i$ is the probability to add the $p_i$th square (the polygonal line $\partial \cK_n$ being oriented clockwise). This vector $\texttt{p}$ being fixed, it is convenient to see it as a deterministic geometric pattern applied at each edge of $\partial\cK_n$ with a suitable rescaling, see Figure \ref{fig:SquareDeter}.

\renewcommand{\arraystretch}{0.8}
\begin{figure}[h!]
\begin{center}
\begin{tabular}{cc}
\includegraphics[scale=0.75]{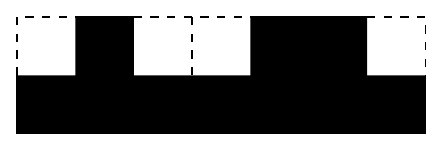} & 
\includegraphics[scale=0.75]{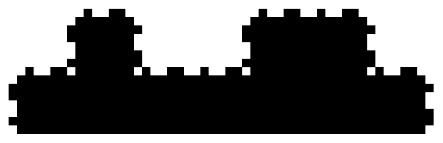} \\
\footnotesize{(a)} & \footnotesize{(b)}
\end{tabular}
\end{center}
\vspace{-.35cm}
\caption{Example of a deterministic pattern for the square tessellation with $\la=7$ and with the probability vector $\texttt{p}=(1,0,0,1,1)$. Here $\beta_\texttt{p}=2$. The figures (a) and (b) show the boundary $\partial\cK_1$ and $\partial\cK_2$ respectively.}
\label{fig:SquareDeter}
\end{figure}
\renewcommand{\arraystretch}{1.5}

Not only does Theorem \ref{theo:main} apply in this particular context but also the self-similarity properties of the set $\partial\cK_\infty$ provides a classical method for deriving the dimensions and the two quantities $\cL(\partial \cK_n)$ and $\cA(\cK_{n+1}\setminus \cK_n)$.

\begin{prop}\label{prop:SquareDeter}
Under the assumptions above the limiting set $\partial\cK_\infty$ is the union of $4$ self-similar sets with equal box and Hausdorff dimension given by 
\begin{equation}\label{eq:dimautosimcarre}
\dimbox(\partial\cK_\infty)=\dimh(\partial\cK_\infty) = \dfrac{\log(\lambda+2\beta_{\emph{\texttt{p}}})}{\log \lambda}
\end{equation}
where $\beta_\emph{\texttt{p}}$ denotes the number of blocks of consecutive $1$'s in $\emph{\texttt{p}}$. \\
Moreover,
\begin{equation}\label{eq:perimaireSquareDeter}
\cL(\cK_n)=4\bigg(\frac{\la+2\beta_{\emph{\texttt{p}}}}{\la}\bigg)^n
\text{ and }
\cA(\cK_{n+1}\setminus \cK_n)=\frac{4\sigma_{\emph{\texttt{p}}}}{\la^2}\bigg(\frac{\la+2\beta_{\emph{\texttt{p}}}}{\la^2}\bigg)^n
\end{equation}
where $\sigma_\emph{\texttt{p}}$ denotes the sum of coefficients of $\emph{\texttt{p}}$.
\end{prop} 

\noi\textit{Proof of \eqref{eq:dimautosimcarre} with the self-similarity}. The set $\partial \cK_\infty$ is the union of 4 isometric sets obtained from each of the $4$ sides of $\cK_0$ by our iterating procedure. It is enough to prove the results for one of these sets denoted by $\cC_\infty$. As for the classical von Koch set and its generalizations, the strategy consists in showing that $\cC_\infty$ is self-similar as the attractor of an Iterated Functions System (IFS) made up with several contracting similarity maps (see \cite[Chap. 9]{falco03} for a general account on IFS theory). Precisely, we observe that the IFS in question consists in a finite number $k$ of contracting similarity maps, all with the same ratio $\la^{-1}$ and that, because of the choice $p_*=0$, it satisfies the so-called \textit{Open Set Condition}. In these conditions it follows from \cite[Chap. 9, Th. 9.3]{falco03} that 
\begin{equation*}\label{eq:dimautosim}
\dimbox(\cC_\infty)=\dimh(\cC_\infty) = -\frac{\log k}{\log \la^{-1}}.
\end{equation*}
Finally, it remains to make the number $k$ of similarity maps explicit. We notice that $k$ is equal to the number of edges along the external boundary of the pattern above one edge of $\partial \cK_n$, see Figure \ref{fig:SquareDeter} (a). When going along the pattern, we do exactly $\la$ steps to the right, $\beta_p$ steps up and $\beta_p$ steps down, which leads to $k=\la+2\beta_\texttt{p}$.

\pass\textit{Proof of \eqref{eq:dimautosimcarre} with Theorem \ref{theo:main}.}
In the same way as in the proof of Proposition \ref{prop:KochAlea}, we assert that the conclusion of Theorem \ref{theo:main} still occurs with our current model. The only task that we need to fulfill consists in making explicit the reproduction matrix $\M$. 

We start by noticing that the squares belonging to $\cR_n$ are decomposed into three families: they are either the ones above the black squares of the pattern, or the ones which fill the voids of the chosen pattern or the ones at the two ends of the pattern. The first family is composed of squares mainly of type T1, i.e. squares which share with $\partial \cK_n$ exactly one edge, and possibly of type T2, i.e. with two edges in common with $\partial \cK_n$, when $p_2=p_{\la-1}=1$. The second family consists of squares either of type T1 when they are not at the ends of a void of the pattern or of type T2, when they are at one of the two ends of the void (and those ends are distinct) or else of type T3, i.e. with three edges in common with $\partial \cK_n$, when the void consists of exactly one square. The third family is trickier to deal with. Indeed, such a square can be of type T1, T2, T3 or even T4, i.e. with $4$ edges in common with $\partial \cK_n$. The possibility of being of type T3 (resp. of type T4) only exists if $p_2$ or $p_{\la-1}$ is equal to $1$ but not both at the same time (resp. if $p_2=p_{\la-1}=1$). The different types for the squares of $\cR_n$ are visible in Figure \ref{fig:SquareDeterChild} in the particular case $\la=7$, $\beta_\texttt{p}=2$, $\gamma_\texttt{p}=0$ and $p_2=p_{\la-1}=1$ and when the parent is either of type T1 (a) or of type T2 (b).

Let us denote by $\gamma_\texttt{p}$ the number of isolated $0$'s in the vector $\texttt{p}$, i.e. the number of voids in the pattern constituted of exactly one square. We deduce from the previous discussion that the reproduction matrix $\M$ is of size $2\times2$ if $p_2=p_{\la-1}=0$ and $\gamma_p=0$. Otherwise, it is a matrix of size $3\times3$, unless $p_2=p_{\la-1}=1$, in which case it is of size $4\times4$. We derive below a general formula for $\M$ as a function of $\la$, $\beta_p$ and $\gamma_p$ which holds in all cases. We explain for instance how to calculate the first line of $\M$, i.e. the cardinality of children of each type when the parent is of type T1. We first assert that the total number of children of a square of type T1 is $\la$, i.e. the number of squares of $\la^{-(n+1)}\cT$ along an edge of a square of $\la^{-n}\cT$, see Figure \ref{fig:SquareDeterChild} (a). Referring to the description above, we notice that the squares of the first family, i.e. above the squares of the chosen pattern, all belong to type T1. From the second family, only the squares which are inside a void of the pattern but not at one of the two ends of that void are of type T1. The rest of the squares of the second family belong to type T3 if they are in a one-square void or else of type T2. Regarding the third family, the square at the left end (resp. right end) is of type T1 if $p_2=0$ (resp. if $p_{\la-1}=0$) or of type T2 if $p_2=1$ (resp if $p_{\la-1}=1$). In other words, we first determine that the number of children of type T1 is $C^{1,1}=\la-2\beta_\texttt{p}+\gamma_\texttt{p}$, then that there are $C^{1,3}=\gamma_\texttt{p}$ children of type T3, no child of type T4, i.e. $C^{1,4}=0$,  and finally $C^{1,2}=\la-C^{1,1}-C^{1,3}-C^{1,4}=2\beta_\texttt{p}-2\gamma_\texttt{p}$  children of type T2. We proceed in a similar way to fill the remaining three lines of $\M$ and obtain that $\M= M_1 + M_2$ 
where
\begin{equation*}\label{eq:matrixsquarevonkoch1}
M_1=\left[\begin{matrix} 
\la-2\beta_\texttt{p}+\gamma_\texttt{p} & 2\beta_\texttt{p}-2\gamma_\texttt{p} & \gamma_\texttt{p} & 0 \\
2(\la-2\beta_\texttt{p}+\gamma_\texttt{p})-2 & 2(2\beta_\texttt{p}-2\gamma_\texttt{p}) & 2\gamma_\texttt{p} & 0 \\
3(\la-2\beta_\texttt{p}+\gamma_\texttt{p})-4& 3(2\beta_\texttt{p}-2\gamma_\texttt{p})
& 3\gamma_\texttt{p} & 0 \\
4(\la-2\beta_\texttt{p}+\gamma_\texttt{p})-8 & 4(2\beta_\texttt{p}-2\gamma_\texttt{p}) & 4\gamma_\texttt{p} & 0
\end{matrix}\right]
\end{equation*}
and
\begin{equation*}\label{eq:matrixsquarevonkoch2}
M_2=\left[\begin{mmatrix} 
0 & 0 & 0 & 0 \\
p_2+p_{\la-1} & -p_2-p_{\la-1} + (1-p_2)(1-p_{\la-1}) & p_2(1-p_{\la-1})+(1-p_2)p_{\la-1} & p_2p_{\la-1} \\
2p_2+2p_{\la-1} & -2p_2-2p_{\la-1} + 2(1-p_2)(1-p_{\la-1})
& 2p_2(1-p_{\la-1}) + 2(1-p_2)p_{\la-1} & 2p_2p_{\la-1} \\
4p_2+4p_{\la-1} & -4p_2-4p_{\la-1} + 4(1-p_2)(1-p_{\la-1})
& 4p_2(1-p_{\la-1}) + 4(1-p_2)p_{\la-1} & 4p_2p_{\la-1}
\end{mmatrix}\right]. 
\end{equation*}

\noi The eigenvalues of $\M$ are $0<1\leq 1+p_2p_{\la-1}<\la+2\beta_\texttt{p}=\rho_\M$. This fact combined with Theorem \ref{theo:main} shows \eqref{eq:dimautosimcarre}.

\renewcommand{\arraystretch}{0.8}
\begin{figure}[h!]
\begin{center}
\begin{tabular}{cc}
\includegraphics[scale=0.6]{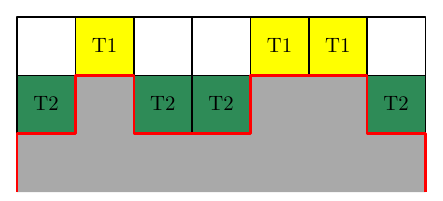} & 
\includegraphics[scale=0.6]{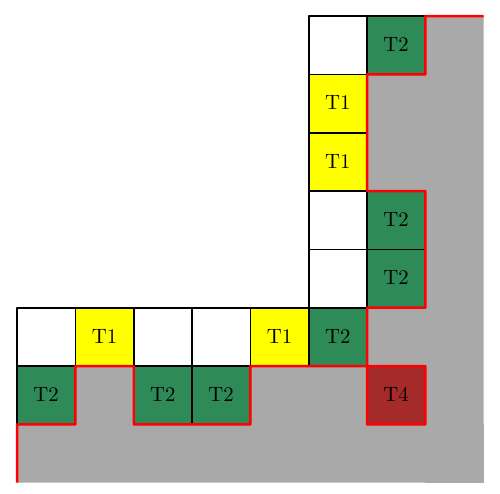} \\
\footnotesize{(a)} & \footnotesize{(b)}
\end{tabular}
\end{center}
\vspace{-.35cm}
\caption{Continuation of the example of the deterministic pattern for the square tessellation with $\la=7$ and $\texttt{p}=(1,0,0,1,1)$. The figures (a) and (b) show the children of a parent square of type T1 and of type T2 respectively.}
\label{fig:SquareDeterChild}
\end{figure}
\renewcommand{\arraystretch}{1.5}

The perimeter $\cL(\cK_n)$ and the area $\cA(\cK_{n+1}\setminus\cK_n)$ satisfy the asymptotic relations given in Theorem \ref{theo:perimaire} with $\rho_\M=\la+2\beta_\texttt{p}$. In fact, by a direct calculation, we can make explicit the constants involved. Indeed, noticing that $\la^n\cL(\cK_n)$ is the number of edges along the boundary of $\cK_n$, we observe that for any $n\ge 0$,
\begin{equation*}
\cL(\cK_{n+1})=\la^n\cL(\cK_n)\la^{-(n+1)}(\la+2\beta_\texttt{p})\mbox{ and }
\cA(\cK_{n+1}\setminus\cK_n)=\la^n\cL(\cK_n)\sigma_\texttt{p}\la^{-2(n+1)}.
\end{equation*}
Since $\cK_0$ is a square with perimeter $4$, this implies \eqref{eq:perimaireSquareDeter}.\hfill $\square$

\newpage
\subsection{Reducing the matrix $\M$}\label{sec:reductionmatrix}

\noi

In most cases as for instance the previous one, it turns out that some of the types appearing in the construction of the matrix $\M$ are redundant and this explains why $0$ appears in the spectrum. As a consequence, as soon as a non-trivial linear relation between the rows is unveiled, we propose to reduce the matrix by removing one type.

We recall that $\M$ is a square matrix of size $\tmax\times\tmax$ constituted with entries equal to $\E[C^{t,u}]$ and for $1\le t,u\le \tmax$, we denote by $\texttt{R}_t=(\E[C^{t,1}],\ldots,\E[C^{t,\tmax}])$ (resp. $\texttt{C}_u=(\E[C^{1,u}],\ldots,\E[C^{\tmax,u}])^*$) the $t$th row (resp. $u$th column) of $\M$. In particular, $\texttt{R}_t$ is the mean contribution of a parent of type $t$ while $\texttt{C}_u$ is the mean number of children of type $u$ per parent.   

\begin{lem}\label{lem:reductionmatrix}
Let us assume the existence of $\alpha_1,\ldots,\alpha_{\tmax-1}\in\R$ such that the linear combination $\texttt{R}_t=\sum_{t=1}^{\tmax-1}\alpha_t\texttt{R}_t$ is satisfied. Then $\M$ has same spectral radius as ${\M}'$ which is a square matrix of size $(\tmax-1)\times(\tmax-1)$ obtained by replacing each column $\texttt{C}_u$, $1\le u\le \tmax-1$, by $\texttt{C}_u+\alpha_u\texttt{C}_{\tmax}$, then removing from $\M$ the line $\texttt{R}_{\tmax}$ and the column $\texttt{C}_{\tmax}$.
\end{lem}

\begin{proof}
We show that the matrix $\M'$ defined above has the same eigenvalues as $\M$ minus one $0$. Proceeding with the operation in $\M$ which consists in substracting to $\texttt{R}_{\tmax}$ the combination $\sum_{t=1}^{\tmax-1}\alpha_t\texttt{R}_t$, we obtain a new matrix with a zero row vector and equal to $\widetilde{\M}=\prod_{i=t}^{\tmax-1}L_{\tmax,t}(-\alpha_t)\M$ where $L_{t,u}(x)$ is the elementary matrix which differs from the identity matrix by a coefficient $x$ in the $(t,u)$ position. Then, for each parent of type $t$, $1\le u\le \tmax-1$, we need to remove the mean number of children of type $\tmax$ and replace it by the equivalent quantity of children of types $1\le u\le \tmax-1$, namely $\alpha_u$ children of type $u$ instead of each child of type $\tmax$. This means proceeding with the operation of adding $\alpha_u\texttt{C}_{\tmax}$ to $\texttt{C}_u$. This leads to a new matrix $\widetilde{\M}$ equal to
\begin{equation}
\widetilde{\M}=\bigg(\prod_{t=1}^{\tmax-1}L_{\tmax,t}(-\alpha_t)\bigg)\M \bigg( \prod_{t=1}^{\tmax-1}L_{\tmax,t}(\alpha_t)\bigg). 
\end{equation}
Since $\prod_{i=t}^{\tmax-1}L_{\tmax,t}(\alpha_t)=\left(\prod_{t=1}^{\tmax-1}L_{\tmax,t}(-\alpha_t)\right)^{-1}$, we deduce that $\widetilde{\M}$ has same spectrum as $\M$ and this spectrum includes $0$ because $\widetilde{\M}$ has a zero row vector in place of the row vector $\texttt{R}_{\tmax}$ of $\M$. Consequently, $\widetilde{\M}$ and its submatrix $\M'$ obtained by removing its last row and last column have same spectral radius.
\end{proof}

An important consequence of Lemma \ref{lem:reductionmatrix} is that we can drastically reduce the size of the matrix $\M$ in all of the examples described in the next sections. Indeed, we can discriminate the types $t$ which are kept, called \textit{survivor types} and the types which are removed, called the \textit{phantom types}. Typically, the survivor types are those from the tiles $\cT_{n,\ell}$ such that the intersection $\partial\cT_{n,\ell}\cap \partial\cK_n$ is smaller (one vertex, one edge, two edges) whereas each phantom type satisfies that its corresponding row vector in $\M$ is a linear combination of the rows of $\M$ corresponding to the survivor types. For example, let us denote by $\texttt{R}_t$, $t=1,2,3$, the row associated with the type T$t$, $t=1,2,3$, where we denote by T$t$ the type of a tile $\cT_{n,\ell}$ such that its intersection with $\partial\cK_n$ consists in $t$ consecutive edges. Similarly, when the intersection of a tile $\cT_{n,\ell}$ with $\partial\cK_n$ is reduced to one vertex, we call T$0$ its type. We observe that
\begin{equation*}
\texttt{R}_3=2\texttt{R}_2-\texttt{R}_1,
\end{equation*}
which means that when we choose T1 and T2 as survivor types, type T3 is then a phantom type. Besides, we notice that a tile with type T2 cannot be replaced by two tiles with type T1. The complete list of reductions rules will be described for each of the three tessellations in Tables \ref{tab:ReductionRulesHexa} and \ref{tab:ReductionRulesSquare}, and in the beginning of Section \ref{sec:annexetriangle} respectively. In particular, these rules preserve the number of edges common to the tile and $\partial \cK_n$.
\enlargethispage*{0.5cm}

\addtocontents{toc}{\vspace{0.2cm}}%
\section{The method in action for the three tessellations}\label{sec:Model}

In this section, we apply Theorems \ref{theo:main} and \ref{theo:perimaire} in several concrete examples, i.e. when the underlying tessellation is either hexagonal, or square or triangular. This requires to identify the useful types related to the underlying multitype Galton-Watson tree and then to make explicit the reproduction matrix given at \eqref{eq:defMatrixM} after a suitable application of the reduction rules described in Lemma \ref{lem:reductionmatrix}. We obtain a matrix of size $2\times2$ in all cases, save for the case when the tessellation is square and the probability $p_*$ is equal to $1$ which leads to a matrix of size $3\times3$. The square case is also the only one which gives birth to a disconnected set $\partial \cK_\infty$, which justifies the asymptotics of the number of holes given in Proposition \ref{prop:squareholes}. 

Most of the calculation procedure relies on the decomposition of the set of potential children into smaller classes whoses cardinality and mean contribution to the offspring can be made explicit. All the details are summarized in precise tables available in the Appendix. Incidentally, we will place more emphasis on the hexagonal case as the other two go along similar lines.

\subsection{The case of the hexagonal tessellation}\label{sec:ModelHexa}

\noi

In this section, $\cT$ is the hexagonal tessellation with tiles of edge length $1$ and such that the origin is at the center of one hexagon from $\cT$. This particular rule is more natural since it guarantees the invariance of the model with respect to a rotation with center at the origin and angle $\frac{\pi}{3}$ and reflects in a more natural way the properties of a Poisson-Voronoi tessellation.

\renewcommand{\arraystretch}{0.8}
\begin{figure}[h!]
\begin{center}
\begin{tabular}{ccc}
\includegraphics[scale=1]{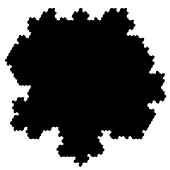} &
\includegraphics[scale=1.35]{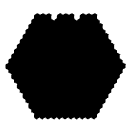} &
\includegraphics[scale=0.8]{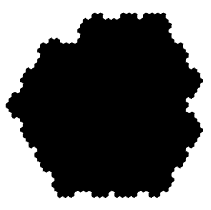} \\
\footnotesize{(a)} & \footnotesize{(b)} & \footnotesize{(c)}
\end{tabular}
\end{center}
\vspace{-.35cm}
\caption{Example of the limit set $\cK_\infty$ obtained with the hexagonal tessellation with $p*=0$, $p=0.5$ and, from left to right, $\la=3$, $\la=4$ and $\la=5$.}
\label{fig:HexagonModelSimu}
\end{figure}
\renewcommand{\arraystretch}{1.5}

In Propositions \ref{prop:Model0HexaCenter} and \ref{prop:Model1HexaCenter} below, we make explicit both the reduced reproduction matrix and its spectral radius in both cases $p_*=0$ and $p_*=1$. Interestingly, the calculation depends on the remainder of the Euclidean division of $\la$ by $3$. This arithmetical discrimination is due to the embedding of the consecutive grids in the particular context of the hexagonal grid. Surprisingly, in both Propositions, the spectral radius does not depend on $p$ when $\la=3\la'+2$. We then use Theorem \ref{theo:main} to calculate the dimension of $\partial \cK_\infty$ as a function of $p\in [0,1]$ and compare in two cases the graph of that function with the numerical estimate obtained by simulation, see Figure \ref{fig:HexaModelSimuDim}.

\newpage
\begin{prop}[Case $p_*=0$]\label{prop:Model0HexaCenter}
For the hexagonal model described above with parameters $(\la,p)$ where $\la=3\la'+r$ with $\la'\ge 1$ and $r\in \{0,1,2\}$, $p\in [0,1]$ and with the choice $p_*=0$, the reduced reproduction matrix is 
\begin{equation*}
\M'_0=\left[\begin{matrix} 1 & 2\la'
\\ 0 & 4\la'-p^2 \end{matrix}\right],\,
 \M'_1=\left[\begin{matrix} 1 & 2\la' 
\\ 0 & 4\la'+1-p \end{matrix}\right],\,
 \M'_2=\left[\begin{matrix} 1 & 2\la'+1
\\ 0 & 4\la'+2 \end{matrix}\right], 
\end{equation*}
where $\M'_0$ (resp. $\M'_1$, $\M'_2$) denotes the reduced reproduction matrix when $r=0$ (resp. $r=1$, $r=2$).

\pass Moreover, the spectral radius of the matrix $\M$ is given by
\begin{equation*}
\rho_\M = \left\{
\begin{array}{ll}
4\la'-p^2 & \text{if $r=0$} \\
4\la'+1-p & \text{if $r=1$} \\
4\la'+2 & \text{if $r=2$}
\end{array}
\right..
\end{equation*}
\end{prop}
\begin{proof} The strategy consists in reducing the number of types so that the reproduction matrix $\M$ is a matrix of size $2\times2$ (Step 1). We then prepare the calculation of each entry by considering the set all potential children of a parent with given type and split it into subsets of hexagons with the same contribution in mean (Step 2). The outcome of Step 2 is summarized in Tables \ref{tab:Model0HexaR0Child}, \ref{tab:Model0HexaR1Child} and \ref{tab:Model0HexaR2Child} in the cases $r=0$, $r=1$ and $r=2$ respectively. We finally proceed with the explicit calculation of $\M$ and subsequently of $\rho_\M$ (Step 3). 

\pass{\bf\it Step 1: Reduction of $\M$.}
We first claim that the reduction of the number of types, see Lemma \ref{lem:reductionmatrix}, allows us to deal with only two survivor types: the type T1 of an hexagon with exactly one edge belonging to $\cK_n$ and the type T2 of an hexagon with exactly two consecutive edges belonging to $\partial\cK_n$. As a consequence of Lemma \ref{lem:reductionmatrix}, the reduction rules allow us to remove the phantom types T3, T4, T5 and T6, see Table \ref{tab:ReductionRulesHexa}. Notice that the two rightmost types in Figure \ref{fig:HexaAllTheTypes} only appear in the particular case $\la=3$.

\renewcommand{\arraystretch}{1.5}
\setlength{\arrayrulewidth}{0.1pt}
\begin{table}[h!]
\centering
\begin{tabular}{|c|c|c|c|c|c|}
\cline{3-6}
\multicolumn{1}{c}{} & & \multicolumn{4}{c|}{Phantom types} \\
\cline{3-6}
\multicolumn{1}{c}{} & & \rule{0cm}{0.75cm} \raisebox{0.5\height}{T3}\hspace{0.2cm} \raisebox{-.02\height}{\includegraphics[scale=0.75]{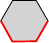}} 
& \rule{0cm}{0.75cm} \raisebox{0.5\height}{T4}\hspace{0.2cm} \raisebox{-.02\height}{\includegraphics[scale=0.75]{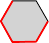}}
& \rule{0cm}{0.75cm} \raisebox{0.5\height}{T5}\hspace{0.2cm} \raisebox{-.02\height}{\includegraphics[scale=0.75]{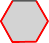}}
& \rule{0cm}{0.75cm} \raisebox{0.5\height}{T6}\hspace{0.2cm} \raisebox{-.02\height}{\includegraphics[scale=0.75]{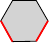}}\\
\hline
\multirow{3}{.7cm}{\rotatebox{90}{\hspace{0.5cm}Survivor}\;\rotatebox{90}{\hspace{0.75cm}types}}
& \rule{0cm}{0.75cm} \raisebox{0.5\height}{T1}\hspace{0.2cm} \raisebox{-.02\height}{\includegraphics[scale=0.75]{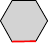}}
& \raisebox{0.5\height}{$-1$} & \raisebox{0.5\height}{$-2$} & \raisebox{0.5\height}{$-3$} & \raisebox{0.5\height}{$2$} \\
\cline{2-6}
& \rule{0cm}{0.75cm} \raisebox{0.5\height}{T2}\hspace{0.2cm} \raisebox{-.02\height}{\includegraphics[scale=0.75]{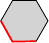}}
& \raisebox{0.5\height}{$2$} & \raisebox{0.5\height}{$3$} & \raisebox{0.5\height}{$4$} & \raisebox{0.5\height}{$0$} \\
\hline
\end{tabular}
\pass\caption{The explicit rules for the reduction of the four different types T3, T4, T5 and T6. Each phantom type may be expressed as a linear combination of the two survivor types T1 and T2 only.}
\label{tab:ReductionRulesHexa}
\end{table}
\pass After application of the reduction rule, the new matrix which shares the same spectral radius as $\M$ is a matrix of size $2\times 2$ called $\M'$.

\newpage
\pass{\bf\it Step 2: Decomposition of the set of potential children of each parent.}
Each parent $\cT_{n,\ell}$ of type T1 or T2 belongs to $\cR_n$ and gives birth to children whose union is the intersection of the set $\cR_{n+1}$ with $\cT_{n,\ell}$. We need to determine the cardinality of the children of both types T1 and T2. To do so, we adopt an inverse procedure, i.e. we start with the whole set of potential children and study the mean contribution of each potential child separately. By {\it mean contribution}, we mean the expected number of children of type T1 and of type T2 which are induced by the potential child. The potential children lie in the first three lines of $\cT_{n+1}$ inside $\cT_{n,\ell}$ \textit{above} the edges (one or two) of $\cT_{n,\ell}$ which lie in $\cK_n$. The contribution of a potential child only depends on the local geometry around it so that we can partition the set of potential children into groups which are labelled with different colors, see Figure \ref{fig:Model0HexaR0AllTheChild} below and Table \ref{tab:HexagonChild}. This partition depends on the remainder $r$ of $\la$ when divided by $3$. This leads us to consider the three different cases $r=0$, $r=1$ and $r=2$. When $r$ is fixed, the cardinality of each group then is an affine function of $\la$. 

\renewcommand{\arraystretch}{0.8}
\begin{figure}[h!]
\begin{center}
\begin{tabular}{cc}
\includegraphics[scale=0.75]{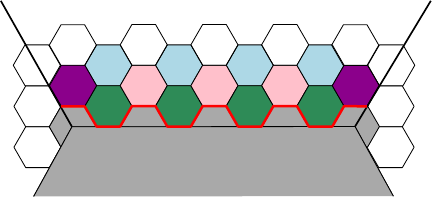} &
\includegraphics[scale=0.75]{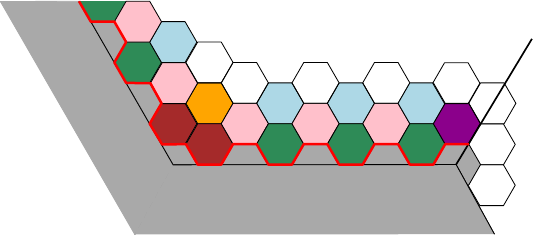} \\
\footnotesize{(a)} & \footnotesize{(b)}
\end{tabular}
\end{center}
\vspace{-.35cm}
\caption{The potential children of a T1-parent (a) and a T2-parent (b) for the case of the hexagonal tessellation when $p_*=0$ and $r=0$ (here $\la=12$).}
\label{fig:Model0HexaR0AllTheChild}
\end{figure}
\renewcommand{\arraystretch}{1.5}

\pass{\bf \it Step 3: Calculation of the reduced matrix coefficients.}
As soon as the color of a potential child, i.e. the group it belongs to, is fixed, the mean contribution of that potential child is a deterministic function of one or two independent Bernoulli variables with parameter $p$ that we can make explicit. Tables \ref{tab:Model0HexaR0Child}, \ref{tab:Model1HexaR1Child} and \ref{tab:Model0HexaR2Child} summarize the groups, the cardinality of each group and the contribution of each member of a fixed group. Let us notice that in order to keep the set of children nested in one parent tile, we allow the occurrence of half-hexagons. For instance, in the case $r=0$, let us explain how to determine the line in Table \ref{tab:Model0HexaR0Child} related to a pink potential child. 

\renewcommand{\arraystretch}{0.8}
\begin{figure}
\begin{center}
\begin{tabular}{cccc}
\includegraphics[scale=1]{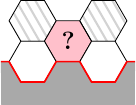} & 
\includegraphics[scale=1]{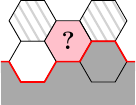} & 
\includegraphics[scale=1]{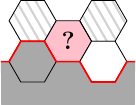} &
\includegraphics[scale=1]{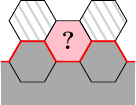} \\
\footnotesize{(a)} & \footnotesize{(b)} & \footnotesize{(c)} & \footnotesize{(d)}
\end{tabular}
\end{center}
\vspace{-.35cm}
\caption{Status of a pink potential child of a parent of type T1 in the case $r=0$ for the model with $p_*=0$. It only depends on the realization of the two Bernoulli variables associated with the tiles under the pink tile.}
\label{fig:Model0HexaR0ChildT1Pink}
\end{figure}
\renewcommand{\arraystretch}{1.5}

\pass Such a potential child lies above a half-hexagon which belongs to $\cK_{n+1}$ and it is surrounded by two hexagons from Group $1$, i.e. they belong to $\cK_{n+1}$ with probability $p$ independently from each other. The type of the pink potential child is decided as soon as the couple $B_p^{(2)}$ of these two independent Bernoulli variables is realized:
\begin{itemize}[parsep=-0.15cm,itemsep=0.2cm,topsep=0.2cm,wide=0.15cm,leftmargin=0.65cm]
\item[$-$] when $B_p^{(2)}=(0,0)$, which happens with probability $(1-p)^2$, only the edge at the bottom of the pink hexagon belongs to $\cK_{n+1}$. Consequently, the pink hexagon is a child of type $T1$ and its contribution is $(1,0)$, see Figure \ref{fig:Model0HexaR0ChildT1Pink} (a).
\item[$-$] when $B_p^{(2)}=(0,1)$ (or $(1,0)$), which happens with probability $p(1-p)$, exactly two consecutive edges of the pink hexagon belong to $\cK_{n+1}$. Consequently, the pink hexagon is a child of type $T2$ and its contribution is $(0,1)$, see Figure \ref{fig:Model0HexaR0ChildT1Pink} (b) and (c).
\item[$-$] when $B_p^{(2)}=(1,1)$, which happens with probability $p^2$, exactly three consecutive edges of the pink hexagon belong to $\cK_{n+1}$. Consequently, the pink hexagon is a child of type $T3$ and its contribution is $(-1,2)$ thanks to the reduction rule listed in Table \ref{tab:ReductionRulesHexa}, see Figure \ref{fig:Model0HexaR0ChildT1Pink} (d).
\end{itemize}

Once the type of the parent is fixed, either T1 or T2, we then have to sum up over the groups the product of the cardinality of the group by its individual mean contribution. This provides the first then the second line of the matrix $\M'$. For instance, in the case $r=0$, denoting by $\cN_\texttt{c}^{(0)}=\cN_\texttt{c}^{(0)}(\mbox{T1})$ the cardinality of potential children of a parent T1 of color \texttt{c} belonging to $\{\texttt{green}, \texttt{pink}, \texttt{purple}, \texttt{blue}, \texttt{brown}, \texttt{orange}\}$, we obtain that the first row $\texttt{R}_1$ of $\M'$ is obtained by the formula 
\begin{align*}
\texttt{R}_1
& = \sum_{\substack{c\,\in \{\texttt{green},\\\texttt{purple}, \texttt{blue}\}}} \cN_{\texttt{c}}^{(0)} \sum_{\eps\in
\{0,1\}}\P(B_p^{(1)}=\eps)F_{\texttt{c}}^{(0)}(\eps)
+\sum_{\substack{c\,\in \{\texttt{pink}, \\ \texttt{brown}, \texttt{orange}\}}}\cN_{\texttt{c}}^{(0)} \sum_{\eps\in \{0,1\}^2}\P(B_p^{(2)}=\eps)F_{\texttt{c}}^{(0)}(\eps)\\
& = \big(1-p,2\la'\big).  
\end{align*}
The proof is then complete.
\end{proof}

In Figure \ref{fig:HexaModelSimuDim}, we represent the dimension as a function of $p$ for fixed $\la$. We call it theoretical and compare it with the numerical dimension set obtained by counting the number of hexagons in the economic covering set $\cR_n$.

\renewcommand{\arraystretch}{0.8}
\begin{figure}[h!]
\begin{center}
\begin{tabular}{cc}
\includegraphics[scale=0.6]{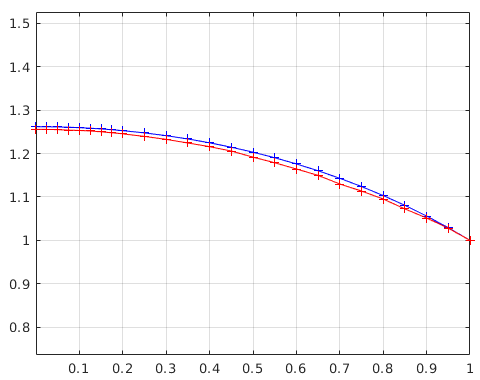} &
\includegraphics[scale=0.6]{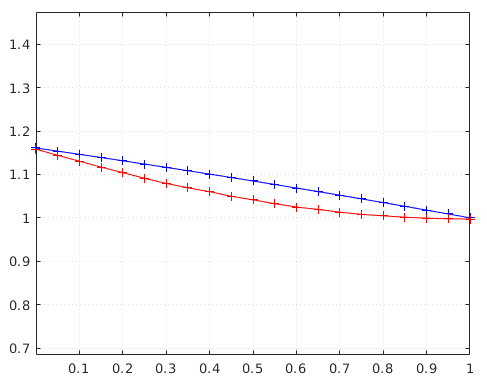} \\
\footnotesize{(a)} & \footnotesize{(b)}
\end{tabular}
\end{center}
\vspace{-.35cm}
\caption{Comparison between the graphs of the theoretical (in blue) and numerical (in red) dimension of $\partial \cK_\infty$ as a function of $p\in [0,1]$ in the hexagonal case, when $p_*=0$, for (a) $\la=3$ and (b) $\la=4$.}
\label{fig:HexaModelSimuDim}
\end{figure} 
\renewcommand{\arraystretch}{1.5}

\newpage
We now focus on the case $p_*=1$. Because of the geometry of the hexagonal tessellation, the initialization of $(Z_n)_{n\ge 0}$ stays the same, i.e. $Z_0$ consists only of one hexagon of type T1. Actually, we can show that as in the case $p_*=0$, only the two types T1 and T2 are needed to construct the reduced reproduction matrix. The calculation of $\rho_\M$ is provided in Proposition \ref{prop:Model1HexaCenter} below.

\begin{prop}[Case $p_*=1$]\label{prop:Model1HexaCenter}
For the hexagonal model described above with parameters $(\la,p)$ where $\la=3\la'+r$ with $\la'\ge 1$ and $r\in \{0,1,2\}$, $p\in [0,1]$ and with the choice $p_*=1$, the reduced reproduction matrix is 
\begin{equation*}
\M'_{0}=\left[\begin{matrix} 1 & 2\la'
\\ 0 & 4\la'-p^2 \end{matrix}\right],\,
\M'_{1}=\left[\begin{matrix} 1 & 2 \la'+1 \\ 0 & 4\la'+2-p \end{matrix}\right],\,
\M'_{2}=\left[\begin{matrix} 1 & 2\la'+2 \\ 0 & 4\la'+3 \end{matrix}\right],
\end{equation*}
where $\M'_{0}$ (resp. $\M'_{1}$, $\M'_{2}$) denotes the reduced reproduction matrix when $r=0$ (resp. $r=1$, $r=2$).

\pass Moreover, the spectral radius of the matrix $\M$ is given by
\begin{equation*}
\rho_\M = \left\{
\begin{array}{ll}
4\la' - p^2 & \text{if $r=0$} \\
4\la'+ 2-p & \text{if $r=1$} \\ 
4\la'+ 3 & \text{if $r=2$}
\end{array}
\right..
\end{equation*}
\end{prop}
\begin{proof}
We claim that only the cases in yellow and in purple need to be updated. When $r=1$, a yellow parent contributes by $(1,\frac12)$ instead of $(0,1)$ and a purple parent contributes by $((1-p)(0,1)+p(-1,2))$ instead of $((1-p)(1,0)+p(0,1))$. This modifies the matrix by adding $1$ to each entry in the second column. In the same way, when $r=2$, the matrix is also modified by adding $1$ to each entry of the second column. We refer the reader to the Tables \ref{tab:Model1HexaR1Child} and \ref{tab:Model1HexaR2Child}.
\end{proof}


\pass{\bf Remark.} We conclude the section with a few statements in several deterministic cases.
\begin{itemize}[parsep=-0.15cm,itemsep=0.25cm,topsep=0.2cm,wide=0.15cm,leftmargin=0.65cm]
\item[$-$] $\la=2$. Our assumption on Group $2$ in the description of the model in Section \ref{sec:intro} naturally excludes the case $\la=2$ from our investigation in the cases of the square and triangular tessellations. But in the hexagonal model, Propositions \ref{prop:Model0HexaCenter} and \ref{prop:Model1HexaCenter} extend to the case $\la=2$ where at each step, a half-hexagon is simply added to each side. The fractal dimensions are then both equal to $1$ so $\partial \cK_\infty$ is not a fractal set. Indeed, we can easily check that $\partial \cK_\infty$ is the dual hexagon with size $2$ rotated with respect to the initial hexagon $\cK_0$. 
\item[$-$] $\la\ge 3$. Assuming now that no hexagons are added at each step, i.e. $p=p_*=0$, then the limiting set $\partial\cK_\infty$ has fractal dimensions given by 
\begin{equation*}\label{eq:dimautosimhexa}
\dimh(\partial\cK_\infty)=\dimbox(\partial\cK_\infty) = \dfrac{\log(4\la'+r)}{\log(3\la'+r)} .
\end{equation*}
\end{itemize}
It may be proved that when $\la\ne 3\la'+2$, the set $\partial\cK_\infty$ is self-similar. It is the attractor of an affine ISF made up with exactly $\la+\lfloor\frac{\la}{3}\rfloor$ maps each contracting with ratio $\la$. In addition with the usual Open Set Condition, it provides an alternative proof for the value of the fractal dimension.

\subsection{The case of the square tessellation}\label{sec:ModelSquare}

\noi

In this section, we consider the case when $\cT$ is a square tessellation with the particular choice $p_i=p\in [0,1]$ and $p_*\in \{0,1\}$. We fix the origin at one of the vertices of a square. The alternative choice would have been to put the origin at the center of a square but in that case, when $n$ is even, we would get at every step $\cK_{n+1}=\cK_n^{\bullet}$ which is a deterministic square and when $n$ is odd, the model is the same as choosing the homothety center at one of the vertices of a square.

Let us fix a scaling factor $\la\ge3$ and a parameter $p\in[0,1]$. A square from $\la^{-n}\cT$ is added to $\cK_n$ accordingly to its position along an edge of $\partial\cK_n$ as shown in Figure \ref{fig:SquareSnake} below. As in the previous section on the hexagonal tessellation, we solve separately the two cases $p_*=0$ and $p_*=1$.

A notable difference with the hexagonal tessellation concerns the squares of $\la^{-n}\cT$ which intersect $\cK_n$ by only a vertex. Indeed, they do not belong to $\cK_{n+1}$ but when $p_*=1$, they have a non-empty intersection with $\partial \cK_{n+2}$ and $\partial\cK_m$ for any $m\ge n+2$, including $m=\infty$. Consequently, they belong to the set $\cR_{n+1}$ defined at \eqref{def:cover} and will be included as a special extra type in the set of children of our Galton-Watson coding. This means that the size of the reproduction matrix $\M$ will be $2\times 2$ when $p_*=0$ and $3\times 3$ when $p_*=1$. Alternatively, we could have modified the model so that these particular squares would be included to $\cK_{n+1}$ with probability $p_*\in \{0,1\}$. We only deal with the case $p_*=0$ below and assert that the case $p_*=1$ could be treated in a similar way.

Due to the geometrical differences between the hexagonal and square tessellations, we could show that the limiting set is simply connected in the hexagonal model but is not in the square model with probability $1$.

\renewcommand{\arraystretch}{0.8}
\begin{figure}[h!]
\centering
\begin{tabular}{cc}
\includegraphics[scale=0.35]{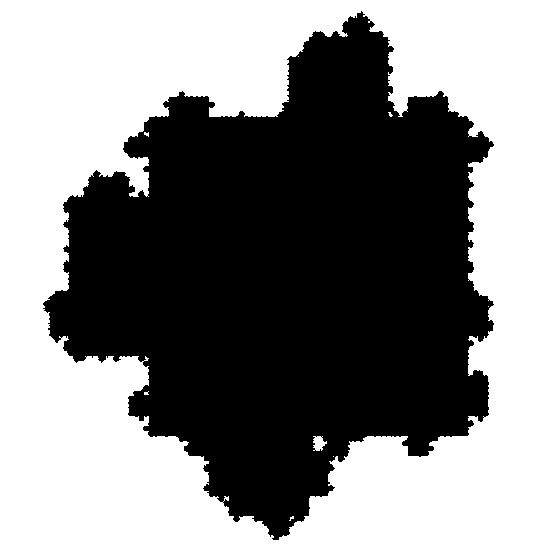}&
\includegraphics[scale=0.35]{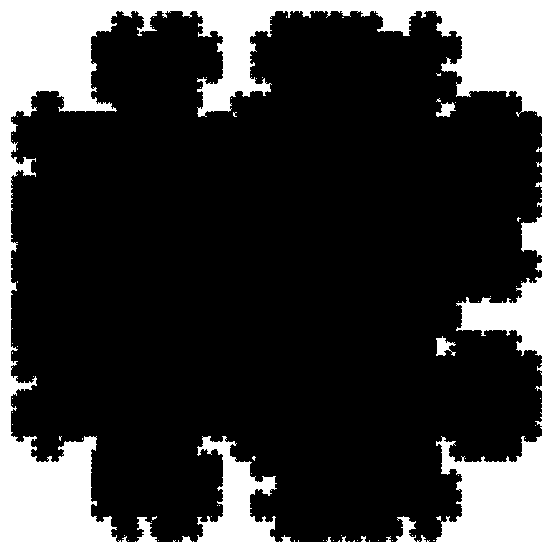}\\
\footnotesize{(a)}&\footnotesize{(b)}
\end{tabular}
\vspace*{-.2cm}
\caption{Example of the limit set $\cK_\infty$ obtained with the square tessellation.\\
(a) $p_*=0$, $\la=4$, $p=0.5$. The fractal dimension of the boundary is $\dimh(\partial\cK_\infty)=1.1822$. (b) $p_*=1$, $\la=4$, $p=0.5$. The fractal dimension of the boundary is $\dimh(\partial\cK_\infty)=1.2412$.}
\label{fig:SquareModelSimu}
\end{figure} 
\renewcommand{\arraystretch}{1.5}

Proposition \ref{prop:SquareModel} below provides the reduced reproduction matrix in both cases $p_*=0$ and $p_*=1$. In particular, when $p_*=0$, the matrix has size $2\times2$ with corresponding types T1 and T2, which makes it possible to calculate easily its spectral radius. When $p_*=1$, as discussed in Section \ref{sec:galton}, the matrix has size $3\times3$ with  the 3 types T1, T2 and T0, and the initialization consists in taking $Z_0=(1,0,1)$. In theory, there is a closed formula for its spectral radius as a function of $\la$ and $p$. Only in practice, it turns out to be very heavy, which is the reason why we have opted to omit it. 

\begin{prop}\label{prop:SquareModel}
For the square model with parameters $\la\ge3$, $p\in[0,1]$ and $p_*=0$, the reduced matrix $\M'$ is 
\begin{equation*} 
\M'=(\la-3)M_1'(p)+M_2'(p)
\end{equation*}
where
\begin{equation*}
M_1'(p)  =\left[
\begin{matrix} 1-2p(1-p) & 2p(1-p) \\ 2-4p(1-p)) & 4p(1-p) \end{matrix}
\right]\mbox{ and }
M_2'(p)  =\left[
\begin{matrix} 3-2p & 2p \\ 2(p+2)(1-p)^2 & 1+p(5-2p-p^2) \end{matrix}
\right].
\end{equation*}

\pass Moreover, the spectral radius of the matrix $\M$ is given by
\begin{align*}
\rho_\M & = \mfrac12 \Big(\la Q_1(p)+Q_2(p)+\sqrt{\la^2 Q_3(p)+\la Q_4(p)+Q_5(p)}
\Big),
 \end{align*}
where $Q_1(p)= -2 p^2 + 2 p + 1$, $Q_2(p) = -p^3 + 4 p^2 - 3 p + 1$, $Q_3(p) = 4p^4-8p^3+4p+1$, $Q_4(p) = -4p^5+4p^4+22p^3-14p-2$ and $Q_5(p) = p^6 + 16 p^5 - 42 p^4 - 18 p^3 + 9 p^2 + 10 p + 1$.

\pass For the square model with parameters $\la\ge3$, $p\in[0,1]$ and $p_*=1$, the reduced matrix $\M'$ is 
\begin{equation*}
\hspace{-2cm}\M'=(\la-3)M_1'(p)+M_2'(p)
\end{equation*}
where
\begin{align*}
M_1'(p) & =\left[
\begin{smallmatrix}
1-2p(1-p) & 2p(1-p) & 2p(1-p) \\ 2-4p(1-p) & 4p(1-p) & 4p(1-p) \\ 0 & 0 & 0
\end{smallmatrix}
\right]\mbox{ and } 
M_2'(p)  =\left[
\begin{smallmatrix}
1+2p & 2(1-p) & 2(1-p) \\ 2-2p(1-p^2) & 3 + p(1-2p-p^2) & (1-p)(2+3p-3p^2) \\
0 & 1 & 2
\end{smallmatrix}
\right].
\end{align*}
\end{prop}

\begin{proof} We adapt the strategy which has already been put into practice in the case of the hexagonal tessellation, i.e.: 
\begin{itemize}[parsep=-0.15cm,itemsep=0.2cm,topsep=0.2cm,wide=0.15cm,leftmargin=0.65cm]
\item[$-$] we list all possible types which are bound to appear in the reproduction matrix as well as the reduction rules which allow us to consider a matrix of size $2\times 2$ (resp. of size $3\times 3$ in the case $p_*=1$) instead (Step 1),
\item[$-$] we identify the set of potential children of a square parent and decompose it into a suitable colored partition such that each color corresponds to a particular contribution in mean (Step 2),
\item[$-$] we calculate explicitly the reduced reduction matrix and its spectral radius (Step 3).
\end{itemize}

\noi{\bf\it Step 1: Reduction of $\M$.} When $p_*=0$, there are 5 possible types denoted by T1, T2, T3, T4 and T1$'$, out of which type T1$'$ only occurs when $\la=3$. The reduction rules combined with Lemma \ref{lem:reductionmatrix} make it possible to keep only two survivor types: the type T1 of a square with exactly one edge belonging to $\cK_n$ and the type T2 of a square with exactly two consecutive edges belonging to $\cK_n$, see the first 3 columns of Table \ref{tab:ReductionRulesSquare}. When $p_*=1$, the main difference is that the set $\cK_{n+1}$ can possibly grow inside the squares which intersect $\cK_n$ by only one vertex or several vertices. This induces a new set of possible types, bringing the total to 13. Again, thanks to the reduction rules, we obtain a matrix of size $3\times 3$ with only 3 types: T1, T2 and T0 which corresponds to a square which intersects $\cK_n$ by only one vertex, see Table \ref{tab:ReductionRulesSquare}. In particular, the set of children of a square parent of type T0 is deterministic and consists in one square of type T2 and two squares of type T0.

\pass{\bf\it Step 2: Decomposition of the set of potential children of each parent.} In the same way as in the proof of Proposition \ref{prop:Model0HexaCenter}, we identify the whole set of potential children of a parent $\cT_{n,\ell}$ of type T1 or T2. In the case of the square tessellation, these potential children lie in the first two lines  of $\cT_{n+1}$ inside $\cT_{n,\ell}$ \textit{above} the edges (one or two) of $\cT_{n,\ell}$ which lie in $\cK_n$. We then put together the squares from those two lines which have the same mean contribution, i.e. we partition the set of potential children into groups which are labelled with different colors, see Table \ref{tab:SquareChild}.

\pass{\bf\it Step 3: Calculation of the reduced matrix coefficients.} The method is identical to the hexagonal case and for sake of brevity, we refer the reader to Tables \ref{tab:SquareChildModel0}, \ref{tab:SquareChildModel1T1} and \ref{tab:SquareChildModel1T2}  which summarize all the calculations. 
\end{proof}

\noi{\bf Remarks.} Notice that for the two deterministic cases when $(p,p_*)=(0,0)$ or $(1,1)$, we expect the set $\partial \cK_\infty$ to be of dimension $1$. This is clearly confirmed by our calculation of the reduced matrix which is $\M'=\left[
\begin{smallmatrix} \la & 0  \\ 2(\la-1) & 1 \end{smallmatrix}
\right]$ 
or $\M'=\left[
\begin{smallmatrix} \la & 0 & 0 \\ 2(\la-2) & 1 & 0 \\ 0 & 1 & 0 \end{smallmatrix}\right]$ respectively, with $\rho_\M=\la$ in both cases.

As in the hexagonal case, we represent in Figure \ref{fig:SquareModelSimuDim} both the theoretical and numerical dimension as a function of $p$. This requires to approximate the value of $\rho_\M$ with numerical methods.

\renewcommand{\arraystretch}{0.8}
\begin{figure}[h!]
\centering
\begin{tabular}{cc}
\includegraphics[scale=0.4]{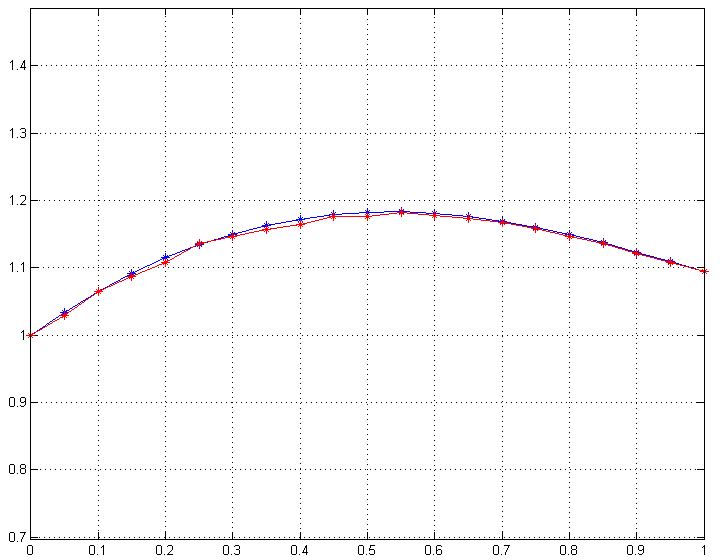} &
\includegraphics[scale=0.4]{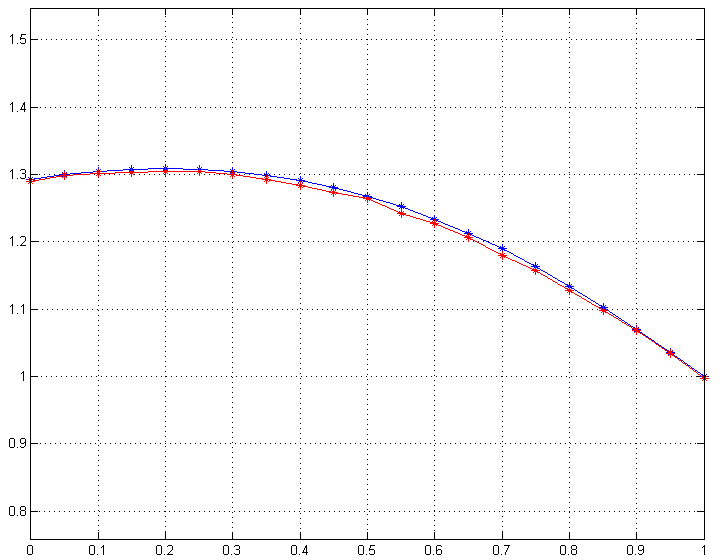}\\
\footnotesize{(a)} & \footnotesize{(b)}
\end{tabular}
\vspace*{-.35cm}
\caption{Comparison between the graphs of the theoretical (in blue) and numerical (in red) dimension of $\partial \cK_\infty$ as a function of $p\in [0,1]$ in the square case, when $p_*=0$, for $\la=4$ when $p_*=0$ (a) and $p_*=1$ (b).}
\label{fig:SquareModelSimuDim}
\end{figure} 
\renewcommand{\arraystretch}{1.5}

Let us end this section with a result about the ``holes'' that can occur in the case of the square tessellation (see Figure \ref{fig:SquareModelSimu}).

\begin{prop}\label{prop:squareholes}
Let $H_n$ be the number of holes included in the set $\cK_n$. There exist a positive random variables $H_{\infty}$ such that $\P(H_{\infty}>0)=1$ and such that almost surely
\begin{equation}\label{eq:limsquareholes}
\lim_{n\to \infty} \rho_\M^{-n} H_n=H_\infty.
\end{equation}
\end{prop}

\begin{proof}
The key observation is that each hole at step $k$ is exactly associated with a child of type T4, and given a hole, it will never be totally recovered by black squares of future generations. For instance, in Figure \ref{fig:friseHolesSquare}, a first hole appears at the generation $n$ as the product of a T4 tile in $\cR_n$. Then at the generation $n+1$, a smaller hole appears inside that T4 tile and two new holes appear at generation $n+2$, one of them being inside the previous one. 
\begin{figure}
\centering
\includegraphics[width=0.75\linewidth]{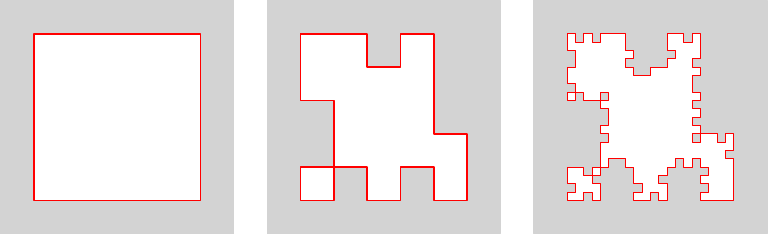}
\caption{Appearance of holes in the square model over 3 generations for $\la=4$}
\label{fig:friseHolesSquare}
\end{figure}
Consequently, we obtain 
\begin{equation*}
H_n = \sum_{k=0}^n Z_k^4 
=  \sum_{k=0}^n \langle Z_k,{\mathbf e_4}\rangle 
=  \sum_{k=0}^n \rho_\M^k\langle \rho_\M^{-k}Z_k,{\mathbf e_4}\rangle.
\end{equation*}
Because of the argument given above \eqref{eq:limcup} related to the rest of the geometric sum, any hole which appears is never fully filled. Since $\rho_\M^n\longrightarrow\infty$ and $\langle \rho_\M^{-k}Z_k,{\mathbf e_4}\rangle \longrightarrow \langle W,{\mathbf e_4}\rangle$, we obtain that when $n\to\infty$, $H_n$ is equivalent almost surely to $\mfrac{\rho_\M}{\rho_\M-1}\langle W,{\mathbf e_4}\rangle\rho_\M^n$, hence the result.
\end{proof}

\subsection{The case of the triangular tessellation}\label{sec:ModelTriangle}

\pass

In this section, we consider the case when $\cT$ is the triangular tessellation and when the origin is fixed at one of the vertices of a triangle. We consider the particular choice $p_*=0$. Curiously, because of the geometry of the triangular tessellation, the case $p_*=1$ involves a lot more new types appearing along the iteration of the process and therefore, we have decided to omit it.

\begin{prop}\label{prop:TriangleModel}
For the triangular model with parameters $\la\ge3$, $p\in[0,1]$ and $p_*=0$, the reduced matrix $\M'$ is 
\begin{equation*} 
\M'=(\la-3)M_1'(p)+M_2'(p)
\end{equation*}
where
\begin{equation*}
M_1'(p)  =\left[
\begin{matrix} 1+p-2p^2 & p^2 \\
2(1+p-2p^2) & 2p^2\end{matrix}
\right]\mbox{ and }
M_2'(p)  =\left[
\begin{matrix} 3+p & 0 \\ 
4+3p-6p^2 & 1-p+3p^2 \end{matrix}
\right].
\end{equation*}

\pass Moreover, the spectral radius of the matrix $\M$ is given by
\begin{align*}
\rho_\M & = \mfrac12 \Big(\la Q_1(p)+Q_2(p)+\sqrt{\la^2 Q_3(p)+\la Q_4(p)+_5(p)}
\Big),
\end{align*}
where $Q_1(p) = p+1$, $Q_2(p) = 3p^2-3p+1$, $Q_3(p) = -16p^4+8p^3+9p^2+2p+1$, $Q_4(p) = 72p^4-46p^3-40p^2-4p-2$ and $Q_5(p) = -63p^4+54p^3+31p^2+2p+1$.
\end{prop}

\begin{proof}
The method follows the exact same pattern as in Propositions \ref{prop:Model0HexaCenter} and \ref{prop:SquareModel}, i.e. we start by identifying the types, we then reduce the matrix, decompose the set of potential children of one triangle with a fixed type, calculate the contribution of each potential child, see Table \ref{tab:TriangleChildModel0}, and finally calculate the entries of the reduced matrix. 
\end{proof}

Again, the theoretical and numerical versions of the dimension are represented as functions of $p$ in Figure \ref{fig:TriangleSimuDim}.

\renewcommand{\arraystretch}{0.8}
\begin{figure}[h!]
\centering
\begin{tabular}{cc}
\includegraphics[scale=0.575]{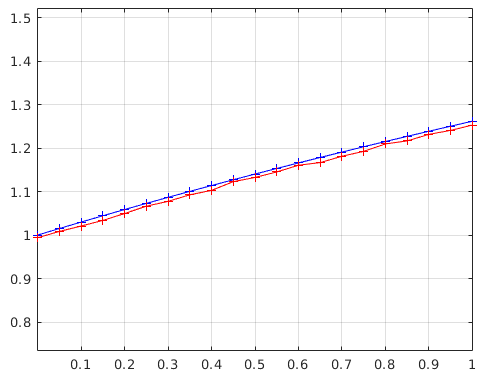} &
\includegraphics[scale=0.575]{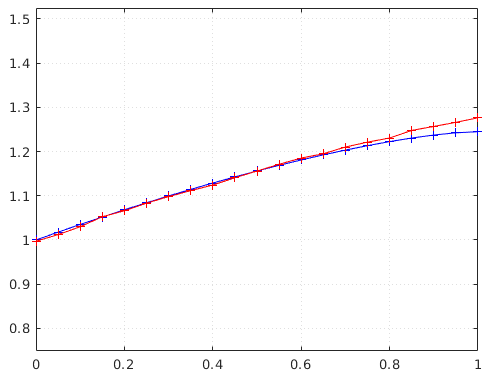} \\
\footnotesize{(a)} & \footnotesize{(b)}
\end{tabular}
\vspace*{-.35cm}
\caption{Comparison between the graphs of the theoretical (in blue) and numerical (in red) dimension of $\partial \cK_\infty$ as a function of $p\in [0,1]$ in the triangular case, when $p_*=0$, for $\la=3$ (a) and $\la=4$ (b).}
\label{fig:TriangleSimuDim}
\end{figure} 
\renewcommand{\arraystretch}{1.5}

\newpage
\addtocontents{toc}{\vspace{0.2cm}}%
\section{Final discussion}\label{sec:outro}

This section collects a few possible extensions and open problems related to our model. Obviously, the construction may generate many possible variants, like for instance rescaling at each step from a vertex of the tessellation $\cT$ in the hexagonal case, making the Bernoulli parameters $p$ and $p_*$ depend on the relative position of a tile of $\la^{-(n+1)}\cT$ with respect to $\partial \cK_n$ or considering other tessellations like the rhombille tiling. We expect most of them to satisfy the same properties, at the expense of inducing higher dimensional reproduction matrices and non-explicit spectral radii. We develop below what we consider to be the more stimulating questions.

\pass{\bf\it Considering higher dimension.} A natural extension of the model consists in considering regular tessellations in higher dimension. For instance, we can investigate the most natural case of the cubic tessellation of $\R^m$, $m\ge 3$, keeping in mind the parallel with the Eden model described at the beginning of Section \ref{sec:intro}. A similar coding by a Galton-Watson tree naturally occurs and the conclusion of Theorem \ref{theo:main} still holds. The difficulty lies in the calculation of the reproduction matrix. This requires to identify the different types which correspond to the possible intersection of a cube of $\cR_n$ with the set $\cK_n$. This might be one or several facets of the cube or even a combination of several lower-dimensional faces. 

\ppass{\bf\it Choosing a real scaling factor $\la$.} We can consider a more general {\it real} scaling parameter $\la\ge2$ and ask again for the convergence of the sequence $(\cK_n)_{n\ge0}$ and the dimension of the limit set $\partial\cK_\infty$. Actually, because of our condition for the tiles which intersect $\intt(\cK_n)$, we observe that in the case of the triangular (resp. square) tessellation we obtain an empty set $\cK_n^\circ$ and a set $\cK_n^\bullet$ equal to a square (resp. a triangle) in the case of the square (resp. triangular) tessellation. This implies that the sequence of random compact sets is trivial. The situation is more interesting in the case of the hexagonal tessellation because the sets $\cK_n$ are not trivial. This difference is due to the fact that a triangle (resp. a square) is a strongly self-similar set, that is it is equal to a finite union of smaller triangles (resp. squares) with disjoints interiors. This is not possible for an hexagon so that the exterior ring is never a trivial set. The sequence $(\cK_n)_{n\ge0}$ is clearly convergent as soon as $\la\ge 2$, thus we can ask again for the dimension of the limit set $\partial\cK_\infty$. This proves to be a quite more difficult question. We will need to distinguish the case when $\la$ is rational or not. In the first case, our model is still random although most of the time, tiles are added to $\cK_n$ deterministically, i.e. they belong to $\cK_n^\bullet$, while random tiles appear at certain stages only. We expect to find bounds for the number of tiles in $\cR_n$, and subsequently for the box-dimension of $\partial \cK_\infty$, through a comparison to another model associated with an integer parameter $\la$. In the second case, the model is not random: at each step of the construction, every tile from $\la^{-(n+1)}\cT$ along $\partial \cK_n$ belongs to $\cK_n^\bullet$ and consequently is kept in $\cK_{n+1}$. The general issue related to this model is that the curves $\partial \cK_n$ are not modified in a self-similar way, making it difficult to estimate the cardinality of the sets $\cR_n$. Let us finally remark that solving that question would certainly constitute an intermediary step towards the study of iterated Poisson-Voronoi models, as described below. 

\pass{\bf\it Randomizing the end tiles.} The most unnatural condition in the current model lies in the assumption that the tiles from $\la^{-(n+1)}\cT$ which are located at the two ends of an edge along $\partial \cK_n$ are governed by a deterministic Bernoulli variable with parameter $p_*\in \{0,1\}$, see the treatment of Group $2$ in the description in Section \ref{sec:intro}. That choice guarantees the complete independence of the lineage of all the tiles involved in $\cR_n$ and therefore, leads us to obtain an exact Galton-Watson process. Obviously, it would be more relevant to associate to all the tiles of $\la^{-(n+1)}\cT$ along $\partial \cK_n$, including Group $2$, an i.i.d. sequence of Bernoulli variables with common parameter $p$ as in a regular percolation model. Such a model would induce a local dependency of the lineage of neighboring tiles. Even though the underlying coding process would no longer be a perfect Galton-Watson tree, we would still expect a similar set of results to hold. This would involve a lot more technical work on the random tree and its associated martingale. Such extension certainly constitutes the material of a possible future work. In fact, the present paper is intended as a first contribution on iterated tessellations and brings to the table the key approach of coding by a tree which may hopefully be developed thereafter. \enlargethispage*{0.5cm}

\pass{\bf\it Reaching iterated Poisson-Voronoi models.} We already discussed in Section \ref{sec:intro} the connection between our contribution and the original work \cite{tchou01} on iterated Poisson-Voronoi tessellations. Actually, let us describe a slight modification of the initial construction done in \cite{tchou01} which could play the role of an intermediate model along the way from our model towards the limiting boundary set described in \cite{tchou01}. Let $\la>1$ and $(\cP_n)_{n\ge0}$ be a family of independent homogeneous Poisson point processes of $\R^d$ such that $\cP_n$ has intensity $\la^n$. Set $\cV_n$ for the Voronoi tessellation associated with $\cP_n$. Then we can define a sequence $(\cK_n)_{n\ge0}$ of random sets as follows: $\cK_0$ is the cell from $\cV_0$ containing the origin ; then $\cK_{n+1}$ is obtained by replacing $\cK_n$ with the set made up with all the cells $\cC$ from $\cV_{n+1}$ such that $\cC\cap\intt(\cK_n)\neq\emptyset$. By construction $\cK_{n+1}\supset \cK_n$ and we expect the sequence $(\cK_n)_{n\ge0}$ to be almost surely bounded and converge to a random set $\cK_\infty$ almost surely. Thus, we can ask again for the dimension of the limit set $\partial\cK_\infty$. To answer that question, we would certainly try to use the tiles of each tessellation as a basis for covering $\partial \cK_\infty$, exactly in the same way as the set $\cR_n$ was considered in Section \ref{sec:galton}. Nevertheless, one of the main difficulties in doing so lies in the non-uniformity of the shapes of the different Voronoi cells, which means in particular that some of them may have an elongated shape with an unusually large diameter. Moreover, the problem of counting the tiles from the set $\cR_n$ would then reduce in estimating the number of cells of a Poisson-Voronoi tessellation crossed by a craggy polygonal line. In particular, the section of a Poisson-Voronoi tessellation with one line already represents an issue and constitutes the substance of a recent paper \cite{Gus24}, see also the construction of the Markov path along the Voronoi nuclei in \cite{Bac00}. In conclusion, the different models derived from iterated Poisson-Voronoi tessellation certainly prove to generate stimulating but tough questions.

\addtocontents{toc}{\vspace{0.2cm}}%
\section{Appendix: Tables and Figures}\label{sec:annexe}

This last section is devoted to the collection of figures and tables needed to calculate each entry of the reproduction matrix whose expression is used for each of the three tessellations, see Propositions \ref{prop:Model0HexaCenter} and \ref{prop:Model1HexaCenter}, Proposition \ref{prop:SquareModel} and Proposition \ref{prop:TriangleModel} respectively. 

\subsection{The hexagonal tessellation} 

\noi

\renewcommand{\arraystretch}{1.5}
\setlength{\arrayrulewidth}{0.1pt}
\begin{table}[h!]
\centering
\resizebox{14.5cm}{!}{%
\begin{tabular}{|c|c|c|c|}
\cline{3-4}
\multicolumn{2}{c|}{} & \multicolumn{1}{c|}{$p_*=0$} & \multicolumn{1}{c|}{$p_*=1$} \\
\hline
\multirow{2}{*}{\raisebox{-1.75\height}{$r=0$}} 
& \raisebox{2.5\height}{Parent of type T1}
& \rule{0cm}{1.9cm}\includegraphics[scale=0.5]{Model0HexaR0ChildT1} 
& \rule{0cm}{1.9cm}\includegraphics[scale=0.5]{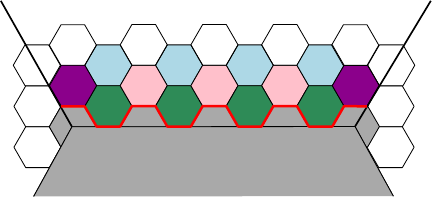} \\
\cline{2-4} 
& \raisebox{3.5\height}{Parent of type T2}
& \rule{0cm}{2.25cm}\includegraphics[scale=0.5]{Model0HexaR0ChildT2} 
& \rule{0cm}{2.25cm}\includegraphics[scale=0.5]{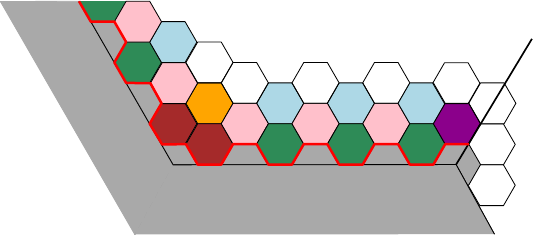} \\
\hline
\multirow{2}{*}{\raisebox{-1.75\height}{$r=1$}} 
& \raisebox{2.5\height}{Parent of type T1}
& \rule{0cm}{1.9cm}\includegraphics[scale=0.5]{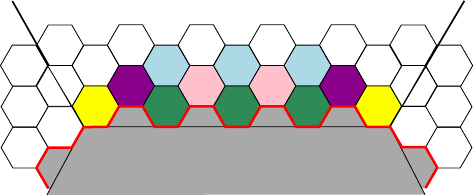} 
& \rule{0cm}{1.9cm}\includegraphics[scale=0.5]{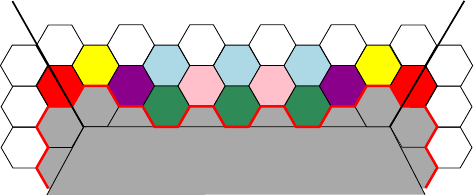} \\
\cline{2-4} 
& \raisebox{3.5\height}{Parent of type T2}
& \rule{0cm}{2.25cm}\includegraphics[scale=0.5]{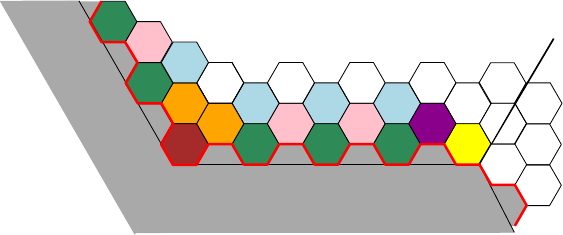} 
& \rule{0cm}{2.25cm}\includegraphics[scale=0.5]{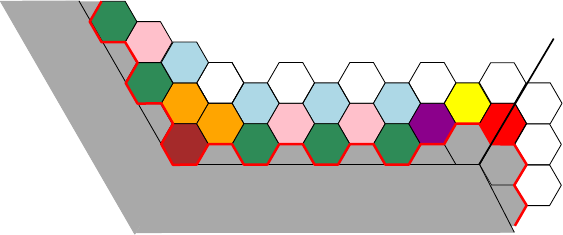} \\
\hline
\multirow{2}{*}{\raisebox{-1.75\height}{$r=2$}} 
& \raisebox{2.5\height}{Parent of type T1}
& \rule{0cm}{1.9cm}\includegraphics[scale=0.5]{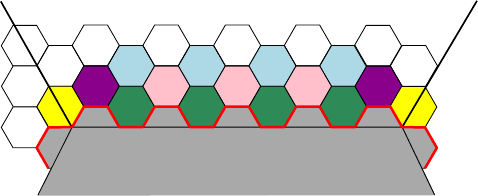} 
& \rule{0cm}{1.9cm}\includegraphics[scale=0.5]{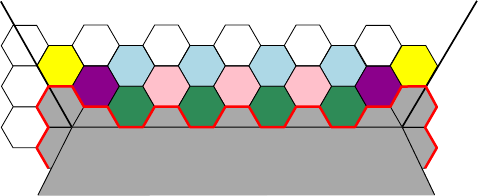} \\
\cline{2-4} 
& \raisebox{3.5\height}{Parent of type T2}
& \rule{0cm}{2.25cm}\includegraphics[scale=0.5]{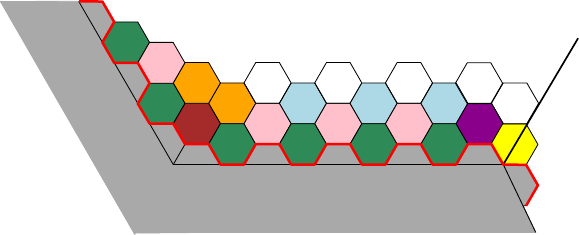} 
& \rule{0cm}{2.25cm}\includegraphics[scale=0.5]{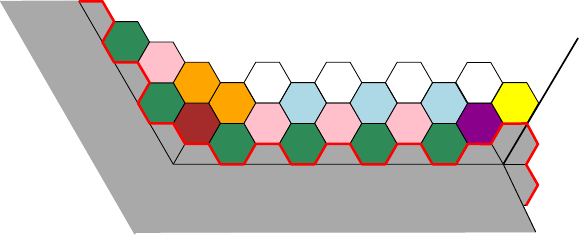} \\
\hline
\end{tabular}}
\pass\caption{The potential children of an hexagonal parent for the case of the centered hexagonal tessellation.}
\label{tab:HexagonChild}
\end{table}

\renewcommand{\arraystretch}{1.5}
\setlength{\arrayrulewidth}{0.1pt}
\begin{table}[h!]
\centering
\resizebox{13.5cm}{!}{%
\begin{tabular}{|c|c|c|c|c|c|}
\cline{1-4}
\multirow{4}{*}{\raisebox{2.25\height}{\includegraphics[scale=0.8]{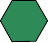}}} & \raisebox{0.75\height}{\begin{minipage}{3cm} $$ \begin{array}{cl}
\text{T1:} & \la' \\ \text{T2:} & 2(\la'-1) \end{array} $$ \end{minipage}} & \rule{0cm}{1.5cm}\includegraphics[scale=0.75]{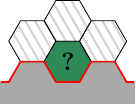} & \rule{0cm}{1.5cm}\includegraphics[scale=0.75]{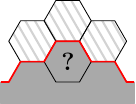} \\ \cline{2-4}
& $\eps$ & $0$ & $1$ \\ \cline{2-4}
& $\P(B_p^{(1)}=\eps)$ & $1-p$ & $p$ \\ \cline{2-4}
& $F^{(0)}_{\footnotesize\texttt{green}}(\eps)$ & $(-1,2)$ & $(0,0)$ \\ 
\cline{1-6}
\multirow{4}{*}{\raisebox{2.25\height}{\includegraphics[scale=0.8]{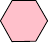}}} & \raisebox{0.75\height}{\begin{minipage}{3cm} $$ \begin{array}{cl}
\text{T1:} & \la'-1 \\ \text{T2:} & 2(\la'-1) \end{array} $$ \end{minipage}} & \rule{0cm}{1.5cm}\includegraphics[scale=0.75]{Model0HexaR0ChildT1Pink00} & \rule{0cm}{1.5cm}\includegraphics[scale=0.75]{Model0HexaR0ChildT1Pink01} & \rule{0cm}{1.5cm}\includegraphics[scale=0.75]{Model0HexaR0ChildT1Pink10} & \rule{0cm}{1.5cm}\includegraphics[scale=0.75]{Model0HexaR0ChildT1Pink11} \\ \cline{2-6}
& $\eps$ & $(0,0)$ & $(0,1)$ & $(1,0)$ & $(1,1)$\\ \cline{2-6}
& $\P(B_p^{(2)}=\eps)$ & $(1-p)^2$ & $(1-p)p$ & $p(1-p)$ & $p^2$ \\ \cline{2-6}
& $F^{(0)}_{\footnotesize\texttt{pink}}(\eps)$ & $(1,0)$ & $(0,1)$ & $(0,1)$ & $(-1,2)$ \\ \cline{2-6}
\cline{1-4}
\multirow{4}{*}{\raisebox{2.25\height}{\includegraphics[scale=0.8]{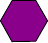}}} & \raisebox{1\height}{\begin{minipage}{3cm} $$ \begin{array}{cl}
\text{T1:} & 2 \\ \text{T2:} & 2 \end{array}$$ \end{minipage}} & \rule{0cm}{2cm}\includegraphics[scale=0.75]{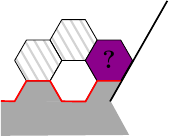} & \rule{0cm}{2cm}\includegraphics[scale=0.75]{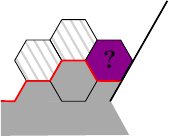} \\ \cline{2-4}
& $\eps$ & $0$ & $1$ \\ \cline{2-4}
& $\P(B_p^{(1)}=\eps)$ & $1-p$ & $p$ \\ \cline{2-4}
& $F^{(0)}_{\footnotesize\texttt{purple}}(\eps)$ & $(1,0)$ & $(0,1)$ \\ 
\cline{1-4}
\multirow{4}{*}{\raisebox{2.25\height}{\includegraphics[scale=0.8]{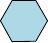}}} & \raisebox{0.75\height}{\begin{minipage}{3cm} $$ \begin{array}{cl}
\text{T1:} & \la' \\ \text{T2:} & 2(\la'-1) \end{array} $$ \end{minipage}} & \rule{0cm}{1.5cm}\includegraphics[scale=0.75]{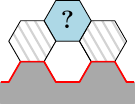} & \rule{0cm}{1.5cm}\includegraphics[scale=0.75]{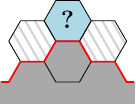} \\ \cline{2-4}
& $\eps$ & $0$ & $1$ \\ \cline{2-4}
& $\P(B_p^{(1)}=\eps)$ & $1-p$ & $p$ \\ \cline{2-4}
& $F^{(0)}_{\footnotesize\texttt{blue}}(\eps)$ & $(0,0)$ & $(1,0)$ \\ 
\cline{1-6}
\multirow{4}{*}{\raisebox{2.25\height}{\includegraphics[scale=0.8]{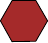}}} & \raisebox{1\height}{ \begin{minipage}{3cm} $$\begin{array}{cl}
\text{T1:} & 0 \\ \text{T2:} & 2 \end{array}$$ \end{minipage}} & \rule{0cm}{2cm}\includegraphics[scale=0.75]{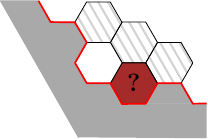} & \rule{0cm}{2cm}\includegraphics[scale=0.75]{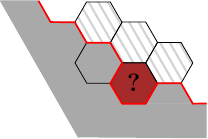} & \rule{0cm}{2cm}\includegraphics[scale=0.75]{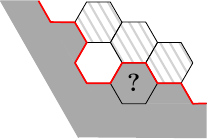} &
\rule{0cm}{2cm}\includegraphics[scale=0.75]{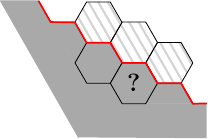} \\ \cline{2-6}
& $\eps$ & $(0,0)$ & $(0,1)$ & $(1,0)$ & $(1,1)$ \\ \cline{2-6}
& $\P(B_p^{(2)}=\eps)$ & $(1-p)^2$ & $(1-p)p$ & $p(1-p)$ & $p^2$ \\ \cline{2-6}
& $F^{(0)}_{\footnotesize\texttt{brown}}(\eps)$ & $(-1,2)$ & $(-2,3)$ & $(0,0)$ & $(0,0)$ \\ \cline{2-6}
\cline{1-6}
\multirow{4}{*}{\raisebox{2.25\height}{\includegraphics[scale=0.8]{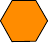}}} & \raisebox{1\height}{ \begin{minipage}{3cm} $$ \begin{array}{cl}
\text{T1:} & 0 \\ \text{T2:} & 1 \end{array} $$ \end{minipage}} & \rule{0cm}{2cm}\includegraphics[scale=0.75]{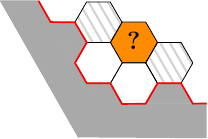} & \rule{0cm}{2cm}\includegraphics[scale=0.75]{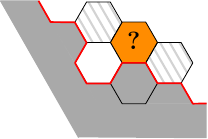} & \rule{0cm}{2cm}\includegraphics[scale=0.75]{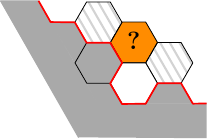} & \rule{0cm}{2cm}\includegraphics[scale=0.75]{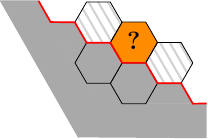} \\ \cline{2-6}
& $\eps$ & $(0,0)$ & $(0,1)$ & $(1,0)$ & $(1,1)$\\ \cline{2-6}
& $\P(B_p^{(2)}=\eps)$ & $(1-p)^2$ & $(1-p)p$ & $p(1-p)$ & $p^2$ \\ \cline{2-6}
& $F^{(0)}_{\footnotesize\texttt{orange}}(\eps)$ & $(0,0)$ & $(1,0)$ & $(1,0)$ & $(0,1)$ \\ \cline{2-6}
\cline{1-6}
\end{tabular}}
\pass\caption{Description of the children of an hexagonal parent when $p_*=0$ and $r=0$.}
\label{tab:Model0HexaR0Child}
\end{table}

\pass

\renewcommand{\arraystretch}{1.5}
\setlength{\arrayrulewidth}{0.1pt}
\begin{table}[h!]
\centering
\resizebox{12.25cm}{!}{%
\begin{tabular}{|c|c|c|c|c|c|}
\cline{1-3}
\multirow{4}{*}{\raisebox{2.25\height}{\includegraphics[scale=0.8]{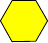}}} & \raisebox{1\height}{\begin{minipage}{3cm}$$\begin{array}{cl}
\text{T1:} & 2 \\ \text{T2:} & 2 \end{array}$$\end{minipage}} & \rule{0cm}{2cm}\includegraphics[scale=0.75]{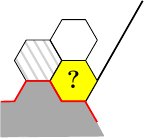} \\ \cline{2-3}
& $\eps$ & $0$ \\ \cline{2-3}
& $\P(B_p^{(1)}=\eps)$ & $1$ \\ \cline{2-3}
& $F^{(1)}_{\footnotesize\texttt{yellow}}(\eps)$ & $(0,1)$  \\ 
\cline{1-4}
\multirow{4}{*}{\raisebox{2.25\height}{\includegraphics[scale=0.8]{HexaGreen}}} & \raisebox{0.75\height}{\begin{minipage}{3cm}$$ \begin{array}{cl}
\text{T1:} & \la'-1 \\ \text{T2:} & 2(\la'-1)  \end{array} $$\end{minipage}} & \rule{0cm}{1.5cm}\includegraphics[scale=0.75]{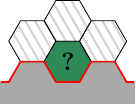} & \rule{0cm}{1.5cm}\includegraphics[scale=0.75]{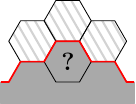} \\ \cline{2-4}
& $\eps$ & $0$ & $1$ \\ \cline{2-4}
& $\P(B_p^{(1)}=\eps)$ & $1-p$ & $p$ \\ \cline{2-4}
& $F^{(1)}_{\footnotesize\texttt{green}}(\eps)$ & $(-1,2)$ & $(0,0)$ \\ 
\cline{1-6}
\multirow{4}{*}{\raisebox{2.25\height}{\includegraphics[scale=0.8]{HexaPink}}} & \raisebox{0.75\height}{\begin{minipage}{3cm}$$\begin{array}{cl}
\text{T1:} & \la'-2 \\ \text{T2:} & 2(\la'-2) \end{array}$$\end{minipage}} & \rule{0cm}{1.5cm}\includegraphics[scale=0.75]{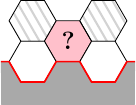} & \rule{0cm}{1.5cm}\includegraphics[scale=0.75]{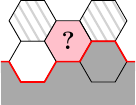} & \rule{0cm}{1.5cm}\includegraphics[scale=0.75]{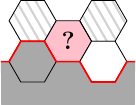} & \rule{0cm}{1.5cm}\includegraphics[scale=0.75]{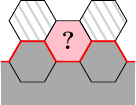} \\ \cline{2-6}
& $\eps$ & $(0,0)$ & $(0,1)$ & $(1,0)$ & $(1,1)$\\ \cline{2-6}
& $\P(B_p^{(2)}=\eps)$ & $(1-p)^2$ & $(1-p)p$ & $p(1-p)$ & $p^2$ \\ \cline{2-6}
& $F^{(1)}_{\footnotesize\texttt{pink}}(\eps)$ & $(1,0)$ & $(0,1)$ & $(0,1)$ & $(-1,2)$ \\ 
\cline{1-6}
\multirow{4}{*}{\raisebox{2.25\height}{\includegraphics[scale=0.8]{HexaPurple}}} & \raisebox{1\height}{\begin{minipage}{3cm}$$\begin{array}{cl}
\text{T1:} & 2 \\ \text{T2:} & 2 \end{array}$$\end{minipage}} & \rule{0cm}{2cm}\includegraphics[scale=0.75]{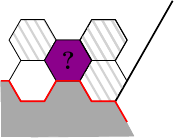} & \rule{0cm}{2cm}\includegraphics[scale=0.75]{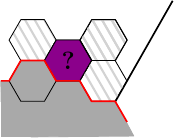} \\ \cline{2-4}
& $\eps$ & $0$ & $1$ \\ \cline{2-4}
& $\P(B_p^{(1)}=\eps)$ & $1-p$ & $p$ \\ \cline{2-4}
& $F^{(1)}_{\footnotesize\texttt{purple}}(\eps)$ & $(1,0)$ & $(0,1)$ \\ 
\cline{1-4}
\multirow{4}{*}{\raisebox{2.25\height}{\includegraphics[scale=0.8]{HexaBlue}}} & \raisebox{0.75\height}{\begin{minipage}{3cm}$$\begin{array}{cl}
\text{T1:} & \la'-1 \\ \text{T2:} & 2(\la'-1)  \end{array}$$\end{minipage}} & \rule{0cm}{1.5cm}\includegraphics[scale=0.75]{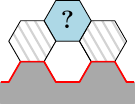} & \rule{0cm}{1.5cm}\includegraphics[scale=0.75]{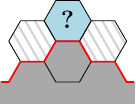} \\ \cline{2-4}
& $\eps$ & $0$ & $1$ \\ \cline{2-4}
& $\P(B_p^{(1)}=\eps)$ & $1-p$ & $p$ \\ \cline{2-4}
& $F^{(1)}_{\footnotesize\texttt{blue}}(\eps)$ & $(0,0)$ & $(1,0)$ \\ 
\cline{1-4}
\multirow{4}{*}{\raisebox{2.25\height}{\includegraphics[scale=0.8]{HexaBrown}}} & \raisebox{1\height}{\begin{minipage}{3cm}$$\begin{array}{cl}
\text{T1:} & 0 \\ \text{T2:} & 1 \end{array}$$\end{minipage}} & \rule{0cm}{2.1cm}\includegraphics[scale=0.75]{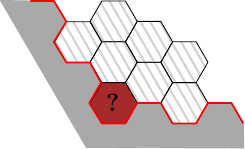} & \rule{0cm}{2.1cm}\includegraphics[scale=0.75]{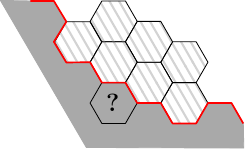} \\ \cline{2-4}
& $\eps$ & $0$ & $1$ \\ \cline{2-4}
& $\P(B_p^{(1)}=\eps)$ & $1-p$ & $p$ \\ \cline{2-4}
& $F^{(1)}_{\footnotesize\texttt{brown}}(\eps)$ & $(-2,3)$ & $(0,0)$ \\ 
\cline{1-6}
\multirow{4}{*}{\raisebox{2.25\height}{\includegraphics[scale=0.8]{HexaOrange}}} & \raisebox{1\height}{\begin{minipage}{3cm}$$\begin{array}{cl}
\text{T1:} & 0 \\ \text{T2:} & 2 \end{array}$$\end{minipage}} & \rule{0cm}{2.1cm}\includegraphics[scale=0.75]{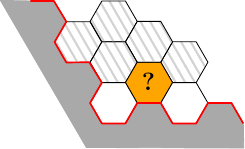} & \rule{0cm}{2.1cm}\includegraphics[scale=0.75]{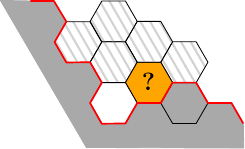} & \rule{0cm}{2.1cm}\includegraphics[scale=0.75]{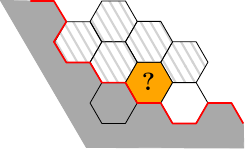} & \rule{0cm}{2.1cm}\includegraphics[scale=0.75]{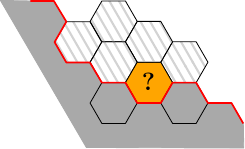} \\ \cline{2-6}
& $\eps$ & $(0,0)$ & $(0,1)$ & $(1,0)$ & $(1,1)$\\ \cline{2-6}
& $\P(B_p^{(2)}=\eps)$ & $(1-p)^2$ & $(1-p)p$ & $p(1-p)$ & $p^2$ \\ \cline{2-6}
& $F^{(1)}_{\footnotesize\texttt{orange}}(\eps)$ & $(1,0)$ & $(0,1)$ & $(0,1)$ & $(-1,2)$ \\ \cline{2-6}
\cline{1-6}
\end{tabular}}
\pass\caption{Description of the children of an hexagonal parent when $p_*=0$ and $r=1$. In the case $\la'=1$, i.e. $\la=4$, the formulas are also consistent because the pink and purple groups are replaced by a new one with only one hexagon.}
\label{tab:Model0HexaR1Child}
\end{table}

\renewcommand{\arraystretch}{1.5}
\setlength{\arrayrulewidth}{0.1pt}
\begin{table}[h!]
\centering
\resizebox{12.5cm}{!}{%
\begin{tabular}{|c|c|c|c|c|c|}
\cline{1-3}
\multirow{4}{*}{\raisebox{2.25\height}{\includegraphics[scale=0.8]{HexaYellow}}} & \raisebox{1\height}{\begin{minipage}{3cm}$$\begin{array}{cl}
\text{T1:} & 2 \\ \text{T2:} & 2 \end{array}$$\end{minipage}} & \rule{0cm}{2cm}\includegraphics[scale=0.75]{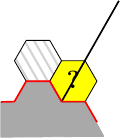} \\ \cline{2-3}
& $\eps$ & $0$ \\ \cline{2-3}
& $\P(B_p^{(1)}=\eps)$ & $1$ \\ \cline{2-3}
& $F^{(2)}_{\footnotesize\texttt{yellow}}(\eps)$ & $(0,\frac12)$  \\ 
\cline{1-4}
\multirow{4}{*}{\raisebox{2.25\height}{\includegraphics[scale=0.8]{HexaGreen}}} & \raisebox{0.75\height}{\begin{minipage}{3cm}$$ \begin{array}{cl}
\text{T1:} & \la' \\ \text{T2:} & 2\la' \end{array} $$\end{minipage}} & \rule{0cm}{1.5cm}\includegraphics[scale=0.75]{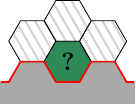} & \rule{0cm}{1.5cm}\includegraphics[scale=0.75]{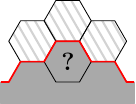} \\ \cline{2-4}
& $\eps$ & $0$ & $1$ \\ \cline{2-4}
& $\P(B_p^{(1)}=\eps)$ & $1-p$ & $p$ \\ \cline{2-4}
& $F^{(2)}_{\footnotesize\texttt{green}}(\eps)$ & $(-1,2)$ & $(0,0)$ \\
\cline{1-6}
\multirow{4}{*}{\raisebox{2.25\height}{\includegraphics[scale=0.8]{HexaPink}}} & \raisebox{0.75\height}{\begin{minipage}{3cm}$$\begin{array}{cl}
\text{T1:} & \la'-1 \\ \text{T2:} & 2(\la'-1) \end{array}$$\end{minipage}} & \rule{0cm}{1.5cm}\includegraphics[scale=0.75]{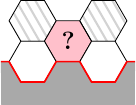} & \rule{0cm}{1.5cm}\includegraphics[scale=0.75]{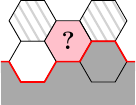} & \rule{0cm}{1.5cm}\includegraphics[scale=0.75]{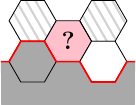} & \rule{0cm}{1.5cm}\includegraphics[scale=0.75]{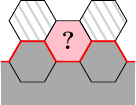} \\ \cline{2-6}
& $\eps$ & $(0,0)$ & $(0,1)$ & $(1,0)$ & $(1,1)$\\ \cline{2-6}
& $\P(B_p^{(2)}=\eps)$ & $(1-p)^2$ & $(1-p)p$ & $p(1-p)$ & $p^2$ \\ \cline{2-6}
& $F^{(2)}_{\footnotesize\texttt{pink}}(\eps)$ & $(1,0)$ & $(0,1)$ & $(0,1)$ & $(-1,2)$ \\ 
\cline{1-6}
\multirow{4}{*}{\raisebox{2.25\height}{\includegraphics[scale=0.8]{HexaPurple}}} & \raisebox{1\height}{\begin{minipage}{3cm}$$\begin{array}{cl}
\text{T1:} & 2 \\ \text{T2:} & 2 \end{array}$$\end{minipage}} & \rule{0cm}{2cm}\includegraphics[scale=0.75]{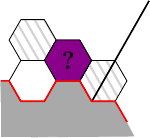} & \rule{0cm}{2cm}\includegraphics[scale=0.75]{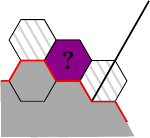} \\ \cline{2-4}
& $\eps$ & $0$ & $1$ \\ \cline{2-4}
& $\P(B_p^{(1)}=\eps)$ & $1-p$ & $p$ \\ \cline{2-4}
& $F^{(2)}_{\footnotesize\texttt{purple}}(\eps)$ & $(1,0)$ & $(0,1)$ \\
\cline{1-4}
\multirow{4}{*}{\raisebox{2.25\height}{\includegraphics[scale=0.8]{HexaBlue}}} & \raisebox{0.75\height}{\begin{minipage}{3cm}$$\begin{array}{cl}
\text{T1:} & \la' \\ \text{T2:} & 2(\la'-1) \end{array}$$\end{minipage}} & \rule{0cm}{1.5cm}\includegraphics[scale=0.75]{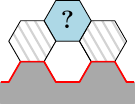} & \rule{0cm}{1.5cm}\includegraphics[scale=0.75]{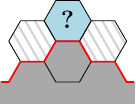} \\ \cline{2-4}
& $\eps$ & $0$ & $1$ \\ \cline{2-4}
& $\P(B_p^{(1)}=\eps)$ & $1-p$ & $p$ \\ \cline{2-4}
& $F^{(2)}_{\footnotesize\texttt{blue}}(\eps)$ & $(0,0)$ & $(1,0)$ \\ 
\cline{1-6}
\multirow{4}{*}{\raisebox{2.25\height}{\includegraphics[scale=0.8]{HexaBrown}}} & \raisebox{0.85\height}{\begin{minipage}{3cm}$$\begin{array}{cl}
\text{T1:} & 0 \\ \text{T2:} & 1 \end{array}$$\end{minipage}} & \rule{0cm}{1.8cm}\includegraphics[scale=0.75]{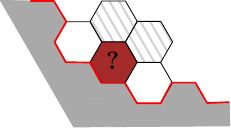} & \rule{0cm}{1.8cm}\includegraphics[scale=0.75]{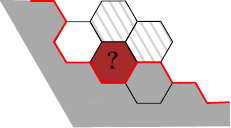} & \rule{0cm}{1.8cm}\includegraphics[scale=0.75]{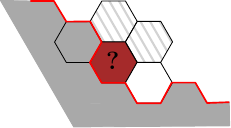} & \rule{0cm}{1.8cm}\includegraphics[scale=0.75]{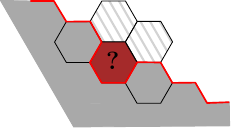} \\ \cline{2-6}
& $\eps$ & $(0,0)$ & $(0,1)$ & $(1,0)$ & $(1,1)$\\ \cline{2-6}
& $\P(B_p^{(2)}=\eps)$ & $(1-p)^2$ & $(1-p)p$ & $p(1-p)$ & $p^2$ \\ \cline{2-6}
& $F^{(2)}_{\footnotesize\texttt{brown}}(\eps)$ & $(0,1)$ & $(-1,2)$ & $(-1,2)$ & $(-2,3)$ \\ \cline{2-6}
\cline{1-6}
\multirow{4}{*}{\raisebox{2.25\height}{\includegraphics[scale=0.8]{HexaOrange}}} & \raisebox{0.85\height}{\begin{minipage}{3cm}$$\begin{array}{cl}
\text{T1:} & 0 \\ \text{T2:} & 2 \end{array}$$\end{minipage}} & \rule{0cm}{1.8cm}\includegraphics[scale=0.75]{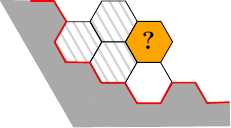} & \rule{0cm}{1.8cm}\includegraphics[scale=0.75]{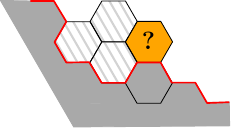}  \\ \cline{2-4}
& $\eps$ & $0$ & $1$ \\ \cline{2-4}
& $\P(B_p^{(1)}=\eps)$ & $1-p$ & $p$ \\ \cline{2-4}
& $F^{(2)}_{\footnotesize\texttt{orange}}(\eps)$ & $(0,0)$ & $(1,0)$ \\
\cline{1-4}
\end{tabular}}
\pass\caption{Description of the children of an hexagonal parent when $p_*=0$ and $r=2$.}
\label{tab:Model0HexaR2Child}
\end{table}

\newpage

\renewcommand{\arraystretch}{1.5}
\setlength{\arrayrulewidth}{0.1pt}
\begin{table}[h!]
\centering
\resizebox{7cm}{!}{%
\begin{tabular}{|c|c|c|c|c|c|}
\cline{1-3}
\multirow{4}{*}{\raisebox{2.25\height}{\includegraphics[scale=0.8]{HexaYellow}}} & \raisebox{1\height}{\begin{minipage}{3cm}$$\begin{array}{cl}
\text{T1:} & 2 \\ \text{T2:} & 2 \end{array}$$\end{minipage}} & \rule{0cm}{2cm}\includegraphics[scale=0.75]{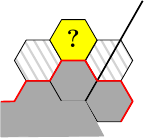} \\ \cline{2-3}
& $\eps$ & $1$ \\ \cline{2-3}
& $\P(B_p^{(1)}=\eps)$ & $1$ \\ \cline{2-3}
& $F^{(1)}_{\footnotesize\texttt{yellow}}(\eps)$ & $(1,0)$  \\ 
\cline{1-3}
\multirow{4}{*}{\raisebox{2.25\height}{\includegraphics[scale=0.8]{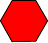}}} & \raisebox{1\height}{\begin{minipage}{3cm}$$\begin{array}{cl}
\text{T1:} & 2 \\ \text{T2:} & 2 \end{array}$$\end{minipage}} & \rule{0cm}{2cm}\includegraphics[scale=0.75]{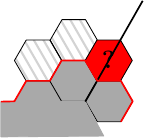} \\ \cline{2-3}
& $\eps$ & $1$ \\ \cline{2-3}
& $\P(B_p^{(1)}=\eps)$ & $1$ \\ \cline{2-3}
& $F^{(1)}_{\footnotesize\texttt{red}}(\eps)$ & $(0,\frac12)$  \\ 
\cline{1-4}
\multirow{4}{*}{\raisebox{2.25\height}{\includegraphics[scale=0.8]{HexaPurple}}} & \raisebox{1\height}{\begin{minipage}{3cm}$$\begin{array}{cl}
\text{T1:} & 2 \\ \text{T2:} & 2 \end{array}$$\end{minipage}} & \rule{0cm}{2cm}\includegraphics[scale=0.75]{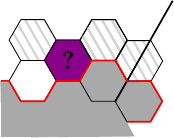} & \rule{0cm}{2cm}\includegraphics[scale=0.75]{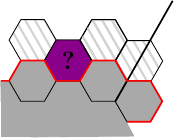} \\ \cline{2-4}
& $\eps$ & $0$ & $1$ \\ \cline{2-4}
& $\P(B_p^{(1)}=\eps)$ & $1-p$ & $p$ \\ \cline{2-4}
& $F^{(1)}_{\footnotesize\texttt{purple}}(\eps)$ & $(0,1)$ & $(-1,2)$ \\ 
\cline{1-4}
\end{tabular}}
\pass\caption{Description of the children of an hexagonal parent when $p_*=1$ and $r=1$.}
\label{tab:Model1HexaR1Child}
\end{table}

\renewcommand{\arraystretch}{1.5}
\setlength{\arrayrulewidth}{0.1pt}
\begin{table}[h!]
\centering
\resizebox{7cm}{!}{%
\begin{tabular}{|c|c|c|c|c|c|}
\cline{1-3}
\multirow{4}{*}{\raisebox{2.25\height}{\includegraphics[scale=0.8]{HexaYellow}}} & \raisebox{1\height}{\begin{minipage}{3cm}$$\begin{array}{cl}
\text{T1:} & 2 \\ \text{T2:} & 2 \end{array}$$\end{minipage}} & \rule{0cm}{2cm}\includegraphics[scale=0.75]{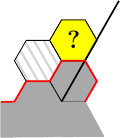} \\ \cline{2-3}
& $\eps$ & $1$ \\ \cline{2-3}
& $\P(B_p^{(1)}=\eps)$ & $1$ \\ \cline{2-3}
& $F^{(2)}_{\footnotesize\texttt{yellow}}(\eps)$ & $(1,0)$  \\
\cline{1-4}
\multirow{4}{*}{\raisebox{2.25\height}{\includegraphics[scale=0.8]{HexaPurple}}} & \raisebox{1\height}{\begin{minipage}{3cm}$$\begin{array}{cl}
\text{T1:} & 2 \\ \text{T2:} & 2 \end{array}$$\end{minipage}} & \rule{0cm}{2cm}\includegraphics[scale=0.75]{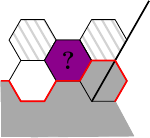} & \rule{0cm}{2cm}\includegraphics[scale=0.75]{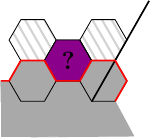} \\ \cline{2-4}
& $\eps$ & $0$ & $1$ \\ \cline{2-4}
& $\P(B_p^{(1)}=\eps)$ & $1-p$ & $p$ \\ \cline{2-4}
& $F^{(2)}_{\footnotesize\texttt{purple}}(\eps)$ & $(0,1)$ & $(-1,2)$ \\
\cline{1-4}
\end{tabular}}
\pass\caption{Description of the children of an hexagonal parent when $p_*=1$ and $r=2$.}
\label{tab:Model1HexaR2Child}
\end{table}

\newpage
\subsection{The square tessellation} 

\noi

\renewcommand{\arraystretch}{1.5}
\setlength{\arrayrulewidth}{0.1pt}
\begin{table}[h!]
\centering
\begin{tabular}{|c|c|c|c|c|c|c|c|c|c|c|c|}
\cline{3-12}
\multicolumn{1}{c}{} & & \multicolumn{10}{c|}{Phantom types} \\
\cline{3-12}
 \multicolumn{1}{c}{} & 
 & \rule{0cm}{0.75cm}\raisebox{-0.05\height}{\includegraphics[scale=0.5]{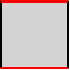}} 
 & \rule{0cm}{0.75cm}\raisebox{-0.05\height}{\includegraphics[scale=0.5]{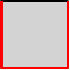}} 
 & \rule{0cm}{0.75cm}\raisebox{-0.05\height}{\includegraphics[scale=0.5]{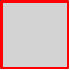}}
 & \rule{0cm}{0.75cm}\raisebox{-0.1\height}{\includegraphics[scale=0.5]{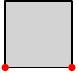}} 
 & \rule{0cm}{0.75cm}\raisebox{-0.1\height}{\includegraphics[scale=0.5]{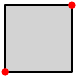}}
 &\rule{0cm}{0.75cm}\raisebox{-0.1\height}{\includegraphics[scale=0.5]{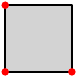}} 
 & \rule{0cm}{0.75cm}\raisebox{-0.1\height}{\includegraphics[scale=0.5]{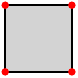}} 
 & \rule{0cm}{0.75cm}\raisebox{-0.05\height}{\includegraphics[scale=0.5]{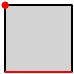}} 
 & \rule{0cm}{0.75cm}\raisebox{-0.05\height}{\includegraphics[scale=0.5]{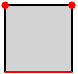}} 
 & \rule{0cm}{0.75cm}\raisebox{-0.05\height}{\includegraphics[scale=0.5]{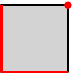}} \\
\hline
\multirow{3}{.35cm}{\rotatebox{90}{\hspace{0.3cm}Survivor types}}
& \rule{0cm}{0.75cm} \raisebox{0.5\height}{T1}\hspace{0.2cm} \raisebox{-.02\height}{\includegraphics[scale=0.5]{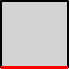}}
& \raisebox{0.5\height}{$2$} & \raisebox{0.5\height}{$-1$} & \raisebox{0.5\height}{$-4$} & \raisebox{0.5\height}{$0$} & \raisebox{0.5\height}{$0$} & \raisebox{0.5\height}{$0$} & \raisebox{0.5\height}{$0$} & \raisebox{0.5\height}{$1$} & \raisebox{0.5\height}{$1$} & \raisebox{0.5\height}{$0$} \\
\cline{2-12}
& \rule{0cm}{0.75cm} \raisebox{0.5\height}{T2}\hspace{0.2cm} \raisebox{-.02\height}{\includegraphics[scale=0.5]{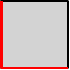}} 
& \raisebox{0.5\height}{$0$} & \raisebox{0.5\height}{$2$} & \raisebox{0.5\height}{$4$} & \raisebox{0.5\height}{$0$} & \raisebox{0.5\height}{$0$} & \raisebox{0.5\height}{$0$} & \raisebox{0.5\height}{$0$} & \raisebox{0.5\height}{$0$} & \raisebox{0.5\height}{$0$} & \raisebox{0.5\height}{$1$}  \\ 
\cline{2-12}
& \rule{0cm}{0.75cm} \raisebox{0.5\height}{T0}\hspace{0.2cm} \raisebox{-.07\height}{\includegraphics[scale=0.5]{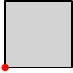}} 
& \raisebox{0.5\height}{$0$} & \raisebox{0.5\height}{$0$} & \raisebox{0.5\height}{$0$} & \raisebox{0.5\height}{$2$} & \raisebox{0.5\height}{$2$} & \raisebox{0.5\height}{$3$} & \raisebox{0.5\height}{$4$} & \raisebox{0.5\height}{$1$} & \raisebox{0.5\height}{$2$} & \raisebox{0.5\height}{$1$} \\
\hline
\end{tabular}
\pass\caption{The explicit rules in the case $p_*=1$ for the reduction of the different types T$i$. Each phantom type T$i$ may be expressed as a linear combination of the $3$ survivor types T1, T2 and T0 only. The rules in the case $p_*=0$ are exactly the same, save for the fact that T0 is not a survivor type, which removes the last line of the table.}
\label{tab:ReductionRulesSquare}
\end{table}

\renewcommand{\arraystretch}{1.5}
\setlength{\arrayrulewidth}{0.1pt}
\begin{table}[h!]
\centering
\resizebox{9.5cm}{!}{%
\begin{tabular}{|c|c|c|}
\cline{2-3}
\multicolumn{1}{c|}{} & \multicolumn{1}{c|}{$p_*=0$} & \multicolumn{1}{c|}{$p_*=1$} \\
\hline
\raisebox{7\height}{Parent of type T1}
& \rule{0cm}{4cm}\includegraphics[scale=0.75]{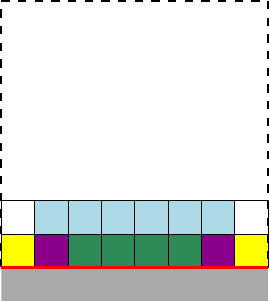} 
& \rule{0cm}{4cm}\includegraphics[scale=0.75]{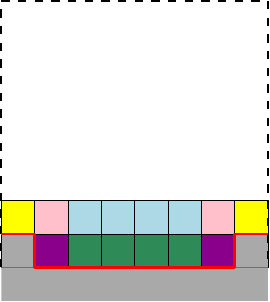} \\
\hline
\raisebox{7\height}{Parent of type T2}
& \rule{0cm}{4cm}\includegraphics[scale=0.75]{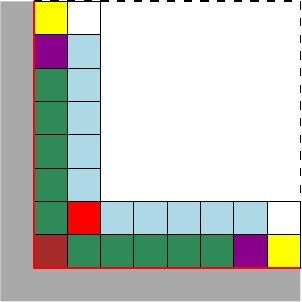} 
& \rule{0cm}{4cm}\includegraphics[scale=0.75]{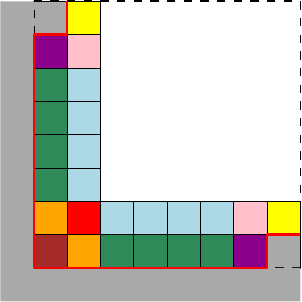} \\
\hline
\raisebox{8\height}{Parent of type T0}
& & \rule{0cm}{4.5cm}\includegraphics[scale=0.75]{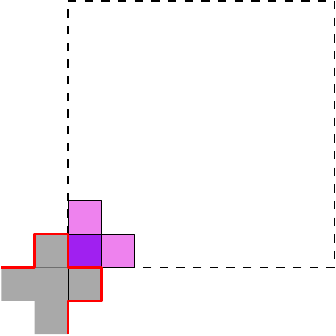} \\
\cline{1-1}\cline{3-3}
\end{tabular}
}
\pass\caption{The potential children of a square parent for the case of the square tessellation.}
\label{tab:SquareChild}
\end{table}

\newpage

\renewcommand{\arraystretch}{1.5}
\setlength{\arrayrulewidth}{0.1pt}
\begin{table}[h!]
\centering
\resizebox{11cm}{!}{%
\begin{tabular}{|c|c|c|c|c|c|}
\cline{1-6}
\multirow{4}{*}{\raisebox{2\height}{\includegraphics[scale=0.8]{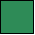}}} & \raisebox{0.75\height}{\begin{minipage}{3cm}$$\begin{array}{cl}
\text{T1:} & \la-4 \\ \text{T2:} & 2(\la-3)  \end{array}$$\end{minipage}} & \rule{0cm}{1.5cm}\includegraphics[scale=0.75]{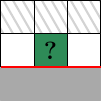} & \rule{0cm}{1.5cm}\includegraphics[scale=0.75]{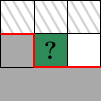} & \rule{0cm}{1.5cm}\includegraphics[scale=0.75]{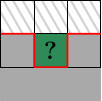} & \rule{0cm}{1.5cm}\includegraphics[scale=0.75]{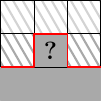} \\ \cline{2-6}
& $\eps$ & $(0,0,0)$ & $(1,0,0)$ or $(0,0,1)$ & $(1,0,1)$ & $(\eps',1,\eps'')$\\ \cline{2-6}
& $\P(B_p^{(3)}=\eps)$ & $(1-p)^3$ & $2p(1-p)^2$ & $p^2(1-p)$ & $p$ \\ \cline{2-6}
& $F_{\footnotesize\texttt{green}}(\eps)$ & $(1,0)$ & $(0,1)$ & $(-1,2)$ & $(0,0)$ \\ 
\cline{1-6}
\multirow{4}{*}{\raisebox{2\height}{\includegraphics[scale=0.8]{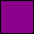}}} & \raisebox{0.75\height}{\begin{minipage}{3cm}$$\begin{array}{cl}
\text{T1:} & 2  \\ \text{T2:} & 2 \end{array}$$\end{minipage}} & \rule{0cm}{1.5cm}\includegraphics[scale=0.75]{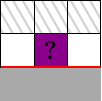} & \rule{0cm}{1.5cm}\includegraphics[scale=0.75]{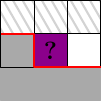} & \rule{0cm}{1.5cm}\includegraphics[scale=0.75]{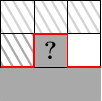} \\ \cline{2-5}
& $\eps$ & $(0,0)$ & $(1,0)$ & $(\eps',1)$ \\ \cline{2-5}
& $\P(B_p^{(2)}=\eps)$ & $(1-p)^2$ & $p(1-p)$ & $p $ \\ \cline{2-5}
& $F_{\footnotesize\texttt{purple}}(\eps)$ & $(1,0)$ & $(0,1)$ & $(0,0)$ \\ 
\cline{1-5}
\multirow{4}{*}{\raisebox{2\height}{\includegraphics[scale=0.8]{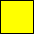}}} & \raisebox{0.75\height}{\begin{minipage}{3cm}$$\begin{array}{cl}
\text{T1:} & 2  \\ \text{T2:} & 2 \end{array}$$\end{minipage}} & \rule{0cm}{1.5cm}\includegraphics[scale=0.75]{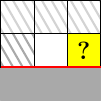} & \rule{0cm}{1.5cm}\includegraphics[scale=0.75]{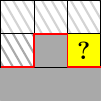} \\ \cline{2-4}
& $\eps$ & $0$ & $1$ \\ \cline{2-4}
& $\P(B_p^{(1)}=\eps)$ & $1-p$ & $p$ \\ \cline{2-4}
& $F_{\footnotesize\texttt{yellow}}(\eps)$ & $(1,0)$ & $(0,1)$ \\ 
\cline{1-4}
\multirow{4}{*}{\raisebox{2\height}{\includegraphics[scale=0.8]{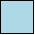}}} & \raisebox{0.75\height}{\begin{minipage}{3cm}$$\begin{array}{cl}
\text{T1:} & \la-2 \\ \text{T2:} & 2(\la-3) \end{array}$$\end{minipage}} & \rule{0cm}{1.5cm}\includegraphics[scale=0.75]{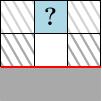} & \rule{0cm}{1.5cm}\includegraphics[scale=0.75]{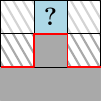} \\ \cline{2-4}
& $\eps$ & $0$ & $1$ \\ \cline{2-4}
& $\P(B_p^{(1)}=\eps)$ & $1-p$ & $p$ \\ \cline{2-4}
& $F_{\footnotesize\texttt{blue}}(\eps)$ & $(0,0)$ & $(1,0)$ \\ \cline{1-6}
\multirow{4}{*}{\raisebox{2\height}{\includegraphics[scale=0.8]{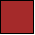}}} & \raisebox{0.75\height}{\begin{minipage}{3cm}$$\begin{array}{cl}
\text{T1:} & 0 \\ \text{T2:} & 1  \end{array}$$\end{minipage}} & \rule{0cm}{1.5cm}\includegraphics[scale=0.75]{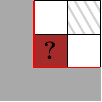} & \rule{0cm}{1.5cm}\includegraphics[scale=0.75]{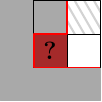} & \rule{0cm}{1.5cm}\includegraphics[scale=0.75]{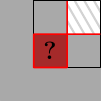} & \rule{0cm}{1.5cm}\includegraphics[scale=0.75]{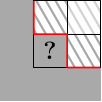} \\ \cline{2-6}
& $\eps$ & $(0,0,0)$ & $(1,0,0)$ or $(0,0,1)$ & $(1,0,1)$ & $(\eps',1,\eps'')$\\ \cline{2-6}
& $\P(B_p^{(3)}=\eps)$ & $(1-p)^3$ & $2p(1-p)^2$ & $p^2(1-p)$ & $p$ \\ \cline{2-6}
& $F_{\footnotesize\texttt{brown}}(\eps)$ & $(0,1)$ & $(-1,2)$ & $(-4,4)$ & $(0,0)$ \\ 
\cline{1-6}
\multirow{4}{*}{\raisebox{2\height}{\includegraphics[scale=0.8]{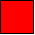}}} & \raisebox{0.75\height}{\begin{minipage}{3cm}$$\begin{array}{cl}
\text{T1:} & 0 \\ \text{T2:} & 1 \end{array}$$\end{minipage}} & \rule{0cm}{1.5cm}\includegraphics[scale=0.75]{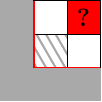} & \rule{0cm}{1.5cm}\includegraphics[scale=0.75]{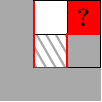} & \rule{0cm}{1.5cm}\includegraphics[scale=0.75]{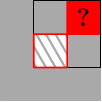} \\ \cline{2-5}
& $\eps$ & $(0,0)$ & $(0,1)$ or $(1,0)$ & $(1,1)$ \\ \cline{2-5}
& $\P(B_p^{(2)}=\eps)$ & $(1-p)^2$ & $2p(1-p)$ & $p^2$ \\ \cline{2-5}
& $F_{\footnotesize\texttt{red}}(\eps)$ & $(0,0)$ & $(1,0)$ & $(0,1)$ \\ 
\cline{1-5}
\end{tabular}}
\pass\caption{Description of the potential children of a square parent of type T1 or type T2 for the model with $p_*=0$ for the case $\la\ge4$.}
\label{tab:SquareChildModel0}
\end{table}

\newpage

\renewcommand{\arraystretch}{1.5}
\setlength{\arrayrulewidth}{0.1pt}
\begin{table}[h!]
\centering
\resizebox{13cm}{!}{%
\begin{tabular}{|c|c|c|c|c|c|}
\cline{1-6}
\multirow{4}{*}{\raisebox{2\height}{\includegraphics[scale=0.8]{SquareGreen}}} & \raisebox{0.75\height}{\begin{minipage}{3cm}$$\begin{array}{cl}
\text{T1:} & \la-4 \\ \text{T2:} & 2(\la-4)  \end{array}$$\end{minipage}} & \rule{0cm}{1.5cm}\includegraphics[scale=0.75]{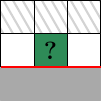} & \rule{0cm}{1.5cm}\includegraphics[scale=0.75]{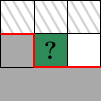} & \rule{0cm}{1.5cm}\includegraphics[scale=0.75]{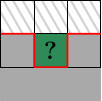} & \rule{0cm}{1.5cm}\includegraphics[scale=0.75]{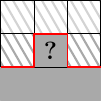} \\ \cline{2-6}
& $\eps$ & $(0,0,0)$ & $(1,0,0)$ or $(0,0,1)$ & $(1,0,1)$ & $(\eps',1,\eps'')$\\ \cline{2-6}
& $\P(B_p^{(3)}=\eps)$ & $(1-p)^3$ & $2p(1-p)^2$ & $p^2(1-p)$ & $p$ \\ \cline{2-6}
& $F_{\footnotesize\texttt{green}}(\eps)$ & $(1,0,0)$ & $(0,1,0)$ & $(-1,2,0)$ & $(0,0,0)$ \\ 
\cline{1-6}
\multirow{4}{*}{\raisebox{2\height}{\includegraphics[scale=0.8]{SquarePurple}}} & \raisebox{0.75\height}{\begin{minipage}{3cm}$$\begin{array}{cl}
\text{T1:} & 2  \\ \text{T2:} & 2 \end{array}$$\end{minipage}} & \rule{0cm}{1.5cm}\includegraphics[scale=0.75]{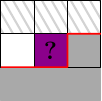} & \rule{0cm}{1.5cm}\includegraphics[scale=0.75]{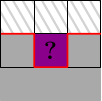} & \rule{0cm}{1.5cm}\includegraphics[scale=0.75]{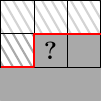} \\ \cline{2-5}
& $\eps$ & $(0,0)$ & $(1,0)$ & $(\eps',1)$ \\ \cline{2-5}
& $\P(B_p^{(2)}=\eps)$ & $(1-p)^2$ & $p(1-p)$ & $p$ \\ \cline{2-5}
& $F_{\footnotesize\texttt{purple}}(\eps)$ & $(0,1,0)$ & $(-1,2,0)$ & $(0,0,0)$ \\ 
\cline{1-6}
\multirow{4}{*}{\raisebox{2\height}{\includegraphics[scale=0.8]{SquareBlue}}} & \raisebox{0.75\height}{\begin{minipage}{3cm}$$\begin{array}{cl}
\text{T1:} & \la-4  \\ \text{T2:} & 2(\la-4) \end{array}$$\end{minipage}} & \rule{0cm}{1.5cm}\includegraphics[scale=0.75]{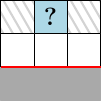} & \rule{0cm}{1.5cm}\includegraphics[scale=0.75]{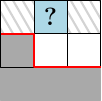} & \rule{0cm}{1.5cm}\includegraphics[scale=0.75]{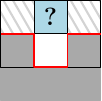} & \rule{0cm}{1.5cm}\includegraphics[scale=0.75]{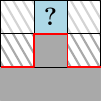} \\ \cline{2-6}
& $\eps$ & $(0,0,0)$ & $(1,0,0)$ or $(0,0,1)$ & $(1,0,1)$ & $(\eps',1,\eps'')$\\ \cline{2-6}
& $\P(B_p^{(3)}=\eps)$ & $(1-p)^3$ & $2p(1-p)^2$ & $p^2(1-p)$ & $p$ \\ \cline{2-6}
& $F_{\footnotesize\texttt{blue}}(\eps)$ & $(0,0,0)$ & $(0,0,1)$ & $(0,0,2)$ & $(1,0,0)$ \\ 
\cline{1-6}
\multirow{4}{*}{\raisebox{2\height}{\includegraphics[scale=0.8]{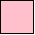}}} & \raisebox{0.75\height}{\begin{minipage}{3cm}$$\begin{array}{cl}
\text{T1:} & 2  \\ \text{T2:} & 2 \end{array}$$\end{minipage}} & \rule{0cm}{1.5cm}\includegraphics[scale=0.75]{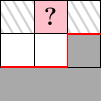} & \rule{0cm}{1.5cm}\includegraphics[scale=0.75]{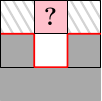} & \rule{0cm}{1.5cm}\includegraphics[scale=0.75]{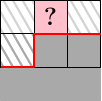} \\ \cline{2-5}
& $\eps$ & $(0,0)$ & $(1,0)$ & $(\eps',1)$ \\ \cline{2-5}
& $\P(B_p^{(2)}=\eps)$ & $(1-p)^2$ & $p(1-p)$ & $p$ \\ \cline{2-5}
& $F_{\footnotesize\texttt{pink}}(\eps)$ & $(0,0,1)$ & $(0,0,2)$ & $(1,0,0)$ \\ 
\cline{1-5}
\multirow{4}{*}{\raisebox{2\height}{\includegraphics[scale=0.8]{SquareYellow}}} & \raisebox{0.75\height}{\begin{minipage}{3cm}$$\begin{array}{cl}
\text{T1:} & 2  \\ \text{T2:} & 2 \end{array}$$\end{minipage}} & \rule{0cm}{1.5cm}\includegraphics[scale=0.75]{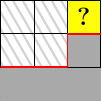} \\ \cline{2-3}
& $\eps$ & $1$ \\ \cline{2-3}
& $\P(B_p^{(1)}=\eps)$ & $1$ \\ \cline{2-3}
& $F_{\footnotesize\texttt{yellow}}(\eps)$ & $(1,0,0)$ \\ 
\cline{1-3}
\end{tabular}}
\pass\caption{Description of the potential children of a square parent of type T1 for the model with $p_*=1$ for the case $\la\ge4$.}
\label{tab:SquareChildModel1T1}
\end{table}

\newpage

\renewcommand{\arraystretch}{1.5}
\setlength{\arrayrulewidth}{0.1pt}
\begin{table}[h!]
\centering
\resizebox{18cm}{!}{%
\hspace{-2cm}\begin{tabular}{|c|c|c|c|c|c|c|c|c|}
\cline{1-6}
\multirow{4}{*}{\raisebox{3\height}{\includegraphics[scale=0.8]{SquareBrown}}} & \raisebox{1\height}{\begin{minipage}{3cm}$$\begin{array}{cl}
\text{T1:} & 0 \\ \text{T2:} & 1  \end{array}$$\end{minipage}} & \rule{0cm}{1.9cm}\includegraphics[scale=0.75]{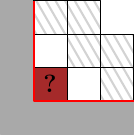} & \rule{0cm}{1.9cm}\includegraphics[scale=0.75]{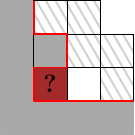} & \rule{0cm}{1.9cm}\includegraphics[scale=0.75]{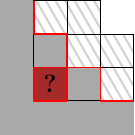} & \rule{0cm}{1.9cm}\includegraphics[scale=0.75]{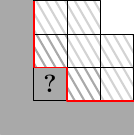} \\ \cline{2-6}
& $\eps$ & $(0,0,0)$ & $(1,0,0)$ or $(0,0,1)$ & $(1,0,1)$ & $(\eps',1,\eps'')$\\ \cline{2-6}
& $\P(B_p^{(3)}=\eps)$ & $(1-p)^3$ & $2p(1-p)^2$ & $p^2(1-p)$ & $p$ \\ \cline{2-6}
& $F_{\footnotesize\texttt{brown}}(\eps)$ & $(0,1,0)$ & $(-1,2,0)$ & $(-4,4,0)$ & $(0,0,0)$ \\ 
\cline{1-9}
\multirow{4}{*}{\raisebox{3\height}{\includegraphics[scale=0.8]{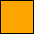}}} & \raisebox{1\height}{\begin{minipage}{3cm}$$\begin{array}{cl}
\text{T1:} & 0 \\ \text{T2:} & 2  \end{array}$$\end{minipage}} & \rule{0cm}{1.9cm}\includegraphics[scale=0.75]{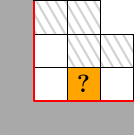} & \rule{0cm}{1.9cm}\includegraphics[scale=0.75]{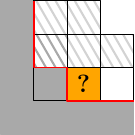} & \rule{0cm}{1.9cm}\includegraphics[scale=0.75]{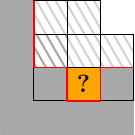} & \rule{0cm}{1.9cm}\includegraphics[scale=0.75]{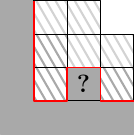} & \rule{0cm}{1.9cm}\includegraphics[scale=0.75]{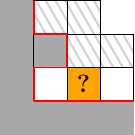} &
\rule{0cm}{1.9cm}\includegraphics[scale=0.75]{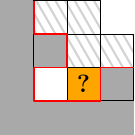} &
\rule{0cm}{1.9cm}\includegraphics[scale=0.75]{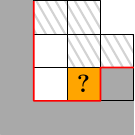}  \\ \cline{2-9}
& $\eps$ & $(0,0,0,0)$ & $(\eps',1,0,0)$ & $(\eps',1,0,1)$ & $(\eps',\eps'',1,\eps''')$ & $(1,0,0,0)$ & $(1,0,0,1)$ & $(0,0,0,1)$ \\ \cline{2-9}
& $\P(B_p^{(4)}=\eps)$ & $(1-p)^4$ & $p(1-p)^2$ & $p^2(1-p)$ & $p$ & $p(1-p)^3$ & $p^2(1-p)^2$ & $p(1-p)^3$ \\ \cline{2-9}
& $F_{\footnotesize\texttt{orange}}(\eps)$ & $(1,0,0)$ & $(0,1,0)$ & $(-1,2,0)$ & $(0,0,0)$ & $(1,0,1)$ & $(0,1,1)$ & $(0,1,0)$ \\ 
\cline{1-9}
\multirow{4}{*}{\raisebox{-.25\height}{\includegraphics[scale=0.8]{SquareRed}}} & \raisebox{1\height}{\begin{minipage}{3cm}$$\begin{array}{cl}
\text{T1:} & 0  \\ \text{T2:} & 1 \end{array}$$\end{minipage}} & \rule{0cm}{1.9cm}\includegraphics[scale=0.75]{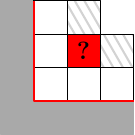} & \rule{0cm}{1.9cm}\includegraphics[scale=0.75]{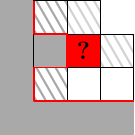} & \rule{0cm}{1.9cm}\includegraphics[scale=0.75]{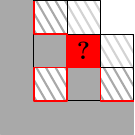} & \rule{0cm}{1.9cm}\includegraphics[scale=0.75]{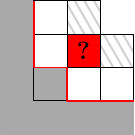} & \rule{0cm}{1.9cm}\includegraphics[scale=0.75]{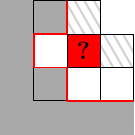} &
\rule{0cm}{1.9cm}\includegraphics[scale=0.75]{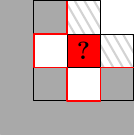} & \rule{0cm}{1.9cm}\includegraphics[scale=0.75]{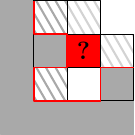} 
\\ \cline{2-9}
& $\eps$ & $(0,0,0,0,0)$ & \begin{minipage}{3cm}$$\begin{array}{l}
\text{\phantom{or} } (\eps',1,\eps'',0,0) \\ \text{or } (0,0,\eps',1,\eps'') \end{array}$$\end{minipage} & $(\eps',1,\eps'',1,\eps''')$ & \begin{minipage}{3cm}$$\begin{array}{l}
\text{\phantom{or} } (1,0,0,0,0) \\ \text{or } (0,0,1,0,0) \\ \text{or } (0,0,0,0,1) \end{array}$$\end{minipage} & \begin{minipage}{3cm}$$\begin{array}{l}
\text{\phantom{or} } (1,0,1,0,0) \\ \text{or } (0,0,1,0,1) \\ \text{or } (1,0,0,0,1) \end{array}$$\end{minipage} & $(1,0,1,0,1)$ & \begin{minipage}{3cm}$$\begin{array}{l}
\text{\phantom{or} } (\eps',1,\eps'',0,1) \\ \text{or } (1,0,\eps',1,\eps'') \end{array}$$\end{minipage} \\ \cline{2-9}
& $\P(B_p^{(5)}=\eps)$ & $(1-p)^5$ & $2p(1-p)^2$ & $p^2$ & $3p(1-p)^4$ & $3p^2(1-p)^3$ & $ p^3(1-p)^2$ & $2p^2(1-p)$ \\ \cline{2-9}
& $F_{\footnotesize\texttt{red}}(\eps)$ & $(0,0,0)$ & $(1,0,0)$ & $(0,1,0)$ & $(0,0,1)$ & $(0,0,2)$ & $(0,0,3)$ & $(1,0,1)$ \\ 
\cline{1-9}
\end{tabular}}
\pass\caption{Description of the specific potential children of a square parent of type T2 for the model with $p_*=1$ for the case $\la\geq4$.}
\label{tab:SquareChildModel1T2}
\end{table}

\subsection{The triangular tessellation}\label{sec:annexetriangle}

\noi

In the case of the triangular tessellation, only the types T1, T2 and T3 are relevant. 
Type T3 is a phantom type which may be reduced using the rule T3$=$3T2$-$3T1.

\renewcommand{\arraystretch}{0.8}
\begin{figure}[h!]
\begin{center}
\begin{tabular}{ccc}
\includegraphics[scale=1.25]{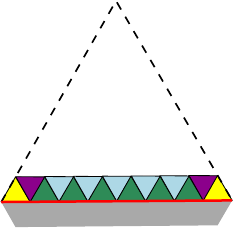} & $\quad$ &
\includegraphics[scale=1.25]{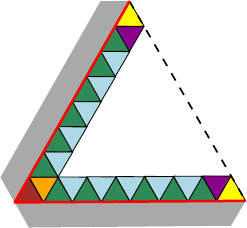} \\
\footnotesize{(a)} & & \footnotesize{(b)}
\end{tabular}
\end{center}
\vspace{-.35cm}
\caption{The potential children of a parent of type T1 (a) and a parent of type T2 (b) for the case of the triangular tessellation (here $\la=8$). A potential children only should be a colored triangle belonging to the first lines neighboring $\partial\cK_n$.}
\label{fig:Model0TriangleAllTheChild}
\end{figure} 
\renewcommand{\arraystretch}{1.5}

\renewcommand{\arraystretch}{1.5}
\setlength{\arrayrulewidth}{0.1pt}
\begin{table}[h!]
\centering
\resizebox{12cm}{!}{%
\begin{tabular}{|c|c|c|c|c|c|}
\cline{1-4}
\multirow{4}{*}{\raisebox{1.5\height}{\includegraphics[scale=0.8]{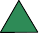}}} & \raisebox{0.5\height}{\begin{minipage}{3cm}$$\begin{array}{cl}
\text{T1:} & \la-2 \\ \text{T2:} & 2(\la-2) \end{array}$$\end{minipage}} & \rule{0cm}{1.25cm}\includegraphics[scale=1.25]{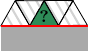} & \rule{0cm}{1.25cm}\includegraphics[scale=1.25]{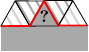} \\ 
\cline{2-4}
& $\eps$ & $0$ & $1$ \\ \cline{2-4}
& $\P(B^{(1)}=\eps)$ & $1-p$ & $p$ \\ 
\cline{2-4}
& $F_{\footnotesize\texttt{green}}(\eps)$ & $(1,0)$ & $(0,0)$ \\ 
\cline{1-5}
\multirow{4}{*}{\raisebox{1.5\height}{\includegraphics[scale=0.8]{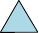}}} & \raisebox{0.5\height}{\begin{minipage}{3cm}$$\begin{array}{cl}
\text{T1:} & \la-3 \\ \text{T2:} & 2(\la-3)  \end{array}$$\end{minipage}} & \rule{0cm}{1.25cm}\includegraphics[scale=1.25]{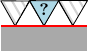} & \rule{0cm}{1.25cm}\includegraphics[scale=1.25]{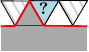} & \rule{0cm}{1.25cm}\includegraphics[scale=1.25]{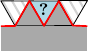} \\ \cline{2-5}
& $\eps$ & $(0,0)$ & $(1,0)$ or $(0,1)$ & $(1,1)$ \\ \cline{2-5}
& $\P(B^{(2)}=\eps)$ & $(1-p)^2$ & $2p(1-p)$ & $p^2$ \\ \cline{2-5}
& $F_{\footnotesize\texttt{blue}}(\eps)$ & $(0,0)$ & $(1,0)$ & $(0,1)$ \\ 
\cline{1-5}
\multirow{4}{*}{\raisebox{1.5\height}{\includegraphics[scale=0.8]{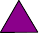}}} & \raisebox{0.5\height}{\begin{minipage}{3cm}$$\begin{array}{cl}
\text{T1:} & 2  \\ \text{T2:} & 2 \end{array}$$\end{minipage}} & \rule{0cm}{1.25cm}\includegraphics[scale=1.25]{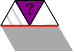} & \rule{0cm}{1.25cm}\includegraphics[scale=1.25]{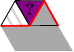} \\ \cline{2-4}
& $\eps$ & $0$ & $1$ \\ \cline{2-4}
& $\P(B^{(1)}=\eps)$ & $1-p$ & $p$ \\ \cline{2-4}
& $F_{\footnotesize\texttt{purple}}(\eps)$ & $(0,0)$ & $(1,0)$ \\ 
\cline{1-4}
\multirow{4}{*}{\raisebox{1.5\height}{\includegraphics[scale=0.8]{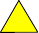}}} & \raisebox{0.5\height}{\begin{minipage}{3cm}$$\begin{array}{cl}
\text{T1:} & 2  \\ \text{T2:} & 2 \end{array}$$\end{minipage}} & \rule{0cm}{1.25cm}\includegraphics[scale=1.25]{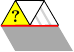} \\ \cline{2-3}
& $\eps$ & $0$ \\ \cline{2-3}
& $\P(B^{(1)}=\eps)$ & $1$ \\ \cline{2-3}
& $F_{\footnotesize\texttt{yellow}}(\eps)$ & $(0,0)$ \\ 
\cline{1-4}
\multirow{4}{*}{\raisebox{3.25\height}{\includegraphics[scale=0.8]{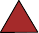}}} & \raisebox{\height}{\begin{minipage}{3cm}$$\begin{array}{cl}
\text{T1:} & 0 \\ \text{T2:} & 1 \end{array}$$\end{minipage}} & \rule{0cm}{1.9cm}\includegraphics[scale=1]{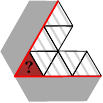} & \rule{0cm}{1.9cm}\includegraphics[scale=1]{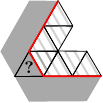} \\ \cline{2-4}
& $\eps$ & $0$ & $1$  \\ \cline{2-4}
& $\P(B^{(1)}=\eps)$ & $1-p$ & $p$ \\ \cline{2-4}
& $F_{\footnotesize\texttt{brown}}(\eps)$ & $(0,1)$ & $(0,0)$  \\ 
\cline{1-6}
\multirow{4}{*}{\raisebox{-0.25\height}{\includegraphics[scale=0.8]{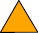}}} & \raisebox{\height}{\begin{minipage}{3cm}$$\begin{array}{cl}
\text{T1:} & 0 \\ \text{T2:} & 1 \end{array}$$\end{minipage}} & \rule{0cm}{1.9cm}\includegraphics[scale=1]{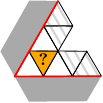} & \rule{0cm}{1.9cm}\includegraphics[scale=1]{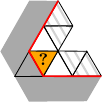} & \rule{0cm}{1.9cm}\includegraphics[scale=1]{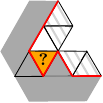} & \rule{0cm}{1.9cm}\includegraphics[scale=1]{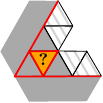} \\ \cline{2-6}
& $\eps$ & $(0,0,0)$ & 
\begin{minipage}{3cm}$$\begin{array}{l}
\text{\phantom{or }} (1,0,0) \\ 
\text{or } (0,1,0) \\ 
\text{or } (0,0,1) \\
\end{array}$$\end{minipage}
& \begin{minipage}{3cm}$$\begin{array}{l}
\text{\phantom{or }} (1,1,0) \\ 
\text{or } (1,0,1) \\ 
\text{or } (0,1,1) \\
\end{array}$$\end{minipage}
& $(1,1,1)$ \\ \cline{2-6}
& $\P(B^{(3)}=\eps)$ & $(1-p)^3$ & $3p(1-p)^2$ & $3(1-p)p^2$ & $p^3$ \\ \cline{2-6}
& $F_{\footnotesize\texttt{orange}}(\eps)$ & $(0,0)$ & $(1,0)$ & $(0,1)$ & $(-3,3)$ \\ 
\cline{1-6}
\end{tabular}}
\pass\caption{Description of the potential children of a triangle parent of type T1 or type T2 for the model with $p_*=0$ for the case $\la\ge4$.}
\label{tab:TriangleChildModel0}
\end{table}

\newpage

\addtocontents{toc}{\vspace{0.2cm}}%
\addcontentsline{toc}{section}{\protect\numberline{}\hspace{-18pt}\bf References}

\let\oldaddcontentsline\addcontentsline
\renewcommand{\addcontentsline}[3]{}
\bibliographystyle{plain}

\end{document}